\documentclass[12pt]{amsart}
\usepackage{amsmath,amssymb,latexsym,cancel,rotating,amsthm,mathrsfs,fancyhdr,graphicx}
\usepackage[all]{xy}

\newtheorem{theorem}{Theorem}[section]
\newtheorem{lemma}[theorem]{Lemma}
\newtheorem{corollary}[theorem]{Corollary}
\newtheorem{definition}[theorem]{Definition}
\newtheorem{proposition}[theorem]{Proposition}

\newtheorem{remark}[theorem]{Remark}

\newtheorem{example}[theorem]{Example}

\def\cP{{\mathcal P}} 
\def\cC{{\mathcal C}} \def\cD{{\mathcal D}}    \def\cL{{\mathcal L}}        \def\rF{{\textrm F}}    \def\rG{{\rm G}} \def\rE{{\rm E}} \def\rC{{\rm C}}  \def\rQ{{\textrm Q}}  \def\rU{{\rm U}} 
\def\cH{{\mathcal H}}   \def\cB{{\mathcal B}} \def\cS{{\mathcal S}} 
\def\cZ{{\mathcal Z}} \def\cO{{\mathcal O}} 
\def\cI{{\mathcal I}}       \def\cT{{\mathcal T}}
\def\cA{{\mathcal A}}    
\def\cK{{\mathcal K}} 
\def\bbN{{\mathbb N}}  \def\bbZ{{\mathbb Z}}  \def\bbQ{{\mathbb Q}}
    \def\bbF{{\mathbb F}}
\def\bbC{{\mathbb C}}    
  \def\bbD{{\mathbb D}}

         \def\bfU{{\bf U}}

\def\Hom{\mbox{\rm Hom}}  
       \def\im{\mbox{\rm Im}\,}
\def\Ext{{\rm Ext}}   
\def\dim{{\rm dim}}   \def\End{\mbox{\rm End}}
\def\Aut{\mbox{\rm Aut}}   \def\Stab{{\rm Stab}} \def\Ind{\mbox{\rm Ind}}  \def\Res{\mbox{\rm Res}}\def\ind{\mbox{\rm ind}}\def\res{\mbox{\rm res}}  
 
  \def\tr{\mbox{\rm tr}}
  
\def\rep{\mbox{\rm rep}}

\def\IC{\mbox{\rm{IC}}}

\def\Id{\mbox{\rm Id}}

\def\Gr{\mbox{\rm Gr}}

\newcommand\ue{{\underline{e}}}

\newcommand\fg{\mathfrak g}
\newcommand\fn{\mathfrak n}

\newcommand\tw{\mathrm{tw}}
\newcommand\ex{\mathrm{ex}}
\newcommand\red{\mathrm{red}}
\newcommand\rad{\mathrm{rad}}

\begin{document}

\title[]{Sheaf realization of Bridgeland's Hall algebra of Dynkin type}

\author{Jiepeng Fang, Yixin Lan, Jie Xiao}

\address{School of Mathematical Sciences, Peking University, Beijing 100871, P.R. China.
{\it Current address}: Department of Mathematics and New Cornerstone Science Laboratory, The University of Hong Kong, Pokfulam, Hong Kong, Hong Kong SAR, P. R. China}
\email{fangjp@pku.edu.cn (J. Fang)}
\address{Academy of Mathematics and Systems Science, Chinese Academy of Sciences, Beijing 100190, P. R. China.
{\it Current address}: Max Planck Institute for mathematics,Vivatsgasse,Bonn, 53111, North Rhine-Westphalia, German}
\email{lanyixin@amss.ac.cn (Y. Lan)}
\address{School of Mathematical Sciences, Beijing Normal University, Beijing 100875, P. R. China}
\email{jxiao@bnu.edu.cn (J. Xiao)}



\bibliographystyle{abbrv}

\begin{abstract}
As one of results in \cite{Bridgeland-2013}, Bridgeland realized the quantum group $\bfU_v$ via the localization of Ringel-Hall algebra for the two-periodic projective complexes of quiver representations over a finite field. In the present paper, we generalize Lusztig's categorical construction for the nilpotent part $\bfU_v^+$ to Bridgeland's Hall algebra of Dynkin type. In particular, we obtain a basis of the Ringel-Hall algebra for the two-periodic projective complexes which has the positivity, and we categorify an integral form of the generic Bridgeland's Hall algebra which is isomorphic to the Poisson integral form of $\bfU_v$, and obtain a $\mathbb{Z}[v,v^{-1}]$-basis of this integral form. 
\end{abstract}

\maketitle

\setcounter{tocdepth}{1}\tableofcontents

\section{Introduction}

\subsection{Background}

Let $\mathbf{Q}$ be a finite acyclic quiver. Then its underlying graph determines a symmetric generalized Cartan matrix. Let $\fg$ be the corresponding Kac-Moody Lie algebra, and $\bfU_v$ be the quantized enveloping algebra (also called the quantum group) of $\fg$. Realization of $\bfU_v$ via category of quiver representations and its related categories such as the category of two-periodic projective complexes, is an important problem in representation theory, and there are developments for this problem. 

Ringel creatively introduced Hall algebra method to realize the nilpotent part $\bfU_v^+$ of $\bfU_v$ in \cite{Ringel-1990}. He defined an associative algebra $\cH(\cA_q)$ for the category $\cA_q=\rep_{\bbF_q}(\mathbf{Q})$ of finite-dimensional $\mathbf{Q}$-representations over a finite field $\bbF_q$, which is called Ringel-Hall algebra. If $\mathbf{Q}$ is a Dynkin quiver, he proved this algebra is isomorphic to $\bfU_{v=\sqrt{q}}^+$, and proved the existence of Hall polynomials which enable him to define the generic Ringel-Hall algebra to realize $\bfU_v^+$. 

Lusztig used perverse sheaves on the varieties of $\mathbf{Q}$-representations over $\overline{\bbF}_q$ to categorify $\bfU_v^+$ in \cite{Lusztig-1991}. He defined the induction and restriction functors which induce a bialgebra structure on the Grothendieck group $\cK^+$ arising from certain perverse sheaves, and proved $\cK^+$ is isomorphic to the integral form $_{\cZ}\bfU_v^+$ of $\bfU_v^+$, where $\cZ=\bbZ[v,v^{-1}]$. The categorification provides $_{\cZ}\bfU_v^+$ with a $\cZ$-basis given by images of simple perverse sheaves, which is called the canonical basis. This basis has very remarkable properties and plays an important role in Lie theory.

The relation between Ringel's realization and Lusztig's categorification is given by the sheaf-function correspondence, see \cite{Lusztig-1998} and \cite{Xiao-Xu-Zhao-2019}.

Bridgeland used the Hall algebra $\cH^{\tw}{\cC_2(\cP_{q})}$ for the category $\cC_2(\cP_{q})$ of two-periodic projective complexes of $\cA_q$ to realize $\bfU_{v=\sqrt{q}}$ globally in \cite{Bridgeland-2013}. He found that for any projective representation $P\in \cA_q$, the elements $b_{K_P},b_{K_P}^*$ corresponding to the contractible complexes $K_P,K_P^*$ satisfy the Ore conditions in the Hall algebra for $\cC_2(\cP_{\bbF_q})$, and defined his Hall algebra $\cD\cH^{\red}(\cA_q)$ to be a reduced quotient of the localization algebra with respect to these elements. If $\mathbf{Q}$ is a Dynkin quiver, he proved his Hall algebra is isomorphic to $\bfU_{v=\sqrt{q}}$.

It is natural to consider whether it is possible to categorify Bridgeland's Hall algebra by perverse sheaves in Lusztig's style. For example, Bridgeland asked this question in \cite{Bridgeland-2013}. Qin made a progress in the case that $\mathbf{Q}$ is a Dynkin quiver in \cite{Qin-2016}. He used perverse sheaves on cyclic quiver varieties to categorify $\bfU_v$, and obtained a basis which contains Lusztig's dual canonical basis of $\bfU_v^+$ up to scalars based on the work by Hernandez and Leclerc in \cite{Hernandez-Leclere-2015}. We also see that Scherotzke-Sibilla gave a geometric construction of $\bfU_v$ using Nakajima categories in \cite{Scherotzke-Sibilla-2016}, and there are other global bases for $\bfU_v$ constructed by Berenstein-Greenstein in \cite{Berenstein-Greenstein-2017} and Shen in \cite{Shen-2022}.

\subsection{Main result}

In the present paper, we give an answer for Bridgeland's question more directly. Let  $\mathbf{Q}$ be a Dynkin quiver, $k=\overline{\bbF}_q$ and $\cC_2(\cP)$ be the category of two-periodic projective complexes of $\rep_k(\mathbf{Q})$ over $k$, we define an affine $k$-variety $\rC_\ue$ with an algebraic group $\rG_{\ue}$-action to parametrize the objects in $\cC_2(\cP)$ of projective dimension vector pair $\ue$. Let $\cD^b_{\rG_{\ue},m}(\rC_{\ue})$ be the subcategory of the $\rG_{\ue}$-equivariant bounded derived category of constructible $\overline{\bbQ}_l$-sheaves on $\rC_{\ue}$ consisting of mixed Weil complexes, and $\cK_\ue$ be its Grothendieck group.

For any $\ue=\ue'+\ue''$, inspired by Lusztig's induction functor and restriction functor, we define our induction functor and restriction functor
\begin{align*}
&\Ind^{\ue}_{\ue',\ue''}:\cD^b_{\rG_{\ue'},m}(\rC_{\ue'})\boxtimes \cD^b_{\rG_{\ue''},m}(\rC_{\ue''})\rightarrow \cD^b_{\rG_\ue,m}(\rC_{\ue}),\\
&\Res^{\ue}_{\ue',\ue''}:\cD^b_{\rG_\ue,m}(\rC_{\ue})\rightarrow \cD^b_{\rG_{\ue'}\times\rG_{\ue''},m}(\rC_{\ue'}\times \rC_{\ue''}).
\end{align*}

We prove that all induction functors induce a multiplication on the direct sum $\cK=\bigoplus_{\ue}\cK_\ue$. By the sheaf-function correspondence, we show that the trace map induces an isomorphism from $\cK$ to the Hall algebra for $\cC_2(\cP)$, which can derive Bridgeland's Hall algebra after taking the localization and the reduced quotient, see Theorem \ref{induction algebra}. This is a constructible sheaves realization of Bridgeland's Hall algebra.

We prove that the restriction functor $\Res^{\ue}_{\ue',\ue''}$ can be restricted to be a functor between the subcategories consisting of mixed semisimple complexes in Proposition \ref{restriction-hyperbolic}, that is, 
$$\Res^{\ue}_{\ue',\ue''}:\cD^{b,ss}_{\rG_\ue,m}(\rC_{\ue})\rightarrow \cD^{b,ss}_{\rG_{\ue'},m}(\rC_{\ue'})\boxtimes \cD^{b,ss}_{\rG_{\ue''},m}(\rC_{\ue''}).$$
Let $\cK^{ss}$ be the Grothendieck group of $\cD^{b,ss}_{\rG_{\ue},m}(\rC_{\ue})$ which has a $\cZ$-module structure defined by $v.[L]=[L[-1](-\frac{1}{2})]$. We prove that these restriction functors induce a comultiplication on the direct sum $\cK^{ss}=\bigoplus_{\ue}\cK^{ss}_{\ue}$, and then the graded dual 
$$\cK^{ss,*}=\bigoplus_{\ue}\cK^{ss,*}_{\ue}=\bigoplus_{\ue}\Hom_{\cZ}(\cK^{ss}_{\ue},\cZ)$$ 
naturally becomes a $\cZ$-algebra. By the sheaf-function correspondence, we show that the trace map induces an isomorphism from $\cK^{ss,*}$ to the Hall algebra $\cH^{\tw}{\cC_2(\cP_{q})}$ for $\cC_2(\cP)$, see Theorem \ref{dual algebra}. Moreover, under this isomorphism, we find the elements $\tilde{I}^*_{K_P},\tilde{I}^*_{K_P^*}\in \cK^{ss,*}$ which correspond to elements $b_{K_P},b_{K_P^*}$ for the contractible complexes $K_P,K_P^*$. As Bridgeland took the localization with respect to the elements $b_{K_P},b_{K_P^*}$ and the reduced quotient, we take the localization of $\cK^{ss,*}$ with respect to the elements $\tilde{I}^*_{K_P},\tilde{I}^*_{K_P^*}$ and take the corresponding reduced quotient. Then the trace map induce an isomorphism from $\cD\cK^{ss,*,\red}$ to Bridgeland's Hall algebra. This is a perverse sheaves realization of Bridgeland's Hall algebra.

\begin{theorem}[=\bf{Theorem \ref{Bridgeland's algebra via perverse sheaf}}]
The subset 
$\{\tilde{I}_{K_P}^*, \tilde{I}_{K_P^*}^*|P\in \cP\}\subset \cK^{ss,*}$
satisfies the Ore conditions, and so there is a well-defined localization
$$\cD\cK^{ss,*}=\cK^{ss,*}[(\tilde{I}_{K_P}^*)^{-1}, (\tilde{I}_{K_P^*}^*)^{-1}|P\in \cP]$$
with a reduced quotient 
$$\cD\cK^{ss,*,\red}=\cD\cK^{ss,*}/\langle \tilde{I}_{K_P}^**_r\tilde{I}_{K_P^*}^*-1|P\in \cP\rangle$$
such that the algebra isomorphisms (\ref{three algebra isomorphism}) induce algebra isomorphisms
\begin{align*}
\cD\cH(\cA_q)\cong \cD\tilde{\cH}^*(&\cA_q)\cong \bbC\otimes_{\cZ}\cD\cK^{ss,*},\\
\cD\cH^{\red}(\cA_q)\cong \cD\tilde{\cH}^{*,\red}(&\cA_q)\cong \bbC\otimes_{\cZ}\cD\cK^{ss,*,\red}.
\end{align*}
\end{theorem}

The $\cZ$-coalgebra $\cK^{ss}_{\ue}$ has a basis $\cI_{\ue}=\{I_{M_\bullet}|\cO_{M_\bullet}\subset \rC_{\ue}\}$, where $I_{M_\bullet}$ is the image of $\IC(\cO_{M_\bullet},\overline{\bbQ}_l)(\frac{\dim \cO_{M_\bullet}}{2})$ in the Grothendieck group for the $\rG_{\ue}$-orbit $\cO_{M_\bullet}\subset \rC_{\ue}$. We prove that the $\cZ$-basis $\cI=\bigsqcup_{\ue}\cI_{\ue}$ of $\cK^{ss}$ is bar-invariant and has positivity. Dually, let $\cI^*_{\ue}=\{I^*_{M_\bullet}|\cO_{M_\bullet}\subset \rC_{\ue}\}\subset \cK^{ss,*}_{\ue}$ be the dual basis of $\cI_{\ue}$, then the $\cZ$-basis $\cI^*=\bigsqcup_{\ue}\cI^*_{\ue}$ of $\cK^{ss,*}$ is bar-invariant and has positivity, see Proposition. 

\begin{proposition}[={\bf Proposition \ref{positivity}}]
The $\cZ$-basis $\cI$ of $\cK^{ss}$ is bar-invariant and it has positivity. Dually, the $\cZ$-basis $\cI^*$ of $\cK^{ss,*}$ is bar-invariant and it has positivity. More precisely, we have $\overline{I_{M_{\bullet}}}=I_{M_{\bullet}}, \overline{I^*_{M_{\bullet}}}=I^*_{M_{\bullet}}$, and if 
\begin{align*}
r(I_{L_\bullet})=\sum_{M_\bullet, N_\bullet} \zeta^{L_\bullet}_{M_\bullet,N_\bullet}(v)I_{M_\bullet}\otimes I_{N_\bullet},\ I^*_{M_{\bullet}}*_rI^*_{N_{\bullet}}=\sum_{L_{\bullet}}\xi^{L_{\bullet}}_{M_{\bullet},N_{\bullet}}(v)I^*_{L_{\bullet}}
\end{align*}
then $\zeta^{L_{\bullet}}_{M_{\bullet},N_{\bullet}}(v), \xi^{L_{\bullet}}_{M_{\bullet}N_{\bullet}}(v)\in \bbN[v,v^{-1}]$. Moreover, we have 
$$\xi^{L_{\bullet}}_{M_\bullet N_\bullet}(v)=\zeta^{L_{\bullet}}_{M_\bullet N_\bullet}(v)=v^{-\Vert\ue_{M_\bullet},\ue_{N_\bullet}\Vert}\zeta^{L_{\bullet}}_{N_\bullet M_\bullet}(v^{-1}).$$
\end{proposition}

Let $\cT^{\red}$ be the subalgebra of $\cD\cK^{ss,*,\red}$ generated by $\tilde{I}^*_{K_P},\tilde{I}^*_{K_P^*}$ for any $K_P,K_P^*$, then it is isomorphic to the group algebra $\cZ[K(\cA)]$ of the Grothendieck group $K(\cA)$ of $\rep_{\overline{\bbF}_q}(\mathbf{Q})$ over $\cZ$, which has a $\cZ$-basis $\tilde{\cI}^{*,\red}=\{\tilde{I}^*_\alpha|\alpha\in \cK(\cA)\}$. By an induction depending on the degeneration order of orbits, we prove that the multiplication $*_r$ defines a free $\cT^{\red}$-module structure on $\cD\cK^{ss,*,\red}$ with a basis $\cI^{*,\rad}=\{I^*_{M_\bullet}|M_\bullet\ \textrm{is radical}\}$.  Thus $\cD\cK^{ss,*,\red}$ has a $\cZ$-basis $\tilde{\cI}^{*,\red}*_r\cI^{*,\rad}$.

\begin{theorem}[={\bf Theorem \ref{global basis}}]
(a) The algebra $\cD\cK^{ss,*}$ has a $\cZ$-basis $\tilde{\cI}^**_r\cI^{*,\rad}$.\\
(b) The algebra $\cD\cK^{ss,*,\red}$ has a $\cZ$-basis $\tilde{\cI}^{*,\red}*_r\cI^{*,\rad}$.
\end{theorem}

We compare our basis with Lusztig's canonical basis for $\bfU_v^+$. The IC sheaves behave significantly differently in two constructions. For example, in Lusztig's construction, the Chevalley generator $E_i$ corresponds to the sheaf $\overline{\bbQ}_l|_{\cO_{S_i}}=\IC(\cO_{S_i},\overline{\bbQ}_l)$ which gives a basis element. While in our construction, inspired by Bridgeland's result, we should consider the sheaf $\overline{\bbQ}_l|_{\cO_{C_{S_i}}}$, where $C_{S_i}$ is the minimal projective resolution of $S_i$. It is clear that $\overline{\bbQ}_l|_{\cO_{C_{S_i}}}\not=\IC(\cO_{C_{S_i}},\overline{\bbQ}_l)$ and the former does not give a basis element. However, we observe that it is better to consider the dual of basis elements (under $\Hom_{\cZ}(-,\cZ)$), and obtain a relation between Lusztig's canonical basis and the basis $\tilde{\cI}^{*,\red}*_r\cI^{*,\rad}$ for maximal orbits, see Proposition \ref{rigid lemma}.

We prove that the $\cZ$-algebra $\cD\cK^{ss,*,\red}$ is isomorphic to the extension counting integral form of the generic Bridgeland's Hall algebra, see Theorem \ref{two integral forms}. This integral form is isomorphic to the Poisson integral form (different from Lusztig's integral form) of the quantized enveloping algebra defined in \S \ref{Quantized enveloping algebra}. In particular, under these isomorphisms, the $\cZ$-basis $\tilde{\cI}^{*,\red}*_r\cI^{*,\rad}$ of $\cD\cK^{ss,*,\red}$ gives us a $\cZ$-basis of the Poisson integral form of the quantized enveloping algebra.

\subsection{Structure of the paper}

In \S \ref{Induction functor and restriction functor}, we define our induction and restriction functors to set our stage. These definitions are inspired by Lusztig's induction and restriction functors in \S \ref{Lusztig's categorification}. In \S \ref{Bridgeland's Hall algebra via functions}, we give a geometric construction of  Bridgeland's Hall algebra in terms of functions. In \S \ref{Constructible sheaves realization of Bridgeland's Hall algebra}, we review the sheaf-function correspondence, and use the induction functors to give the constructible sheaves realization of Bridgeland's Hall algebra. In \S \ref{Perverse sheaves realization of Bridgeland's Hall algebra}, we use the restriction functors and subcategories of semisimple complexes to give the perverse sheaves realization of Bridgeland's Hall algebra. In \S \ref{Basis}, we construct the bases and study their properties. In \S \ref{compare with Lusztig}, we compare the integral forms and the bases.

\subsection*{Acknowledgements}

J. Fang is partially supported by National Key R\&D Program of China (Grant No. 2020YFE0204200). Y. Lan is partially supported by National Natural Science Foundation of China (Grant No. 12288201). J. Xiao is partially supported by National Natural Science Foundation of China (Grant No. 12031007).

\section{Preliminaries}

Throughout this paper, we fix a  Dynkin quiver $\mathbf{Q}=(I,H,s,t)$ whose underlying graph is a Dynkin diagram of the form $A_n,D_n,E_6,E_7$ or $E_8$. Let $(a_{ij})_{i,j\in I}$ be the corresponding Cartan matrix,  and $\fg$ be the corresponding complex semisimple Lie algebra.

\subsection{Quantized enveloping algebra}\label{Quantized enveloping algebra}

We refer to \cite[Chapter 3]{Lusztig-1993} for details.

Let $\bbC(v)$ be the field of rational functions in $v$ with coefficients in $\bbC$, and let $\cZ=\bbZ[v,v^{-1}]$ be the subring of Laurent polynomials with coefficients in $\bbZ$. Let
$$[n]=\frac{v^n-v^{-n}}{v-v^{-1}},\ [n]!=\prod^n_{l=1}[l]\ \textrm{for any}\ n\in \bbN.$$

The quantized enveloping algebra $\bfU_v$ of the Lie algebra $\fg$ is the $\bbC(v)$-algebra generated by $E_i,F_i,K_i,K_i^{-1}$ for any $i\in I$ subject to the following relations
\begin{align*}
&K_iK_i^{-1}=K_i^{-1}K_i=1,K_iK_j=K_jK_i\ \textrm{for any}\ i,j\in I;\\
&K_iE_j=v^{a_{ij}}E_jK_i,\ K_iF_j=v^{-a_{ij}}F_jK_i\ \textrm{for any}\ i,j\in I;\\
&E_iF_j-F_jE_i=\delta_{i,j}\frac{K_i-K_i^{-1}}{v-v^{-1}}\ \textrm{for any}\ i,j\in I;\\
&\sum_{s=0}^{1-a_{ij}}(-1)^sE_j^{(s)}E_iE_j^{(1-a_{ij}-s)}=0 \ \textrm{for any}\  i\not=j\ \textrm{in}\ I;\\
&\sum_{s=0}^{1-a_{ij}}(-1)^sF_j^{(s)}F_iF_j^{(1-a_{ij}-s)}=0 \ \textrm{for any}\  i\not=j\ \textrm{in}\ I,
\end{align*}
where $E_i^{(n)}=E_i^n/[n]!, F_i^{(n)}=F_i^n/[n]!$ for any $i\in I$ and $n\in \bbN$. We denote the subalgebras $\bfU_v^+=\langle E_i\mid i\in I\rangle,\ \bfU_v^0=\langle K_i,K_i^{-1}\mid i\in I\rangle,\ \bfU_v^-=\langle F_i\mid i\in I\rangle$.

The Lusztig integral form $_{\cZ}\bfU_v$ of $\bfU_v$ is the $\cZ$-subalgebra of $\bfU_v$ generated by $E_i^{(n)}, F_i^{(n)}, K_i,K_i^{-1}$ for any $i\in I$ and $n\in \bbN$. The Lusztig integral form $_{\cZ}\bfU_v^+$ of $\bfU_v^+$ is the $\cZ$-subalgebra of $\bfU_v^+$ generated by $E_i^{(n)}$ for any $i\in I$ and $n\in \bbN$.

We refer to \cite[Definition 5.5]{Murphy-2018} for the definition of the Poisson integral form. Let $\alpha_1,\ldots,\alpha_n$ be all simple roots, and $w_0$ be the longest element in the Weyl group. We take a reduced expression $w_0=s_{i_1}\cdots s_{i_N}$ of $w_0$ into simple reflections, then $\beta_l=s_{i_1}\cdots s_{i_{l-1}}(\alpha_{i_l}), l=1,\ldots,N$ are all positive roots. For any $i\in I$, let $T_i:\bfU_v\rightarrow \bfU_v$ be the algebra isomorphism defined in \cite[Section 1.3]{Lusztig-1990.3}. For $l=1,\ldots,N$, let $E_{\beta_l}=T_{i_1}^{-1}\ldots T_{i_{l-1}}^{-1}(E_{i_l}), F_{\beta_l}=T_{i_1}^{-1}\ldots T_{i_{l-1}}^{-1}(F_{i_l})$. The Poisson integral form $_{\cZ}\bfU_v^P$ of $\bfU_v$ is the $\cZ$-subalgebra $\bfU_v$ generated by $(v^2-1)E_{\beta_l},(v^2-1)F_{\beta_l},K_i,K_i^{-1}$ for any $i\in I$ and $l=1,\ldots,N$. The Poisson integral form $_{\cZ}\bfU_v^{+,P}$ of $\bfU_v^+$ is the $\cZ$-subalgebra of $\bfU_v^+$ generated by $(v^2-1)E_{\beta_l}$ for any $l=1,\ldots,N$.

\subsection{Two-periodic projective complex}\label{Category of two-periodic projective complexes}

We refer to \cite[\S 2]{Fang-Lan-Xiao-2024} for details.

Let $K$ be a field, $\cA$ be a $K$-linear abelian category, and $\cP$ be the full subcategory of $\cA$ consisting of projective objects. Let $\cC_2(\cA)$ be the category of two-periodic complexes of $\cA$. Its objects are of the form
$$\xymatrix@C=1cm{{M_\bullet=(M^1,M^0,d^1,d^0)=M^1}  \ar@<.5ex>[r]^-{d^1} &M^0, \ar@<.5ex>[l]^-{d^0}}$$
where $M^j\in \cA$ and $d^j\in \Hom_{\cA}(M^j,M^{j+1})$ such that $d^{j+1}d^j=0$ for any $j\in \bbZ_2$. A morphism $f:M_\bullet\rightarrow N_\bullet$ is given by a pair $f=(f^1,f^0)$, where $f^j\in \Hom_{\cA}(M^j,N^j)$ such that $d^j_Nf^j=f^{j+1}d^j_M$  for any $j\in \bbZ_2$, that is, the following diagram commutes
$$\xymatrix@C=1.5cm{M^1 \ar@<.5ex>[r]^-{d_M^1} \ar[d]_-{f^1} &M^0 \ar@<.5ex>[l]^-{d_M^0} \ar[d]^-{f^0}\\ N^1 \ar@<.5ex>[r]^-{d_N^1} &N^0. \ar@<.5ex>[l]^-{d_N^0}}$$

An object $M_{\bullet}\in \cC_2(\cA)$ is said to be a two-periodic projective complex, if $M^j\in \cP$ for any $j\in \bbZ_2$. Let $\cC_2(\cP)$ be the full subcategory of $\cC_2(\cA)$ consisting of two-periodic projective complexes. Note that $\cC_2(\cA)$ is an abelian category and $\cC_2(\cP)$ is an extension closed subcategory, thus $\cC_2(\cP)$ is an exact category in the sense of \cite{Quillen-1973}. For any objects $M_\bullet,N_\bullet\in \cC_2(\cP)$, we can identify $\Ext^1_{\cC_2(\cA)}(M_\bullet,N_\bullet)=\Ext^1_{\cC_2(\cP)}(M_\bullet,N_\bullet)$, where the left hand side is the extension group in an abelian category, and the right hand side is the extension group in an exact category.

Two morphisms $f,g\in \Hom_{\cC_2(\cP)}(M_\bullet, N_\bullet)$ are said to be homotopic, denoted by $f\sim g$, if there exists $s^j\in \Hom_{\cA}(M^j,N^{j+1})$ such that $f^j-g^j=d^{j+1}_Ns^j+s^{j+1}d^j_M$ for any $j\in \bbZ_2$. An object $M_\bullet\in \cC_2(\cP)$ is said to be a contractible complex, if $1_{M_\bullet}\sim 0$. For any $P\in \cP$, there are two contractible complexes
$$\xymatrix@C=1cm{K_P=P \ar@<.5ex>[r]^-{1} &P \ar@<.5ex>[l]^-{0},\  K_P^*=P \ar@<.5ex>[r]^-{0} &P. \ar@<.5ex>[l]^-{1}}$$

Let $\cK_2(\cP)$ be the homotopy category of two-periodic projective complexes, whose objects are the same as $\cC_2(\cP)$ and 
$$\Hom_{\cK_2(\cP)}(M_\bullet, N_\bullet)=\Hom_{\cC_2(\cP)}(M_\bullet, N_\bullet)/\textrm{Htp}(M_\bullet, N_\bullet),$$
where $\textrm{Htp}(M_\bullet, N_\bullet)=\{f\in\Hom_{\cC_2(\cP)}(M_\bullet, N_\bullet)\mid f\sim 0\}$. By \cite[Proposition 7.1]{Peng-Xiao-1997}, $\cC_2(\cP)$ is a Frobenius category whose projective-injective objects are contractible complexes, and its stable category is $\cK_2(\cP)$. By Happel's theorem \cite{Happel-1988}, $\cK_2(\cP)$ is a triangulated category.

The shift functors on $\cC_2(\cP),\cK_2(\cP)$ are involutions, denoted by
$$\xymatrix@C=1cm{-^*:M_\bullet=(M^1 \ar@<.5ex>[r]^-{d^1} &M^0 \ar@<.5ex>[l]^-{d^0}) \ar@{|->}[r] &M_\bullet^*=(M^0 \ar@<.5ex>[r]^-{-d^0} &M^1 \ar@<.5ex>[l]^-{-d^1})}.$$

The following lemma is a refinement of  \cite[Lemma 3.3]{Bridgeland-2013}.

\begin{lemma}[{\cite[Lemma 6.4]{Fang-Lan-Xiao-2024}}]\label{bijection between Ext and Hom}
For any $M_\bullet,N_\bullet,L_\bullet\in \cC_2(\cP)$, there is a bijection 
$$\Ext^1_{\cC_2(\cP)}(M_\bullet,N_\bullet)\cong \Hom_{\cK_2(\cP)}(M_\bullet,N_\bullet^*).$$
Moreover, if $\Ext^1_{\cC_2(\cP)}(M_\bullet,N_\bullet)_{L_\bullet} \neq\varnothing$, then above bijection restricts to
$$\Ext^1_{\cC_2(\cP)}(M_\bullet,N_\bullet)_{L_\bullet}\cong \Hom_{\cK_2(\cP)}(M_\bullet,N_\bullet^*)_{L_\bullet^*},$$
where the subset $\Ext^1_{\cC_2(\cP)}(M_\bullet,N_\bullet)_{L_\bullet}\subset\Ext^1_{\cC_2(\cP)}(M_\bullet,N_\bullet)$ consists of the equivalence classes of short exact sequences whose middle terms are isomorphic to $L_{\bullet}$ in $\cC_2(\cP)$, and the subset $\Hom_{\cK_2(\cP)}(M_\bullet,N_\bullet^*)_{L_\bullet^*}\subset \Hom_{\cK_2(\cP)}(M_\bullet,N_\bullet^*)$ consists of the morphisms whose mapping cones are isomorphic to $L_\bullet^*$ in $\cK_2(\cP)$.
\end{lemma}

\subsection{Ringel-Hall algebra}\label{Hall algebra for abelian category}

We refer to \cite{Schiffmann-2012-1} for details.

Let $\bbF_q$ be the finite field of order $q$ and $v_q\in \bbC$ be a fixed square root of $q$. Let $\cA_q$ be a category satisfying the following conditions 
\begin{equation}\label{*}\tag{$*$}
\begin{aligned}
&\textrm{abelian, essentially small, $\bbF_q$-linear, $\Hom$-finite with}\\
&\textrm{finite global dimension and enough projective objects.} 
\end{aligned}
\end{equation}
Let $\mathrm{Iso}(\cA_q)$ be the set of isomorphism classes of objects of $\cA_q$. For any object $M\in \cA_q$, let $[M]\in \mathrm{Iso}(\cA_q)$ be the isomorphism class of $M$. The Ringel-Hall algebra $\cH(\cA_q)$ of $\cA_q$ is the $\bbC$-vector space having a basis $\{u_{[M]}\mid [M]\in \mathrm{Iso}(\cA_q)\}$ together with the multiplication defined by
$$u_{[M]}\diamond u_{[N]}=\sum_{[L]\in \mathrm{Iso}(\cA_q)}g^L_{MN}u_{[L]}\ \textrm{for any}\ [M],[N]\in \mathrm{Iso}(\cA_q),$$
where $g^L_{MN}$ is the number of subobjects $L'$ of $L$ such that $L/L'\cong M,L'\cong N$. It is an associative algebra with the unit $u_{[0]}$. For any object $M\in \cA_q$, let $a_M=|\Aut_{\cA_q}(M)|$. By Riedtmann-Peng's formula (see \cite[II. Section 4]{Ringel-1996}),
$$g^L_{MN}=\frac{|\Ext_{\cA_q}^1(M,N)_L|}{|\Hom_{\cA_q}(M,N)|}\frac{a_L}{a_Ma_N}\ \textrm{for any}\ [M],[N],[L]\in \mathrm{Iso}(\cA_q),$$
where $\Ext_{\cA_q}^1(M,N)_L\subset \Ext_{\cA_q}^1(M,N)$ consists of the equivalence classes of short exact sequences whose middle terms are isomorphic to $L$, the multiplication can be rewritten as
$$(a_Mu_{[M]})\diamond (a_Nu_{[N]})=\sum_{[L]\in \mathrm{Iso}(\cA_q)}\frac{|\Ext_{\cA_q}^1(M,N)_L|}{|\Hom_{\cA_q}(M,N)|}(a_Lu_{[L]})\ \textrm{for any}\ [M],[N]\in \mathrm{Iso}(\cA_q)$$
under another basis $\{a_Mu_{[M]}\mid [M]\in \mathrm{Iso}(\cA_q)\}$. 

Let $K(\cA_q)$ be the Grothendieck group of $\cA_q$, and $\hat{M}\in K(\cA_q)$ be the image of any object $M\in \cA_q$ in the Grothendieck group. The Euler form $\langle-,-\rangle$ of $\cA_q$ is the bilinear form $K(\cA_q)\times K(\cA_q)\rightarrow \bbZ$ induced by 
$$\langle \hat{M},\hat{N}\rangle=\sum_{s\in \bbZ}(-1)^s\dim_{\bbF_q}\Ext^s_{\cA_q}(M,N)\ \textrm{for any}\ M,N\in \cA_q.$$
Note that the summation is finite, since $\cA_q$ has finite global dimension. The twisted Ringel-Hall algebra $\cH^{\tw}(\cA_q)$ of $\cA_q$ is the $\bbC$-vector space $\cH(\cA_q)$ together with a twisted multiplication defined by
$$u_{[M]}*u_{[N]}=v_q^{\langle \hat M,\hat N\rangle}u_{[M]}\diamond u_{[N]}\ \textrm{for any}\ [M],[N]\in \mathrm{Iso}(\cA_q).$$

\subsubsection{Ringel-Hall algebra of quiver representations}\label{Hall algebra for the category of quiver representations}

Let $\cA_q=\rep_{\bbF_q}(\mathbf{Q})$ be the category of finite-dimensional representations of the Dynkin quiver $\mathbf{Q}$ over $\bbF_q$. Then $\cA_q$ satisfies the conditions (\ref{*}). For any $i\in I$, let $S_i\in \cA_q$ be the corresponding simple representation. Let $\cH^{\tw}(\cA_q)$ be the corresponding twisted Ringel-Hall algebra.

\begin{theorem}[{\cite{Ringel-1990}}]\label{Ringel's isomorphism}
There is an algebra isomorphism
\begin{align*}
\rU_{v_q}(\fn^+)\rightarrow \cH^{{\rm{tw}}}(\cA_q), E_i\mapsto u_{[S_i]}.
\end{align*}
\end{theorem}

\subsubsection{Bridgeland's Hall algebra}\label{Hall algebra for the category of two-periodic complexes}

We refer to \cite{Bridgeland-2013} for details. 

Let $\cA_q$ be an abelian category satisfying the conditions (\ref{*}). Note that the category $\cC_2(\cA_q)$ is not of finite global dimension in general, but the original Hall algebra $\cH(\cC_2(\cA_q))$ still makes sense, because it does not involve the twist by the Euler form of $\cC_2(\cA_q)$. Let $\cH(\cC_2(\cP_q))\subset \cH(\cC_2(\cA_q))$ be the subspace spanned by $u_{[M_\bullet]}$ for any $M_\bullet\in \cC_2(\cP_q)$. Then it is a subalgebra, since $\cC_2(\cP_q)\subset \cC_2(\cA_q)$ is an extension closed subcategory. The twisted form $\cH^{\tw}(\cC_2(\cP_q))$ is the $\bbC$-vector space $\cH(\cC_2(\cP_q))$ together with a twisted multiplication defined by
$$u_{[M_\bullet]}*u_{[N_\bullet]}=v_q^{\langle \hat M_0,\hat N_0\rangle+\langle \hat M_1,\hat N_1\rangle}u_{[M_\bullet]}\diamond u_{[N_\bullet]}\ \textrm{for any}\ [M_\bullet],[N_\bullet]\in \mathrm{Iso}(\cC_2(\cP_q)).$$
For any $P\in \cP_q$, let 
$b_{K_P}=a_{K_P}u_{[K_P]},\ b_{K_P^*}=a_{K_P^*}u_{[K_P^*]}\in \cH^{\tw}(\cC_2(\cP_q))$. Then the subset $\{b_{K_P},b_{K_P^*}\mid P\in \cP_q\}$ satisfies the Ore conditions, and so there is a well-defined localization
$$\cD\cH(\cA_q)=\cH^{\tw}(\cC_2(\cP_q))[b_{K_P}^{-1},b_{K_P^*}^{-1}\mid P\in \cP_q].$$
The Bridgeland's Hall algebra $\cD\cH^{\red}(\cA_q)$ of $\cA_q$ is the reduced quotient 
$$\cD\cH^{\red}(\cA_q)=\cD\cH(\cA_q)/\langle b_{K_P}*b_{K_P^*}-1\mid P\in \cP_q\rangle.$$

Now, let $\cA_q=\rep_{\bbF_q}(\mathbf{Q})$. For any $\alpha\in K(\cA_q)$, it can be written as $\alpha=\hat{P}-\hat{Q}$ for some $P,Q\in \cP_q$, and there is a well-defined element $b_\alpha=b_{K_P}*b_{K_Q^*}\in \cD\cH^{\red}(\cA_q)$. For any $M\in \cA_q$, let $0\rightarrow P\xrightarrow{f}Q\xrightarrow{g}M\rightarrow 0$ be the minimal projective resolution of $M$, and let
\begin{align*}
&\xymatrix{C_M=(P \ar@<.5ex>[r]^-{f} &Q)\in \cC_2(\cP_q), \ar@<.5ex>[l]^{0}}\\
E_M=v_q^{\langle \hat P,\hat M\rangle}b_{-\hat{P}}*(a_{C_M}&u_{[C_M]}),\ F_M=v_q^{\langle \hat P,\hat M\rangle}b_{\hat{P}}*(a_{C_M^*}u_{[C_M^*]})\in \cD\cH^{\red}(\cA_q).
\end{align*}

\begin{theorem}[{\cite[Theorem 4.9]{Bridgeland-2013}}]\label{Bridgeland's isomorphism}
There is an algebra isomorphism
\begin{align*}
\rU_{v_q}(\fg)&\rightarrow \cD\cH^{\red}(\cA_q)\\
E_i\mapsto E_{S_i}/(q-1), F_i\mapsto -v_q&F_{S_i}/(q-1), K_i\mapsto b_{\hat S_i}, K_i^{-1}\mapsto b_{-\hat S_i}.
\end{align*}
\end{theorem}

\subsection{Mixed equivariant semisimple complex}\label{Mixed equivariant semisimple complex}

We refer to \cite{Pramod-2021,Bernstein-Lunts-1994,Beilinson-Bernstein-Deligne-1982,Kiehl-Rainer-2001} for details. 

Let $p,l$ be two fixed distinct prime numbers and $q=p^n$ for some positive integer $n$. Let $k=\overline{\bbF}_q$ be the algebraic closure of $\bbF_q$, and $\overline{\bbQ}_l$ be the algebraic closure of the field of $l$-adic numbers. Throughout this paper, all varieties are algebraic over $k$, and all sheaves are $\overline{\bbQ}_l$-sheaves.

Let $X$ be a $k$-variety with a $\bbF_q$-structure, and $G$ be a connected algebraic group with a $\bbF_q$-structure acting on $X$. We denote by
\begin{itemize}
\setlength{\itemindent}{-1cm}
\item $\cD^b(X)$ the bounded derived category of constructible sheaves on $X$;
\item $\cD^{b,ss}(X)$ the subcategory of $\cD^b(X)$ consisting of semisimple complexes;
\item $\cD_{G}^b(X)$ the $G$-equivariant bounded derived category of constructible sheaves on $X$;
\item $\cD_{G}^{b,ss}(X)$ the subcategory of $\cD_{G}^b(X)$ consisting of semisimple complexes;
\item $\cD_{G,m}^b(X)$ the subcategory of $\cD_{G}^b(X)$ consisting of mixed Weil complexes;
\item $\cD_{G,m}^{b,ss}(X)$ the subcategory of $\cD_{G}^b(X)$ consisting of mixed semisimple Weil complexes.
\end{itemize}

For any $n\in \bbZ$, let $[n]$ be the shift functor, and $(\frac{n}{2})$ be the Tate twist if $n$ is even or the square root of the Tate twist if $n$ is odd. For any  morphism $f:X\rightarrow Y$ between varieties, the derived functors of $f^*,f_*,f_!$ are still denoted by $f^*, f_*,f_!$ respectively. 

Let $f:X\rightarrow Y$ be a principal $G$-bundle. By \cite[Theorem 6.5.9, Proposition 6.2.10]{Pramod-2021}, $f^*:\cD^b(Y)\rightarrow \cD^b_{G}(X)$ is an equivalence, and it can be restricted to be an equivalence $f^*:\cD^{b,ss}(Y)\rightarrow \cD^{b,ss}_{G}(X)$. The quasi-inverse $f_{\flat}$ of $f^*$ is called the equivariant descent functor.


\subsection{Lusztig's categorification for $\bfU_v^+$}\label{Lusztig's categorification}

In this subsection, we review Lusztig's categorification for $\bfU_v^+$ and refer to \cite{Lusztig-1991,Lusztig-1993,Schiffmann-2012} for details. 

Let $\mathbf{Q}=(I,H,s,t)$ be the Dynkin quiver. For any $\nu\in \bbN[I]$, we fix a $I$-graded $k$-vector space $V$ of dimension vector $\nu$, and define 
$$\rE_{\nu}=\bigoplus_{h\in H}\Hom_k(V_{s(h)},V_{t(h)}),\ \rG_{\nu}=\prod_{i\in I}\textrm{GL}_k(V_i).$$
Then $\rG_{\nu}$ acts on $\rE_{\nu}$ by $(g.x)_h=g_{t(h)}x_hg_{s(h)}^{-1}$ for any $h\in H$.

For any $x\in \rE_{\nu}$, a $I$-graded subspace $W\subset V$ is called $x$-stable, if $x_h(W_{s(h)})\subset W_{t(h)}$ for any $h\in H$.

Let $\nu,\nu',\nu''\in \bbN [I]$ be such that $\nu=\nu'+\nu''$. We have fixed $I$-graded vector spaces $V,V',V''$ of dimension vectors $\nu,\nu',\nu''$ respectively. For any $x\in \rE_{\nu}$ and $x$-stable $I$-graded subspace $W\subset V$ of dimension vector $\nu''$, let $x_W:W\rightarrow W$ and $x_{V/W}:V/W\rightarrow V/W$ be the restriction and the quotient of $x$ respectively. Furthermore, for any $I$-graded linear isomorphisms $\rho_1:V/W\rightarrow V',\rho_2:W\rightarrow V''$, we define $\rho_{1}.x_{V/W}\in \rE_{\nu'}$ and $\rho_2.x_W\in \rE_{\nu''}$ by 
$$(\rho_{1}.x_{V/W})_h=(\rho_{1})_{t(h)}(x_{V/W})_h(\rho_1)_{s(h)}^{-1},\ (\rho_{2}.x_W)_h=(\rho_{2})_{t(h)}(x_W)_h(\rho_2)_{s(h)}^{-1}$$ 
for any $h\in H$. Let $\rE''$ be the variety of $(x,W)$, where $x\in \rE_{\nu}$ and $W\subset V$ is a $x$-stable $I$-graded subspace of dimension vector $\nu''$. Then $\rG_{\nu}$ acts on it by $g.(x,W)=(g.x, g(W))$. Let $\rE'$ be the variety of $(x,W,\rho_1,\rho_2)$, where $(x,W)\in \rE''$ and $\rho_1:V/W\rightarrow V',\rho_2:W\rightarrow V''$ are $I$-graded linear isomorphisms. Then $\rG_{\nu'}\times \rG_{\nu''}\times \rG_{\nu}$ acts on it by $(g',g'',g).(x,W,\rho_1,\rho_2)=(g.x,g(W),g'\rho_1g^{-1},g''\rho_2g^{-1})$.

Consider the following morphisms
$$\xymatrix{\rE_{\nu'}\times \rE_{\nu''} &\rE' \ar[l]_-{p_1} \ar[r]^-{p_2} &\rE'' \ar[r]^-{p_3} &\rE_{\nu},}$$
where $p_1(x,W,\rho_1,\rho_2)=(\rho_{1}.x_{V/W}, \rho_2.x_W), p_2(x,W,\rho_1,\rho_2)=(x,W), p_3(x,W)=x$. Note that $p_1$ is smooth with connected fibers and $\rG_{\nu'}\times \rG_{\nu''}\times \rG_{\nu}$-equivariant with respect to the trivial action of $\rG_{\nu}$ on $\rE_{\nu'}\times \rE_{\nu''}$, $p_2$ is a principal $\rG_{\nu'}\times \rG_{\nu''}$-bundle and $\rG_{\nu}$-equivariant, $p_3$ is proper and $\rG_{\nu}$-equivariant. 

The induction functor $\Ind^{\nu}_{\nu',\nu''}:\cD^{b}_{\rG_{\nu'}}(\rE_{\nu'})\boxtimes\cD^{b}_{\rG_{\nu''}}(\rE_{\nu''})\rightarrow \cD^{b}_{\rG_\nu}(\rE_\nu)$ is defined by 
$$\Ind^{\nu}_{\nu',\nu''}(A\boxtimes B)=(p_3)_!(p_2)_\flat(p_1)^*(A\boxtimes B)[d_1-d_2](\frac{d_1-d_2}{2})$$
for any $A\in \cD^{b}_{\rG_{\nu'}}(\rE_{\nu'}),B\in \cD^{b}_{\rG_{\nu''}}(\rE_{\nu''})$, where $d_1,d_2$ are the dimensions of the fibers of $p_1,p_2$ respectively.  By \cite{Beilinson-Bernstein-Deligne-1982} (also see \cite[\S 2.4]{Fang-Lan-Xiao-2023}), $\Ind^{\nu}_{\nu',\nu''}$ can be restricted to 
$$\Ind^{\nu}_{\nu',\nu''}:\cD^{b,ss}_{\rG_{\nu'},m}(\rE_{\nu'})\boxtimes\cD^{b,ss}_{\rG_{\nu''},m}(\rE_{\nu''})\rightarrow \cD^{b,ss}_{\rG_\nu,m}(\rE_\nu).$$ 

To define the restriction functor, we fix a $I$-graded subspace $W\subset V$ of dimension vector $\nu''$, and two $I$-graded linear isomorphisms $\rho_1:V/W\rightarrow V', \rho_2:W\rightarrow V''$. Let $\rQ\subset \rG_\nu$ be the stabilizer of $W\subset V$ and $\rU$ be the unipotent radical of $\rQ$. Then there is a canonical isomorphism $\rQ/\rU\cong \rG_{\nu'}\times \rG_{\nu''}$. Let $\rF$ be the closed subvariety of $\rE_{\nu}$ consisting of $x$ such that $W$ is $x$-stable, then $\rQ$ acts on $\rF$ through the embedding $\rQ\hookrightarrow \rG_{\nu}$, and acts on $\rE_{\nu'}\times \rE_{\nu''}$ through the quotient $\rQ\rightarrow \rQ/\rU\cong \rG_{\nu'}\times \rG_{\nu''}$. 

Consider the following morphisms
$$\xymatrix{\rE_{\nu'}\times \rE_{\nu''} &\rF \ar[l]_-{\kappa} \ar[r]^-{\iota} &\rE_{\nu},}$$
where $\kappa(x,W)=(\rho_{1}.x_{V/W}, \rho_2.x_W), \iota(x)=x$. Note that $\kappa$ is a vector bundle and $\rQ$-equivariant, $\iota$ is the inclusion and $\rQ$-equivariant. 

The restriction functor $\Res^{\nu}_{\nu',\nu''}:\cD^b_{\rG_{\nu}}(\rE_{\nu})\rightarrow \cD^b_{\rG_{\nu'}\times \rG_{\nu''}}(\rE_{\nu'}\times \rE_{\nu''})$ is defined by
$$\Res^{\nu}_{\nu',\nu''}(C)=\kappa_!\iota^*(C)[-\langle\nu',\nu''\rangle](-\frac{\langle\nu',\nu''\rangle}{2})$$
for any $C\in \cD^b_{\rG_{\nu}}(\rE_{\nu})$, where $\langle\nu',\nu''\rangle=\sum_{i\in I}\nu'_i\nu''_i-\sum_{h\in H}\nu'_{s(h)}\nu''_{t(h)}$ is the Euler form. By \cite[\S 2.5]{Fang-Lan-Xiao-2023}, $\Res^{\nu}_{\nu',\nu''}$ can be restricted to 
$$\cD^{b,ss}_{\rG_{\nu},m}(\rE_{\nu})\rightarrow\cD^{b,ss}_{\rG_{\nu'}\times \rG_{\nu''},m}(\rE_{\nu'}\times \rE_{\nu''}).$$ 

Let $\cK_{\nu}^+$ be the Grothendieck group of $\cD^{b,ss}_{\rG_\nu}(\rE_\nu)$  and $\cK^+=\bigoplus_{\nu\in \bbN[I]}\cK_{\nu}^+$. We define a $\cZ$-module structure on them by $v.[L]=[L[-1](-\frac{1}{2})]$.

\begin{theorem}[{\cite[Chapter 13]{Lusztig-1993}}]
 All induction functors define a multiplication and all restriction functors define a comultiplication on $\cK^+$ such that $\cK^+$ is isomorphic to $_{\cZ}\bfU_v^+$ as bialgebras.
\end{theorem}

The $\cZ$-algebra $\cK^+$ has a $\cZ$-basis $\cB=\bigsqcup_{\nu\in \bbN[I]}\cB_{\nu}$ formed by $\rG_{\nu}$-equivariant simple perverse sheaves. The image of $\cB$ under the isomorphism $\cK^+\cong \bfU_v^+$ is called the canonical basis of $\bfU_v^+$. Since $\mathbf{Q}$ is a Dynkin quiver, for any $\nu\in \bbN[I]$ we have
$$\cB_{\nu}=\{[\IC(\cO, \overline{\bbQ}_l)(\frac{\dim \cO}{2})]\mid \cO\subset\rE_{\nu}\ \textrm{is a $\rG_{\nu}$-orbit}\}.$$

\section{Induction functor and restriction functor}\label{Induction functor and restriction functor}

Let $k=\overline{\bbF}_q, \cA=\rep_k(\mathbf{Q})$ be the category of finite-dimensional representations of the Dynkin quiver $\mathbf{Q}$ over $k$, and $\cC_2(\cP)$ be the category of two-periodic projective complexes of $\cA$. Since $\cA$ has the Krull-Schmidt property, we can fix a complete set of indecomposable projective objects $\{P_i\mid i\in I\}$ up to isomorphism.

\subsection{Moduli variety of two-periodic projective complexes}\label{Moduli variety of two-periodic projective complexes}

For any object $M_\bullet=(M^1,M^0,d^1,d^0)\in \cC_2(\cP)$, we define its projective dimension vector pair to be $\ue_{M_\bullet}=(e^1,e^0)\in \bbN[I]\times \bbN[I]$ such that $M^j\cong \bigoplus_{i\in I}e^j_iP_i$, where $e^j_iP_i$ is the direct sum of $e^j_i$-copies of $P_i$ for any $j\in \bbZ_2$ and $i\in I$.

\begin{definition}
Let $\ue=(e^1,e^0)\in \bbN[I]\times \bbN[I]$. Then we define an affine variety 
$$\rC_{\ue}=\{(d^1,d^0)\in \Hom_{\cA}(P^1,P^0)\times \Hom_{\cA}(P^0,P^1)\mid d^{j+1}d^j=0\},$$
where $P^j=\bigoplus_{i\in I}e^j_iP_i$ for any $j\in \bbZ_2$.
\end{definition}

It is clear that any $(d^1,d^0)\in \rC_{\ue}$ determines an object  $M_\bullet\in\cC_2(\cP)$ such that $M^j=P^j$ for any $j\in \bbZ_2$. In this case, we write $M_\bullet=(d^1,d^0)\in \rC_{\ue}$.

Recall that in \S \ref{Lusztig's categorification}, for any $\nu\in \bbN[I]$, we define an affine space $\rE_\nu$ together with an algebraic group $\rG_\nu$-action. There is a bijection between the set of isomorphism classes of representations of dimension vector $\nu$ and the set of $\rG_\nu$-orbits. Moreover, for any $x\in \rE_\nu$, the automorphism group $\Aut_{\cA}(V, x)$ of the representation $(V,x)$ is isomorphic to the stabilizer $\Stab(\cO_x)$ of the $\rG_\nu$-orbit $\cO_x$. Thus $\Stab(\cO_x)$ is a connected algebraic group, since it is isomorphic to an open dense subset $\Aut_{\cA}(V, x)$ of the affine space $\End_{\cA}(V,x)$ which consists of invertible elements (see \cite[Lemma 10.1.7]{Pramod-2021}). 

\begin{definition}
Let $x^j\in \rE_{\nu^j}$ be such that $(V^j,x^j)\cong P^j$ for any $j\in \bbZ_2$. Then we define the connected algebraic group 
$$\rG_{\ue}=\Stab(\cO_{x^1})\times \Stab(\cO_{x^0})\cong \Aut_{\cA}(P^1)\times \Aut_{\cA}(P^0)$$
which acts on $\rC_{\ue}$ by $(g^1,g^0).(d^1,d^0)=(g^0d^1(g^1)^{-1},g^1d^0(g^0)^{-1})$.
\end{definition}

By definitions, there is a bijection between the set of isomorphism classes of objects in $\cC_2(\cP)$ of projective dimension vector pair $\ue$ and the set of $\rG_{\ue}$-orbits, denoted by $[M_{\bullet}]\leftrightarrow \cO_{M_{\bullet}}$. Then the isomorphism group $\Aut_{\cC_2(\cP)}(M_{\bullet})$ of $M_{\bullet}$ is isomorphic to the stabilizer $\Stab(\cO_{M_\bullet})$ of $\cO_{M_{\bullet}}$. Thus $\Stab(\cO_{M_\bullet})$ is a connected algebraic group, since it is isomorphic to an open dense subset $\Aut_{\cC_2(\cP)}(M_{\bullet})$ of the affine space $\End_{\cC_2(\cP)}(M_{\bullet})$ consisting of invertible elements. 

\begin{lemma}\label{Finite orbits}
Let $\ue\in \bbN[I]\times \bbN[I]$. Then\\
(a) there are only finitely many $\rG_{\ue}$-orbits in $\rC_{\ue}$.\\
(b) the constant sheaf $\overline{\bbQ}_l$ is the unique $\rG_{\ue}$-equivariant local system on each $\rG_{\ue}$-orbit.
\end{lemma}
\begin{proof}
(a) By \cite[Lemma 4.2]{Bridgeland-2013}, any object $L_{\bullet}\in \cC_2(\cP)$ can be decomposed as $L_\bullet\cong C_M\oplus C_N^*\oplus K_P\oplus K_Q^*$, where $M=H^0(L_{\bullet}), N=H^1(L_{\bullet})$. There are only finitely many dimension vectors of $H^0(L_{\bullet}), H^1(L_{\bullet})$ for various $L_{\bullet}\in \rC_{\ue}$. Since $\mathbf{Q}$ is a Dynkin quiver, there are  finitely many isomorphism classes of representations of a fixed dimension vector, and so there are finitely many isomorphism classes of $L_{\bullet}$, that is, there are only finitely many $\rG_{\ue}$-orbits in $\rC_{\ue}$.\\
(b) By \cite[Proposition 6.2.13]{Pramod-2021}, any $\rG_{\ue}$-equivariant local system on a $\rG_{\ue}$-orbit must be constant, since the stabilizer of this $\rG_{\ue}$-orbit is connected.
\end{proof}

\subsection{Induction functor}\label{Induction functor}

Let $\ue,\ue',\ue''\in \bbN[I]\times \bbN[I]$ be such that $\ue=\ue'+\ue''$. Let $P^j=\bigoplus_{i\in I}e^j_iP_i, P'^j=\bigoplus_{i\in I}e'^j_iP_i,P''^j=\bigoplus_{i\in I}e''^j_iP_i$ so that $P^j=P'^j\oplus P''^j$ for any $j\in \bbZ_2$.

For any $(d^1,d^0)\in \rC_{\ue}$ and direct summand $W^j$ of $P^j$  for any $j\in \bbZ_2$, the pair $(W^1,W^0)$ is said to be $(d^1,d^0)$-stable, if $d^j(W^j)\subset W^{j+1}$ for any $j\in \bbZ_2$. For any $(d^1,d^0)$-stable pair $(W^1,W^0)$ such that $W^j\cong P''^j$, let $d^j_{W^j}:W^j\rightarrow W^{j+1}$ and $d^j_{P^j/W^j}:P^j/W^j\rightarrow P^{j+1}/W^{j+1}$  be the restriction and the quotient of $d^j$ respectively for any $j\in \bbZ_2$. Furthermore, for any isomorphisms $\rho^j_1:P^j/W^j\rightarrow P'^j$ and $\rho^j_2:W^j\rightarrow P''^j$ in $\cA$  for any $j\in \bbZ_2$, we define
\begin{align*}
&\rho_{1}.(d^1,d^0)=(\rho^0_1d^1_{P^1/W^1}(\rho^1_1)^{-1},\rho^1_1d^0_{P^0/W^0}(\rho^0_1)^{-1})\in \rC_{\ue'},\\
&\rho_{2}.(d^1,d^0)=(\rho^0_2d^1_{W^1}(\rho^1_2)^{-1},\rho^1_2d^0_{W^0}(\rho^0_2)^{-1})\in \rC_{\ue''}.
\end{align*}

Let $\rC''$ be the variety of $(d^1,d^0,W^1,W^0)$, where $(d^1,d^0)\in \rC_{\ue}$ and $(W^1,W^0)$ is a $(d^1,d^0)$-stable pair such that $W^j\cong P''^j$  for any $j\in \bbZ_2$. Then the group $\rG_{\ue}$ acts on it by 
$(g^1,g^0).(d^1,d^0,W^1,W^0)=((g^1,g^0).(d^1,d^0),g^1(W^1),g^0(W^0))$. 

Let $\rC'$ be the variety of $(d^1,d^0,W^1,W^0, \rho^1_1,\rho^0_1,\rho^1_2,\rho^0_2)$, where $(d^1,d^0,W^1,W^0)\in \rC''$ and $\rho^j_1:P^j/W^j\rightarrow P'^j, \rho^j_2:W^j\rightarrow P''^j$ are isomorphisms in $\cA$  for any $j\in \bbZ_2$. Then the group $\rG_{\ue'}\times \rG_{\ue''}\times \rG_{\ue}$ acts on it by 
\begin{align*}
&(g'^1,g'^0,g''^1,g''^0,g^1,g^0).(d^1,d^0,W^1,W^0, \rho^1_1,\rho^0_1,\rho^1_2,\rho^0_2)=\\
&((g^1,g^0).(d^1,d^0,W^1,W^0),g'^1\rho^1_1(g^1)^{-1}\!,g'^0\rho^0_1(g^0)^{-1}\!,g''^1\rho^1_2(g^1)^{-1}\!,g''^0\rho^0_2(g^0)^{-1}).
\end{align*}

Consider the following morphisms
Consider the following morphisms
$$\xymatrix{\rC_{\ue'}\times \rC_{\ue''} &\rC' \ar[l]_-{p_1} \ar[r]^-{p_2} &\rC'' \ar[r]^-{p_3} &\rC_{\ue},}$$
where 
\begin{align*}
&p_1(d^1,d^0,W^1,W^0, \rho^1_1,\rho^0_1,\rho^1_2,\rho^0_2)=(\rho_{1}.(d^1,d^0),\rho_{2}.(d^1,d^0)),\\
&p_2(d^1,d^0,W^1,W^0,\rho^1_1,\rho^0_1,\rho^1_2,\rho^0_2)=(d^1,d^0,W^1,W^0),\\
&p_3(d^1,d^0,W^1,W^0)=(d^1,d^0).
\end{align*}
Note that $p_3$ is proper and $\rG_{\ue}$-equivariant, since its fiber at $(d^1,d^0)$ is a closed subset of the product $\prod_{i\in I,j\in \bbZ_2}\Gr(e^j_i,e''^j_i)$ of Grassmannians which is projective, $p_2$ is a principal $\rG_{\ue'}\times \rG_{\ue''}$-bundle and $\rG_{\ue}$-equivariant, and $p_1$ is $\rG_{\ue'}\times \rG_{\ue''}\times \rG_{\ue}$-equivariant with respect to the trivial action of $\rG_{\ue}$ on $\rC_{\ue'}\times \rC_{\ue''}$.

\begin{lemma}\label{fibre of p_1}
The fiber of the morphism $p_1$ at $(M_\bullet,N_\bullet)\in \rC_{\ue'}\times \rC_{\ue''}$ is isomorphic to $$\prod_{i\in I,j\in \bbZ_2}\Gr(e^j_i,e''^j_i)\times \rG_{\ue'}\times \rG_{\ue''}\times \Hom_{\cC_2(\cP)}(M_\bullet,N_\bullet^*).$$
\end{lemma}
\begin{proof}
Let $(d^1,d^0,W^1,W^0, \rho^1_1,\rho^0_1,\rho^1_2,\rho^0_2)$ be an element in  the fiber of $p_1$ at $(M_\bullet,N_\bullet)=((d'^1,d'^0),(d''^1,d''^0))\in \rC_{\ue'}\times \rC_{\ue''}$, it is clear that giving $(W^1,W^0,\rho^1_1,\rho^0_1,\rho^1_2,\rho^0_2)$ is the same as giving an element in $$\prod_{i\in I,j\in \bbZ_2}\textrm{Gr}(e^j_i,e''^j_i)\times \rG_{\ue'}\times \rG_{\ue''}.$$
For any fixed $(W^1,W^0,\rho^1_1,\rho^0_1,\rho^1_2,\rho^0_2)$, it remains to find $(d^1,d^0)\in \rC_{\ue}$ such that $(W^1,W^0)$ is $(d^1,d^0)$-stable and $\rho_{1}.(d^1,d^0)=(d'^1,d'^0),\rho_{2}.(d^1,d^0)=(d''^1,d''^0)$. Such $(d^1,d^0)$ are of the form
\begin{equation}\label{cone complex}
\xymatrix@C=6cm{{P^1/W^1\bigoplus W^1} \ar@<.5ex>[r]^-{{d^1=\begin{pmatrix}(\rho^0_1)^{-1}d'^1\rho^1_1 &0\\
(\rho^0_2)^{-1}(-f^1)\rho^1_1 &(\rho^0_2)^{-1}d''^1\rho^1_2\end{pmatrix}}} &{P^0/W^0\bigoplus W^0} \ar@<.5ex>[l]^-{{d^0=\begin{pmatrix}(\rho^1_1)^{-1}d'^0\rho^0_1 &0\\
(\rho^1_2)^{-1}(-f^0)\rho^0_1 &(\rho^1_2)^{-1}d''^0\rho^0_2\end{pmatrix}}}}
\end{equation}
where $f^j\in \Hom_{\cA}(P'^j,P''^{j+1})$, then the condition $d^{j+1}d^j=0$ is equivalent to the condition that $(f^1,f^0)$ satisfies $(\rho^j_2)^{-1}f^{j+1}d'^j\rho^j_1+(\rho^j_2)^{-1}d''^{j+1}f^j\rho^j_1=0$, that is, $f^{j+1}d'^j=-d''^{j+1}f^j$ for any $j\in \bbZ_2$. Hence $(f^1,f^0)\in \Hom_{\cC_2(\cP)}(M_\bullet,N_\bullet^*)$.
\end{proof}
\begin{remark}\label{mapping cone}
The complex {\rm{(\ref{cone complex})}} is isomorphic to ${\rm{Cone}}\,(f^1,f^0)^*$ the shift of the mapping cone of $(f^1,f^0)\in \Hom_{\cC_2(\cP)}(M_\bullet,N_\bullet^*)$.
\end{remark}

For any $A\in \cD^b_{\rG_{\ue'}}(\rC_{\ue'}),B\in \cD^b_{\rG_{\ue''}}(\rC_{\ue''})$, we have $$A\boxtimes B\in \cD^b_{\rG_{\ue'}\times \rG_{\ue''}}(\rC_{\ue'}\times \rC_{\ue''})=\cD^b_{\rG_{\ue'}\times \rG_{\ue''}\times \rG_{\ue}}(\rC_{\ue'}\times \rC_{\ue''}),$$ 
then $(p_1)^*(A\boxtimes B)\in \cD^b_{\rG_{\ue'}\times \rG_{\ue''}\times \rG_{\ue}}(\rC')$ and $(p_2)_\flat(p_1)^*(A\boxtimes B)\in \cD^b_{\rG_{\ue}}(\rC'')$. Finally, we have $(p_3)_!(p_2)_\flat(p_1)^*(A\boxtimes B)\in \cD^b_{\rG_{\ue}}(\rC_{\ue}).$

\begin{definition}
The induction functor $\Ind^{\ue}_{\ue',\ue''}:\cD^b_{\rG_{\ue'}}(\rC_{\ue'})\boxtimes \cD^b_{\rG_{\ue''}}(\rC_{\ue''})\rightarrow \cD^b_{\rG_{\ue}}(\rC_{\ue})$ is defined by
$$\Ind^{\ue}_{\ue',\ue''}(A\boxtimes B)=(p_3)_!(p_2)_\flat(p_1)^*(A\boxtimes B)[-|\ue',\ue''|](-\frac{|\ue',\ue''|}{2})$$
for any $A\in \cD^b_{\rG_{\ue'}}(\rC_{\ue'})$ and $B\in \cD^b_{\rG_{\ue''}}(\rC_{\ue''})$, where 
\begin{align*}
|\ue',\ue''|=\langle P'^1,P''^1\rangle+\langle P'^0,P''^0\rangle
=\dim_k \Hom_{\cA}(P'^1,P''^1)+\dim_k \Hom_{\cA}(P'^0,P''^0).
\end{align*}
\end{definition}

We remark that $\Ind^{\ue}_{\ue',\ue''}(\cD^{b,ss}_{\rG_{\ue'}}(\rC_{\ue'})\boxtimes \cD^{b,ss}_{\rG_{\ue''}}(\rC_{\ue''}))\not\subset\cD^{b,ss}_{\rG_{\ue}}(\rC_{\ue})$, since the morphism $p_1$ is not smooth in general, which is different from Lusztig's induction functor $\Ind^\nu_{\nu',\nu''}$.

\begin{proposition}\label{associativity}
For any $\ue_1,\ue_2,\ue_3\in \bbN[I]\times \bbN[I]$, we have
$$\Ind^{\ue_1+\ue_2+\ue_3}_{\ue_1,\ue_2+\ue_3}(\Id \boxtimes \Ind^{\ue_2+\ue_3}_{\ue_2,\ue_3})\cong\Ind^{\ue_1+\ue_2+\ue_3}_{\ue_1+\ue_2,\ue_3}(\Ind^{\ue_1+\ue_2}_{\ue_1,\ue_2}\boxtimes \Id).$$
\end{proposition}
\begin{proof}
We denote by $^kP^j=\bigoplus_{i\in I}(e_k)_i^jP_i$ for $k=1,2,3$ and $j\in \bbZ_2$. It is clear that
$$|\ue_2,\ue_3|+|\ue_1,\ue_2+\ue_3|=|\ue_1,\ue_2|+|\ue_1+\ue_2,\ue_3|.$$ 
We only need to prove that 
\begin{equation}
\begin{aligned}\label{associative}
&(p_3)_!(p_2)_\flat(p_1)^*(1\times p_3)_!(1\times p_2)_\flat(1\times p_1)^*\\
\cong &(p_3)_!(p_2)_\flat(p_1)^*(p_3\times 1)_!(p_2\times 1)_\flat(p_1\times 1)^*,
\end{aligned}
\end{equation}
where $1$ is the identity morphism. Consider the following diagram
$$\xymatrix{ &Z \ar[r]^-{r} &\rC''^{\ue_1+\ue_2+\ue_3}_{\ue_1,\ue_2+\ue_3} \ar[r]^-{p_3} &\rC_{\ue_1+\ue_2+\ue_3}\\
Y \ar[d]_-{\hat{p_1}} \ar[r]^-{\tilde{p_2}} \ar@{}[dr]|{\square} &X \ar[u]^-s \ar@{}[ur]|{\square} \ar[r]^-{\tilde{p_3}} \ar[d]_-{\tilde{p_1}} \ar@{}[dr]|{\square} &\rC'^{\ue_1+\ue_2+\ue_3}_{\ue_1,\ue_2+\ue_3} \ar[u]_-{p_2} \ar[d]^-{p_1}\\
{\rC_{\ue_1}\times \rC'^{\ue_2+\ue_3}_{\ue_2,\ue_3}} \ar[r]^-{1\times p_2} \ar[d]_-{1\times p_1} &{\rC_{\ue_1}\times \rC''^{\ue_2+\ue_3}_{\ue_2,\ue_3}} \ar[r]^-{1\times p_3} &\rC_{\ue_1}\times \rC_{\ue_2+\ue_3}\\
\rC_{\ue_1}\times \rC_{\ue_2}\times \rC_{\ue_3},}$$
where 
\begin{align*}
X&=\rC'^{\ue_1+\ue_2+\ue_3}_{\ue_1,\ue_2+\ue_3}\times_{(\rC_{\ue_1}\times \rC_{\ue_2+\ue_3})}(\rC_{\ue_1}\times \rC''^{\ue_2+\ue_3}_{\ue_2,\ue_3}),\\ Y&=X\times_{(\rC_{\ue_1}\times \rC''^{\ue_2+\ue_3}_{\ue_2,\ue_3})}({\rC_{\ue_1}\times \rC'^{\ue_2+\ue_3}_{\ue_2,\ue_3}})
\end{align*}
are fiber products, $\hat{p_1},\tilde{p_1},\tilde{p_2},\tilde{p_3}$ are natural morphisms. Then $X$ consists of 
$$(d^1,d^0,W^1,W^0,U^1,U^0,\rho^1_1,\rho^0_1,\rho^1_2,\rho^0_2),$$ 
where $(d^1,d^0,W^1,W^0,\rho^1_1,\rho^0_1,\rho^1_2,\rho^0_2)\in \rC'^{\ue_1+\ue_2+\ue_3}_{\ue_1,\ue_2+\ue_3}$, $U^j$ is a direct summand of $W^j$ such that $U^j\cong {^3P^j}$ for any $j\in \bbZ_2$, and $(U^1,U^0)$ is $(d^1,d^0)$-stable. And $Y$ consists of $$(d^1,d^0,W^1,W^0,U^1,U^0,\rho^1_1,\rho^0_1,\rho^1_2,\rho^0_2,\rho'^1_1,\rho'^0_1,\rho'^1_2,\rho'^0_2),$$ 
where $(d^1,d^0,W^1,W^0,U^1,U^0,\rho^1_1,\rho^0_1,\rho^1_2,\rho^0_2)\in X$, $\rho'^j_1:W^j/U^j\xrightarrow{\cong} {^2}P^j, \rho'^j_2:U^j\xrightarrow{\cong} {^3P^j}$ are isomorphisms in $\cA$ for any $j\in \bbZ_2$. Let $Z$ be the variety of $(d^1,d^0,W^1,W^0,U^1,U^0)$, where $(d^1,d^0,W^1,W^0)\in \rC''^{\ue_1+\ue_2+\ue_3}_{\ue_1,\ue_2+\ue_3}$ and $U^j$ is a direct summand of $W^j$ such that $U^j\cong {^3P^j}$ for any $j\in \bbZ_2$, and $(U^1,U^0)$ is $(d^1,d^0)$-stable, and let $s:X\rightarrow Z, r:Z\rightarrow \rC''^{\ue_1+\ue_2+\ue_3}_{\ue_1,\ue_2+\ue_3}$ be natural projection, then $s$ is a principal $\rG_{\ue_1}\times \rG_{\ue_2+\ue_3}$-bundle and
$$\xymatrix{X \ar[d]_-s \ar[r]^-{\tilde{p_3}} &\rC'^{\ue_1+\ue_2+\ue_3}_{\ue_1,\ue_2+\ue_3} \ar[d]_-{p_2}\\
Z \ar[r] \ar@{}[ur]|{\square} &\rC''^{\ue_1+\ue_2+\ue_3}_{\ue_1,\ue_2+\ue_3}}$$
becomes a Cartesian square. By base change, the left hand side of (\ref{associative}) is 
\begin{align*}
&(p_3)_!(p_2)_\flat(p_1)^*(1\times p_3)_!(1\times p_2)_\flat(1\times p_1)^*\\
\cong &(p_3)_!(p_2)_\flat(\tilde{p_3})_!(\tilde{p_1})^*(1\times p_2)_\flat(1\times p_1)^*\\
\cong &(p_3)_!r_!s_\flat(\tilde{p_2})_\flat(\hat{p_1})^*(1\times p_1)^*.
\end{align*}
By similar argument, the right hand side of (\ref{associative}) is isomorphic to the same expression $(p_3)_!r_!s_\flat(\tilde{p_2})_\flat(\hat{p_1})^*(1\times p_1)^*$, as desired.
 \end{proof}
 
For $\underline{0}=(0,0)\in \bbN[I]\times \bbN[I]$, the variety $\rC_{\underline{0}}=\{(0,0)\}$ is a single point. By definition, $\Ind^{\ue}_{\ue,\underline{0}}(A\boxtimes \overline{\bbQ}_l|_{\rC_{\underline{0}}})\cong A\cong \Ind^{\ue}_{\underline{0},\ue}(\overline{\bbQ}_l|_{\rC_{\underline{0}}}\boxtimes A)$ for any $A\in \cD^b_{\rG_{\ue}}(\rC_{\ue})$ and $\ue\in \bbN[I]\times \bbN[I]$.

\subsection{Restriction functor}\label{Restriction functor}

Let $\ue,\ue',\ue''\in \bbN[I]\times \bbN[I]$ be such that $\ue=\ue'+\ue''$. Let $P^j=\bigoplus_{i\in I}e^j_iP_i, P'^j=\bigoplus_{i\in I}e'^j_iP_i,P''^j=\bigoplus_{i\in I}e''^j_iP_i$ so that $P^j=P'^j\oplus P''^j$ for any $j\in \bbZ_2$. To define the restriction functor, let $\rho^j_1:P^j/P''^j\xrightarrow{\cong} P'^j$ be the natural isomorphism and $\rho^j_2:P''^j\xrightarrow{1} P''^j$ be the identity morphism.
 
Let $x^j\in \rE_{\nu^j}$ be the element such that $(V^j,x^j)\cong P^j$, $\rQ^j\subset \Stab(\cO_{x^j})\cong \Aut_{\cA}(P^j)$ be the stabilizer of $P''^j\subset P^j$, and $\rU^j$ be the unipotent radical of $\rQ^j$ for any $j\in \bbZ_2$. We denote by $\rQ=\rQ^1\times \rQ^0$ and $\rU=\rU^1\times \rU^0$, then there is a canonical isomorphism $\rQ/\rU\cong \rG_{\ue'}\times \rG_{\ue''}$.

Let $\rF$ be the closed subvariety of $\rC_{\ue}$ consisting of $(d^1,d^0)$ such that $(P''^1,P''^0)$ is $(d^1,d^0)$-stable. Then $\rQ$ acts on $\rF$ via the embedding $\rQ\hookrightarrow \rG_{\ue}$, and acts on $\rC_{\ue'}\times \rC_{\ue''}$ through the quotient $\rQ\rightarrow\rQ/\rU\cong \rG_{\ue'}\times \rG_{\ue''}$.

Consider the following morphisms
$$\xymatrix{\rC_{\ue'}\times \rC_{\ue''} &\rF \ar[r]^-{\iota} \ar[l]_-{\kappa} &\rC_{\ue},}$$
where 
\begin{align*}
&\kappa(d^1,d^0)=(\rho_{1}.(d^1,d^0),\rho_{2}.(d^1,d^0)),\\
&\iota(d^1,d^0)=(d^1,d^0).
\end{align*}
It is clear that $\kappa$ and $\iota$ are $\rQ$-equivariant. By similar argument to Lemma \ref{fibre of p_1}, we have the following lemma.

\begin{lemma}\label{fibre of kappa}
The fiber of the morphism $\kappa$ at $(M_\bullet, N_\bullet)\in \rC_{\ue'}\times \rC_{\ue''}$ is isomorphic to $\Hom_{\cC_2(\cP)}(M_\bullet,N^*_\bullet)$.
\end{lemma}

For any $C\in \cD^b_{\rG_{\ue}}(\rC_{\ue})$, it can be viewed as an object of $\cD^b_{\rQ}(\rC_{\ue})$ via the forgetful functor $\cD^b_{\rG_{\ue}}(\rC_{\ue})\rightarrow \cD^b_{\rQ}(\rC_{\ue})$, then $\kappa_!\iota^*(C)\in \cD^b_{\rQ}(\rC_{\ue'}\times \rC_{\ue''})\simeq \cD^b_{\rG_{\ue'}\times \rG_{\ue''}}(\rC_{\ue'}\times \rC_{\ue''})$, where the equivalence of categories follows from \cite[Theorem 6.6.16]{Pramod-2021}, since $\rU$ acts trivially on $\rC_{\ue'}\times \rC_{\ue''}$.

\begin{definition}
The restriction functor $\Res^{\ue}_{\ue',\ue''}:\cD^b_{\rG_{\ue}}(\rC_{\ue})\rightarrow \cD^b_{\rG_{\ue'}\times \rG_{\ue''}}(\rC_{\ue'}\times \rC_{\ue''})$ is defined by
$$\Res^{\ue}_{\ue',\ue''}(C)=\kappa_!\iota^*(C)[|\ue',\ue''|](\frac{|\ue',\ue''|}{2})$$
for any $C\in \cD^b_{\rG_{\ue}}(\rC_{\ue})$, where 
$$|\ue',\ue''|=\langle P'^1,P''^1\rangle+\langle P'^0,P''^0\rangle=\dim_k \Hom_{\cA}(P'^1,P''^1)+\dim_k \Hom_{\cA}(P'^0,P''^0).$$
\end{definition}

We will prove that the restriction functor is a hyperbolic localization functor in Proposition \ref{restriction-hyperbolic}, which is similar to Lusztig's restriction functor $\Res^\nu_{\nu',\nu''}$. As a result, it can be restricted to  
$$\cD^{b,ss}_{\rG_{\ue},m}(\rC_{\ue})\rightarrow \cD^{b,ss}_{\rG_{\ue'},m}(\rC_{\ue'})\boxtimes \cD^{b,ss}_{\rG_{\ue''},m}(\rC_{\ue''}).$$

For any $\ue_1,\ue_2,\ue_3\in \bbN[I]\times \bbN[I]$, let $\Res^{\ue_1+\ue_2}_{\ue_1,\ue_2}\times \Id=(\kappa\times 1)_!(\iota\times 1)^*[|\ue_1,\ue_2|](\frac{|\ue_1,\ue_2|}{2})$
be the functor defined by the morphisms
$$\xymatrix{\rC_{\ue_1}\times \rC_{\ue_2}\times \rC_{\ue_3} &\rF^{\ue_1+\ue_2}_{\ue_1,\ue_2}\times \rC_{\ue_3} \ar[l]_-{\kappa\times 1} \ar[r]^-{\iota\times 1} &\rC_{\ue_1+\ue_2}\times \rC_{\ue_3}.}$$
Similarly, let $\Id\times \Res^{\ue_2+\ue_3}_{\ue_2,\ue_3}=(1\times \kappa)_!(1\times \iota)^*[|\ue_2,\ue_3|](\frac{|\ue_2,\ue_3|}{2})$ be the functor defined by the morphisms 
$$\xymatrix{\rC_{\ue_1}\times \rC_{\ue_2}\times \rC_{\ue_3} &\rC_{\ue_1}\times\rF^{\ue_2+\ue_3}_{\ue_2,\ue_3} \ar[l]_-{1\times \kappa} \ar[r]^-{1\times \iota} &\rC_{\ue_1}\times \rC_{\ue_2+\ue_3}.}$$

\begin{proposition}\label{coassociativity}
For any $\ue_1,\ue_2,\ue_3\in \bbN[I]\times \bbN[I]$, we have
$$(\Id\times \Res^{\ue_2+\ue_3}_{\ue_2,\ue_3})\Res^{\ue_1+\ue_2+\ue_3}_{\ue_1,\ue_2+\ue_3}\cong (\Res^{\ue_1+\ue_2}_{\ue_1,\ue_2}\times \Id)\Res^{\ue_1+\ue_2+\ue_3}_{\ue_1+\ue_2,\ue_3}.$$
\end{proposition}
\begin{proof}
We denote by $^kP^j=\bigoplus_{i\in I}(e_k)_i^jP_i$ for $k=1,2,3$ and $j\in \bbZ_2$. It is clear that
$$|\ue_2,\ue_3|+|\ue_1,\ue_2+\ue_3|=|\ue_1,\ue_2|+|\ue_1+\ue_2,\ue_3|.$$ 
We only need to prove that $(1\times \kappa)_!(1\times \iota)^*\kappa_!\iota^*\cong (\kappa\times 1)_!(\iota\times 1)^*\kappa_!\iota^*$.
Consider the following diagram
$$\xymatrix{\rC_{\ue_1+\ue_2}\times \rC_{\ue_3} &\rF^{\ue_1+\ue_2+\ue_3}_{\ue_1+\ue_2,\ue_3} \ar[l]_-{\kappa} \ar[r]^-{\iota} &\rC_{\ue_1+\ue_2+\ue_3}\\
\rF^{\ue_1+\ue_2}_{\ue_1,\ue_2}\times \rC_{\ue_3} \ar[u]^-{\iota\times 1} \ar[d]_-{\kappa\times 1} &X \ar[l]_-{\hat{\kappa}} \ar[u]_-{\hat{\iota}}  \ar[d]^-{\tilde{\kappa}} \ar[r]^-{\tilde{\iota}} \ar@{}[dr]|{\square} \ar@{}[ul]|{\square} &\rF^{\ue_1+\ue_2+\ue_3}_{\ue_1,\ue_2+\ue_3}  \ar[u]_-{\iota} \ar[d]_-{\kappa} \\
\rC_{\ue_1}\times \rC_{\ue_2}\times \rC_{\ue_3} &\rC_{\ue_1}\times \rF^{\ue_2+\ue_3}_{\ue_2,\ue_3} \ar[l]_-{1\times \kappa} \ar[r]^-{1\times \iota} &\rC_{\ue_1}\times \rC_{\ue_2+\ue_3},}$$
where $X=\rF^{\ue_1+\ue_2+\ue_3}_{\ue_1,\ue_2+\ue_3}\times_{(\rC_{\ue_1}\times \rC_{\ue_2+\ue_3})}(\rC_{\ue_1}\times \rF^{\ue_2+\ue_3}_{\ue_2,\ue_3})$ is the fiber product and $\tilde{\kappa},\tilde{\iota}$ are natural morphisms. Note that $X$ is the subvariety of $\rC_{\ue_1+\ue_2+\ue_3}$ consisting of $(d^1,d^0)$ such that $({^2P^1}\bigoplus {^3P^1},{^2P^0}\bigoplus {^3P^0})$ and $({^3P^1},{^3P^0})$ are $(d^1,d^0)$-stable, and so $X$ is a subvariety of $\rF^{\ue_1+\ue_2+\ue_3}_{\ue_1+\ue_2,\ue_3}$. We denote by $\hat{\iota}:X\rightarrow \rF^{\ue_1+\ue_2+\ue_3}_{\ue_1+\ue_2,\ue_3}$ the inclusion, then for any $(d^1,d^0)\in X$, we have $\kappa\hat{\iota}(d^1,d^0)=(\rho_{1}.(d^1,d^0),\rho_{2}.(d^1,d^0))\in \rC_{\ue_1+\ue_2}\times \rC_{\ue_3}$ such that $({^2P^1},{^2P^0})$ is $\rho_{1}.(d^1,d^0)$-stable, and so the image of $\kappa\hat{\iota}$ is contained in the subvariety $\rF^{\ue_1+\ue_2}_{\ue_1,\ue_2}\times \rC_{\ue_3}$ of $\rC_{\ue_1+\ue_2}\times \rC_{\ue_3}$. We denote by $\hat{\kappa}:X\rightarrow \rF^{\ue_1+\ue_2}_{\ue_1,\ue_2}\times \rC_{\ue_3}$ the morphism such that $\kappa\hat{\iota}=(\iota\times 1)\hat{\kappa}$, then by definition, above diagram commutes and 
$$\xymatrix{X \ar[d]_-{\hat{\iota}} \ar[r]^-{\hat{\kappa}} \ar@{}[dr]|{\square} &\rF^{\ue_1+\ue_2}_{\ue_1,\ue_2}\times \rC_{\ue_3} \ar[d]^-{\iota\times 1}\\
\rF^{\ue_1+\ue_2+\ue_3}_{\ue_1+\ue_2,\ue_3} \ar[r]^-{\kappa} &\rC_{\ue_1+\ue_2}\times \rC_{\ue_3}}$$
becomes a Cartesian square. By base change, we have 
\begin{align*}
(1\times \kappa)_!(1\times \iota)^*\kappa_!\iota^*
\cong (1\times \kappa)_!\tilde{\kappa}_!\tilde{\iota}^*\iota^*
=(\kappa\times 1)_!\hat{\kappa}_!\hat{\iota}^*\iota^*
\cong (\kappa\times 1)_!(\iota\times 1)^*\kappa_!\iota^*,
\end{align*}
as desired.
\end{proof}

For $\underline{0}=(0,0)\in \bbN[I]\times \bbN[I]$, the variety $\rC_{\underline{0}}=\{(0,0)\}$ is a single point. We identify $\rC_{\ue}\times \rC_{\underline{0}}\cong \rC_{\ue}\cong\rC_{\underline{0}}\times \rC_{\ue}$, then by definition, it is clear that $\Res^{\ue}_{\ue,\underline{0}}(C)\cong C\cong \Res^{\ue}_{\underline{0},\ue}(C)$ for any $C\in \cD^b_{\rG_{\ue}}(\rC_{\ue})$ and $\ue\in \bbN[I]\times \bbN[I]$.

\section{Construction of Bridgeland's Hall algebra via functions}\label{Bridgeland's Hall algebra via functions}

Let $X$ be a $k$-variety together with a connected algebraic group $G$-action. Suppose $X,G$ have $\bbF_q$-structures, we denote by $\sigma$ their Frobenius maps, $X^{\sigma},G^{\sigma}$ the $\sigma$-fixed points sets of $X,G$ respectively, and $\tilde{\cH}_{G^{\sigma}}(X^{\sigma})$ the $\bbC$-vector space of $G^{\sigma}$-invariant $\bbC$-valued functions on $X^{\sigma}$. For any $G$-equivariant morphism $\varphi:X\rightarrow Y$ which is compatible with $\bbF_q$-structures, there are two linear maps
\begin{align*}
\varphi^*:\tilde{\cH}_{G^{\sigma}}(Y^{\sigma})&\rightarrow \tilde{\cH}_{G^{\sigma}}(X^{\sigma})  &\varphi_!:\tilde{\cH}_{G^{\sigma}}(X^{\sigma})&\rightarrow \tilde{\cH}_{G^{\sigma}}(Y^{\sigma})\\
g&\mapsto (x\mapsto g(\varphi(x))), &f&\mapsto (y\mapsto \sum_{x\in \varphi^{-1}(y)}f(x)).
\end{align*}

For any $\ue\in \bbN[I]\times \bbN[I]$, since $\mathbf{Q}$ is a Dynkin quiver, we can assume that the projective representation $P^j=\bigoplus_{i\in I}e^j_iP_i$ has a $\bbF_q$-rational structure with Frobenius map $\sigma:P^j\rightarrow P^j$, that is, an additive isomorphism satisfying $\sigma(\lambda p)=\lambda^q\sigma(p)$ for any $\lambda \in k$ and $p\in P^j$ for any $j\in \bbZ_2$. Then $\rC_{\ue}$ has a $\bbF_q$-structure with Frobenius map $\sigma:\rC_{\ue}\rightarrow \rC_{\ue}$ such that for any $(d^1,d^0)\in\rC_{\ue}$, the image $\sigma(d^1,d^0)=(\tilde{d}^1,\tilde{d}^0)$ satisfies $\tilde{d}^j(\sigma(p))=\sigma(d^j(p))$,
and $\Aut_{\cA}(P^j)$ has a $\bbF_q$-structure with Frobenius map $\sigma:\Aut_{\cA}(P^j)\rightarrow \Aut_{\cA}(P^j)$ such that for any $g\in \Aut_{\cA}(P^j)$, the image $\sigma(g)$ satisfies $\sigma(g).\sigma(p)=\sigma(g.p)$
for any $j\in \bbZ_2$. Hence $\rG_{\ue}\cong \Aut_{\cA}(P^1)\times \Aut_{\cA}(P^0)$ has a $\bbF_q$-structure with Frobenius map $\sigma:\rG_{\ue}\rightarrow \rG_{\ue}$. The $\rG_{\ue}$-action on $\rC_{\ue}$ restricts to a $\rG_{\ue}^\sigma$-action on $\rC_{\ue}^\sigma$. Moreover, for any $\ue,\ue',\ue''\in\bbN[I]\times \bbN[I]$ such that $\ue=\ue'+\ue''$, the varieties and morphisms \begin{align*}
\rC_{\ue'}\times \rC_{\ue''}\xleftarrow{p_1}\rC'\xrightarrow{p_2}\rC''\xrightarrow{p_3}\rC_{\ue},\
\rC_{\ue'}\times \rC_{\ue''}\xleftarrow{\kappa} \rF\xrightarrow{\iota}\rC_{\ue},
\end{align*}
appearing in \S \ref{Induction functor} and \S \ref{Restriction functor}, have $\bbF_q$-rational structures. Taking the $\sigma$-fixed points sets, we obtain 
\begin{align*}
\rC^{\sigma}_{\ue'}\times \rC^{\sigma}_{\ue''}\xleftarrow{p_1}\rC'^{\sigma}\xrightarrow{p_2}\rC''^{\sigma}\xrightarrow{p_3}\rC^{\sigma}_{\ue},\
\rC^{\sigma}_{\ue'}\times \rC^{\sigma}_{\ue''}\xleftarrow{\kappa}\rF^{\sigma}\xrightarrow{\iota}\rC^{\sigma}_{\ue},
\end{align*}
where $p_1,p_2,p_3,\kappa,\iota$ are restrictions of the original morphisms denoted by the same notations. 

Since $\rC_{\ue'}^\sigma,\rC_{\ue''}^\sigma$ are finite sets, there is an isomorphism of $\bbC$-vector spaces
\begin{align*}
\tilde{\cH}_{\rG^{\sigma}_{\ue'}}(\rC^{\sigma}_{\ue'})\otimes \tilde{\cH}_{\rG^{\sigma}_{\ue''}}(\rC^{\sigma}_{\ue''})&\xrightarrow{\cong}\tilde{\cH}_{\rG^{\sigma}_{\ue'}\times \rG^{\sigma}_{\ue''}}(\rC^{\sigma}_{\ue'}\times \rC^{\sigma}_{\ue''})\\
f\otimes g&\mapsto ((x',x'')\mapsto f(x')g(x'')).
\end{align*}
For any $f\in \tilde{\cH}_{\rG^{\sigma}_{\ue'}}(\rC^{\sigma}_{\ue'})$ and $g\in \tilde{\cH}_{\rG^{\sigma}_{\ue''}}(\rC^{\sigma}_{\ue''})$, we have 
\begin{align*}
f\otimes g\in &\tilde{\cH}_{\rG^{\sigma}_{\ue'}}(\rC^{\sigma}_{\ue'})\otimes \tilde{\cH}_{\rG^{\sigma}_{\ue''}}(\rC^{\sigma}_{\ue''})\cong\tilde{\cH}_{\rG^{\sigma}_{\ue'}\times \rG^{\sigma}_{\ue''}}(\rC^{\sigma}_{\ue'}\times \rC^{\sigma}_{\ue''})\\
=&\tilde{\cH}_{\rG^{\sigma}_{\ue'}\times \rG^{\sigma}_{\ue''}\times \rG^{\sigma}_{\ue}}(\rC^{\sigma}_{\ue'}\times \rC^{\sigma}_{\ue''})
\end{align*}
and $(p_1)^*(f\otimes g)\in \tilde{\cH}_{\rG^{\sigma}_{\ue'}\times \rG^{\sigma}_{\ue''}\times \rG^{\sigma}_{\ue}}(\rC'^{\sigma})$. Note that $p_2:\rC'^{\sigma}\rightarrow \rC''^{\sigma}$ is a principal $\rG^{\sigma}_{\ue'}\times \rG^{\sigma}_{\ue''}$-bundle, and $(p_1)^*(f\otimes g)$ is constant on $\rG^{\sigma}_{\ue'}\times \rG^{\sigma}_{\ue''}$-orbits, there exists a unique $h\in \tilde{\cH}_{\rG^{\sigma}_{\ue}}(\rC''^{\sigma})$ such that $(p_2)^*(h)=(p_1)^*(f\otimes g)$.
Indeed, we have 
$$h=\frac{1}{|\rG^{\sigma}_{\ue'}\times \rG^{\sigma}_{\ue''}|}(p_2)_!(p_1)^*(f\otimes g).$$
Finally, we can form $(p_3)_!(h)\in \tilde{\cH}_{\rG^{\sigma}_{\ue}}(\rC^{\sigma}_{\ue})$.
Similarly, for any $h\in \tilde{\cH}_{\rG^\sigma_\ue}(\rC^\sigma_{\ue})$, we can form $\kappa_!\iota^*(h)\in  \tilde{\cH}_{\rG^{\sigma}_{\ue'}\times \rG^{\sigma}_{\ue''}}(\rC^{\sigma}_{\ue'}\times \rC^{\sigma}_{\ue''})\cong \tilde{\cH}_{\rG^{\sigma}_{\ue'}}(\rC^{\sigma}_{\ue'})\otimes \tilde{\cH}_{\rG^{\sigma}_{\ue''}}(\rC^{\sigma}_{\ue''})$.

Let $v_q\in \bbC$ be the fixed square root of $q$, which is the same as one in the definition of Ringel-Hall algebra, see \S \ref{Hall algebra for abelian category}.

\begin{definition}
Let  $\ue,\ue',\ue''\in\bbN[I]\times \bbN[I]$ be such that $\ue=\ue'+\ue''$. We define\\
(a) the induction linear map 
\begin{align*}
\ind^{\ue}_{\ue',\ue''}:\tilde{\cH}_{\rG^{\sigma}_{\ue'}}(\rC^{\sigma}_{\ue'})\otimes \tilde{\cH}_{\rG^{\sigma}_{\ue''}}(\rC^{\sigma}_{\ue''})&\rightarrow \tilde{\cH}_{\rG^{\sigma}_{\ue}}(\rC^{\sigma}_{\ue})\\
f\otimes g&\mapsto \frac{v_q^{|\ue',\ue''|}}{|\rG^{\sigma}_{\ue'}\times \rG^{\sigma}_{\ue''}|}(p_3)_!(p_2)_!(p_1)^*(f\otimes g);
\end{align*}
(b) the restriction linear map
\begin{align*}
\res^{\ue}_{\ue',\ue''}:\tilde{\cH}_{\rG^{\sigma}_{\ue}}(\rC^{\sigma}_{\ue})&\rightarrow \tilde{\cH}_{\rG^{\sigma}_{\ue'}}(\rC^{\sigma}_{\ue'})\otimes \tilde{\cH}_{\rG^{\sigma}_{\ue''}}(\rC^{\sigma}_{\ue''})\\
h&\mapsto v_q^{-|\ue',\ue''|}\kappa_!\iota^*(h);
\end{align*}
(c) the direct sum and its graded dual
\begin{align*}
&\tilde{\cH}^{\mathrm{tw}}(\cC_2(\cP_q))=\bigoplus_{\ue}\tilde{\cH}_{\rG^{\sigma}_{\ue}}(\rC^{\sigma}_{\ue}),\\ \tilde{\cH}^{\mathrm{tw},*}(\cC_2(\cP_q))=&\bigoplus_{\ue}\tilde{\cH}^*_{\rG^{\sigma}_{\ue}}(\rC^{\sigma}_{\ue})=\bigoplus_{\ue}\Hom_{\bbC}(\tilde{\cH}_{\rG^{\sigma}_{\ue}}(\rC^{\sigma}_{\ue}),\bbC).
\end{align*}
\end{definition}

By similar argument to Proposition \ref{associativity} and \ref{coassociativity}, we have the following corollary. 

\begin{corollary}
We have\\
(a) all induction maps $\ind^{\ue}_{\ue',\ue''}$ for $\ue=\ue'+\ue''$ define an  associative multiplication 
$$*:\tilde{\cH}^{\mathrm{tw}}(\cC_2(\cP_q))\otimes \tilde{\cH}^{\mathrm{tw}}(\cC_2(\cP_q))\rightarrow \tilde{\cH}^{\mathrm{tw}}(\cC_2(\cP_q));$$ 
(b) all restriction maps $\res^{\ue}_{\ue',\ue''}$ for $\ue=\ue'+\ue''$ define a coassociative comultiplication 
$$r:\tilde{\cH}^{\mathrm{tw}}(\cC_2(\cP_q))\rightarrow \tilde{\cH}^{\mathrm{tw}}(\cC_2(\cP_q))\otimes\tilde{\cH}^{\mathrm{tw}}(\cC_2(\cP_q));$$ 
Dually, the map $r$ induces an associative multiplication
$$*_r:\tilde{\cH}^{\mathrm{tw},*}(\cC_2(\cP_q))\otimes \tilde{\cH}^{\mathrm{tw},*}(\cC_2(\cP_q))\rightarrow \tilde{\cH}^{\mathrm{tw},*}(\cC_2(\cP_q)).$$
\end{corollary}

Recall that $\cC_2(\cP_{q})$ is the category of two-periodic projective complexes of $\cA_q$ over $\bbF_q$ in \S \ref{Hall algebra for the category of two-periodic complexes}.  Any point $(d^1,d^0)\in \rC_{\ue}^\sigma$ determines an object $M_\bullet\in \cC_2(\cP_{q})$ satisfying $\ue_{M_\bullet}=\ue$. In this case, we write $M_\bullet\in \rC_{\ue}^\sigma$, and denote by $\cO_{M_{\bullet}}\subset \rC_{\ue}^\sigma$ the corresponding $\rG_{\ue}^\sigma$-orbit and  $1_{\cO_{M_{\bullet}}}\in \tilde{\cH}_{\rG^{\sigma}_{\ue}}(\rC^{\sigma}_{\ue})$ the corresponding characteristic function. 

There is a bijection between the set of $\rG_{\ue}^\sigma$-orbits in $\rC_{\ue}^\sigma$ and the set of isomorphism classes of objects in $\cC_2(\cP_{q})$ of projective dimension vector pair $\ue$, that is, $M_\bullet,M'_\bullet$ are isomorphic as objects in $\cC_2(\cP_{q})$ if and only if they lie in the same $\rG_{\ue}^\sigma$-orbit.

\begin{proposition}\label{key proposition}
Let  $\ue,\ue',\ue''\in\bbN[I]\times \bbN[I]$ be such that $\ue=\ue'+\ue''$. Then for any $M_\bullet\in \rC_{\ue'}^{\sigma}, N_\bullet\in \rC_{\ue''}^{\sigma}, L_\bullet\in \rC_{\ue}^{\sigma}$, we have 
\begin{align}
&\ind^{\ue}_{\ue',\ue''}(1_{\cO_{M_{\bullet}}}\otimes 1_{\cO_{N_{\bullet}}})(L_{\bullet})=v_q^{|\ue',\ue''|}g^{L_\bullet}_{M_\bullet,N_\bullet}, \label{formula 1}\\
&\res^{\ue}_{\ue',\ue''}(1_{\cO_{L_{\bullet}}})(M_{\bullet},N_{\bullet})=v_q^{|\ue',\ue''|}\frac{|\Ext^1_{\cC_2(\cP_{q})}(M_\bullet, N_\bullet)_{L_\bullet}|}{|\Hom_{\cC_2(\cP_{q})}(M_\bullet, N_\bullet)|}, \label{formula 2}
\end{align}
where $g^{L_\bullet}_{M_\bullet,N_\bullet}$ and $\Ext^1_{\cC_2(\cP_{q})}(M_\bullet, N_\bullet)_{L_\bullet}$ are defined in \S \ref{Hall algebra for abelian category}. 
\end{proposition}
\begin{proof}
By definition, we have 
\begin{align*}
&\ind^{\ue}_{\ue',\ue''}(1_{\cO_{M_{\bullet}}}\otimes 1_{\cO_{N_{\bullet}}})(L_{\bullet})
=\frac{v_q^{|\ue',\ue''|}}{|\rG_{\ue'}^{\sigma}\times \rG_{\ue''}^{\sigma}|}(p_3)_!(p_2)_!(p_1)^*(1_{\cO_{M_{\bullet}}}\otimes 1_{\cO_{N_{\bullet}}})(L_{\bullet})\\
=&\frac{v_q^{|\ue',\ue''|}}{|\rG_{\ue'}^{\sigma}\times \rG_{\ue''}^{\sigma}|}\sum_{x\in (p_3p_2)^{-1}(L_{\bullet})}(p_1)^*(1_{\cO_{M_{\bullet}}}\otimes 1_{\cO_{N_{\bullet}}})(x)
=\frac{v_q^{|\ue',\ue''|}}{|\rG_{\ue'}^{\sigma}\times \rG_{\ue''}^{\sigma}|}|S|,
\end{align*}
where $S$ is a subset of $\rC'^\sigma$ consisting of $(d^1,d^0,W^1,W^0,\rho_1^1,\rho_1^0,\rho_2^1,\rho_2^0)$ such that $(d^1,d^0)=L_{\bullet}$ and $\rho_{1}.(d^1,d^0)\in \cO_{M_{\bullet}},\rho_{2}.(d^1,d^0)\in \cO_{N_{\bullet}}$. We denote by $S'$ the set of subobjects $L'_\bullet$ of $L_\bullet$ satisfying $L_\bullet/L'_\bullet\cong M_\bullet,L'_\bullet\cong N_\bullet$, so $|S'|=g^{L_\bullet}_{M_\bullet,N_\bullet}$. It is clear that $S\not=\varnothing$ if and only if $S'\not=\varnothing$. If $S=S'=\varnothing$, then both two sides in the formula (\ref{formula 1}) are zero. Otherwise, note that for any $L'_\bullet\in S'$, since $L^j/L'^j\cong M^j$ is projective, the canonical short exact sequence $0\rightarrow L'^j\rightarrow L^j\rightarrow L^j/L'^j\rightarrow 0$ must split, thus $L'^j$ is a direct summand of $L^j$ for any $j\in \bbZ_2$. Hence
$$(d^1,d^0,W^1,W^0,\rho_1^1,\rho_1^0,\rho_2^1,\rho_2^0) \longmapsto \xymatrix@C=1.5cm{(W^1 \ar@<.5ex>[r]^-{d^1_{W^1}} &W^0) \ar@<.5ex>[l]^-{d^0_{W^0}}}$$
defines a surjection $S\rightarrow S'$ whose fiber is isomorphic to $\rG_{\ue'}^\sigma\times \rG_{\ue''}^\sigma$, and so 
\begin{align*}
\ind^{\ue}_{\ue',\ue''}(1_{\cO_{M_{\bullet}}}\otimes 1_{\cO_{N_{\bullet}}})(L_{\bullet})=&\frac{v_q^{|\ue',\ue''|}}{|\rG_{\ue'}^{\sigma}\times \rG_{\ue''}^{\sigma}|}|\rG_{\ue'}^{\sigma}\times \rG_{\ue''}^{\sigma}|g^{L_\bullet}_{M_\bullet,N_\bullet}
=v_q^{|\ue',\ue''|}g^{L_\bullet}_{M_\bullet,N_\bullet}.
\end{align*}

By definition, we have 
\begin{align*}
\res^{\ue}_{\ue',\ue''}(1_{\cO_{L_{\bullet}}})(M_{\bullet},N_{\bullet})=&v_q^{-|\ue',\ue''|}\kappa_!\iota^*(1_{\cO_{L_{\bullet}}})(M_{\bullet},N_{\bullet})\\
=&v_q^{-|\ue',\ue''|}\sum_{L'_{\bullet}\in \kappa^{-1}(M_{\bullet},N_{\bullet})}\iota^*(1_{\cO_{L_{\bullet}}})(L'_{\bullet})\\
=&v_q^{-|\ue',\ue''|}|\{L'_{\bullet}\in \cO_{L_\bullet}\cap \rF^{\sigma}|\kappa(L'_{\bullet})=(M_{\bullet},N_{\bullet})\}|.
\end{align*}
Note that there exists $L'_{\bullet}\in \cO_{L_\bullet}\cap \rF^{\sigma}$ such that $\kappa(L'_{\bullet})=(M_{\bullet},N_{\bullet})$, if and only if $L'_{\bullet} \cong L_{\bullet}$ and $N_\bullet$ is a subobject of $L'_{\bullet}$ such that $L'_{\bullet}/N_{\bullet}=M_{\bullet}$, if and only if $\Ext^1_{\cC_2(\cP_q)}(M_{\bullet},N_{\bullet})_{L_{\bullet}}\not=\varnothing$. If $\Ext^1_{\cC_2(\cP_q)}(M_{\bullet},N_{\bullet})_{L_{\bullet}}=\varnothing$, then both two sides in the formula (\ref{formula 2}) are zero. Otherwise, by Lemma \ref{fibre of kappa}, the fiber of $\kappa$ at $(M_{\bullet},N_{\bullet})$ is isomorphic to $\Hom_{\cC_2(\cP_q)}(M_\bullet,N_\bullet^*)$, and the element $(f^1,f^0)\in \Hom_{\cC_2(\cP_q)}(M_\bullet,N_\bullet^*)$ corresponds to $\textrm{Cone}(f^1,f^0)^*$, see Remark \ref{mapping cone}. Thus there is a bijection between $\{L'_{\bullet}\in \cO_{L_\bullet}\cap \rF^{\sigma}\mid \kappa(L'_{\bullet})=(M_{\bullet},N_{\bullet})\}$ and 
\begin{align*}
\Hom_{\cC_2(\cP_q)}(M_\bullet,N_\bullet^*)_{L_\bullet^*}=
\{(f^1,f^0)\in \Hom_{\cC_2(\cP_q)}(M_\bullet,N_\bullet^*)\mid \textrm{Cone}(f^1,f^0)\cong L_{\bullet}^*\}.
\end{align*}
Consider the following natural short exact sequence 
$$0\rightarrow \textrm{Htp}(M_\bullet,N_\bullet^*)\rightarrow \Hom_{\cC_2(\cP_q)}(M_\bullet,N_\bullet^*)\xrightarrow{\pi} \Hom_{\cK_2(\cP_q)}(M_\bullet,N_\bullet^*)\rightarrow 0,$$
notice that $\Hom_{\cC_2(\cP_q)}(M_\bullet,N_\bullet^*)_{L_\bullet^*}=\pi^{-1}(\Hom_{\cK_2(\cP_q)}(M_\bullet,N_\bullet^*)_{L^*_\bullet})$, where the subset $\Hom_{\cK_2(\cP_q)}(M_\bullet,N_\bullet^*)_{L^*_\bullet}\subset \Hom_{\cK_2(\cP_q)}(M_\bullet,N_\bullet^*)$ is defined in Lemma \ref{bijection between Ext and Hom}, thus we have 
$$0\rightarrow \textrm{Htp}(M_\bullet,N_\bullet^*)\rightarrow \Hom_{\cC_2(\cP_q)}(M_\bullet,N_\bullet^*)_{L^*_\bullet}\rightarrow \Hom_{\cK_2(\cP_q)}(M_\bullet,N_\bullet^*)_{L^*_\bullet}\rightarrow 0.$$
By above short exact sequence and Lemma \ref{bijection between Ext and Hom}, we obtain
\begin{align*}
\res^{\ue}_{\ue',\ue''}(1_{\cO_{L_{\bullet}}})(M_{\bullet},N_{\bullet})=&v_q^{-|\ue',\ue''|}|\Hom_{\cC_2(\cP_q)}(M_\bullet,N_\bullet^*)_{L_\bullet^*}|\\
=&v_q^{-|\ue',\ue''|}|\textrm{Htp}(M_\bullet,N_\bullet^*)||\Hom_{\cK_2(\cP_q)}(M_\bullet,N_\bullet^*)_{L_\bullet^*}|\\
=&v_q^{-|\ue',\ue''|}|\textrm{Htp}(M_\bullet,N_\bullet^*)||\Ext^1_{\cC_2(\cP_q)}(M_{\bullet},N_{\bullet})_{L_{\bullet}}|.
\end{align*}
By the definition of $\textrm{Htp}(M_\bullet,N_\bullet^*)$, there is a surjection 
\begin{align*}
\Hom&_{\cA_{\bbF_q}}(M^1,N^1)\times \Hom_{\cA_{\bbF_q}}(M^0,N^0)\rightarrow \textrm{Htp}(M_\bullet,N_\bullet^*)\\
&(s^1,s^0)\mapsto (-d_N^1s^1+s^0d_M^1,-d_N^0s^0+s^1d_M^0)
\end{align*}
whose kernel is $\Hom_{\cC_2(\cP_q)}(M_\bullet,N_\bullet)$. Therefore,
\begin{align*}
&\res^{\ue}_{\ue',\ue''}(1_{\cO_{L_{\bullet}}})(M_{\bullet},N_{\bullet})\\=&v_q^{-|\ue',\ue''|}\frac{|\Hom_{\cA_{\bbF_q}}(M^1,N^1)\times \Hom_{\cA_{\bbF_q}}(M^0,N^0)|}{|\Hom_{\cC_2(\cP_q)}(M_\bullet,N_\bullet)|}|\Ext^1_{\cC_2(\cP_q)}(M_\bullet,N_\bullet)_{L_\bullet}|\\
=&v_q^{-|\ue',\ue''|+2|\ue',\ue''|}\frac{|\Ext^1_{\cC_2(\cP_q)}(M_\bullet,N_\bullet)_{L_\bullet}|}{|\Hom_{\cC_2(\cP_q)}(M_\bullet,N_\bullet)|}
=v_q^{|\ue',\ue''|}\frac{|\Ext^1_{\cC_2(\cP_q)}(M_\bullet,N_\bullet)_{L_\bullet}|}{|\Hom_{\cC_2(\cP_q)}(M_\bullet,N_\bullet)|},
\end{align*}
as desired.
\end{proof}

\begin{corollary}\label{function-isomorphism}
There are algebra isomorphisms 
\begin{align*}
\tilde{\cH}^{\mathrm{tw}}(\cC_2(\cP_q))&\xrightarrow{\cong} \cH^{\mathrm{tw}}(\cC_2(\cP_q)) &\tilde{\cH}^{\mathrm{tw},*}(\cC_2(\cP_q))&\xrightarrow{\cong} \cH^{\mathrm{tw}}(\cC_2(\cP_q)) \\
1_{\cO_{M_{\bullet}}}&\mapsto u_{[M_\bullet]}, &1_{\cO_{M_{\bullet}}}^*&\mapsto (a_{M_\bullet}u_{[M_\bullet]}),
\end{align*}
where $\{1_{\cO_{M_{\bullet}}}^*\mid \cO_{M_\bullet}\subset\rC_{\ue}^\sigma\}$ is the dual basis of $\{1_{\cO_{M_{\bullet}}}\mid \cO_{M_\bullet}\subset\rC_{\ue}^\sigma\}$. Hence 
\begin{align*}
\{a_{K_P}1_{\cO_{K_P}},a_{K_P^*}1_{\cO_{K_P^*}}\mid P\in \cP_q\}&\subset \tilde{\cH}^{\mathrm{tw}}(\cC_2(\cP_q)),\\
\{1_{\cO_{K_P}}^*,1_{\cO_{K_P^*}}^*\mid P\in \cP_q\}&\subset \tilde{\cH}^{\mathrm{tw},*}(\cC_2(\cP_q))
\end{align*}
which are the inverse images of $\{b_{K_P},b_{K_P^*}\mid P\in \cP_q\}\subset \cH^{\mathrm{tw}}(\cC_2(\cP_q))$, also satisfy the Ore conditions, and so there are well-defined localizations
\begin{align*} 
\cD\tilde{\cH}(\cA_q)&=\tilde{\cH}^{\mathrm{tw}}(\cC_2(\cP_q))[(a_{K_P}1_{\cO_{K_P}})^{-1},(a_{K_P^*}1_{\cO_{K_P^*}})^{-1}\mid P\in \cP_q]\cong \cD\cH(\cA_q),\\
\cD\tilde{\cH}^*(\cA_q)&=\tilde{\cH}^{\mathrm{tw},*}(\cC_2(\cP_q))[(1_{\cO_{K_P}}^*)^{-1},(1_{\cO_{K_P^*}}^*)^{-1}\mid P\in \cP_q]\cong \cD\cH(\cA_q)
\end{align*}
with the reduced quotients
\begin{align*}
\cD\tilde{\cH}^{\red}(\cA_q)&=\cD\tilde{\cH}(\cA_q)/\langle (a_{K_P}1_{\cO_{K_P}})*(a_{K_P^*}1_{\cO_{K_P^*}})-1\mid P\in \cP_q\rangle\cong\cD\cH^{\red}(\cA_q),\\
\cD\tilde{\cH}^{*,\red}(\cA_q)&=\cD\tilde{\cH}^*(\cA_q)/\langle 1_{\cO_{K_P}}^**_r1_{\cO_{K_P^*}}^*-1\mid P\in \cP_q \rangle \cong \cD\cH^{\red}(\cA_q)
\end{align*}
which are isomorphic to Bridgeland's Hall algebra.
\end{corollary}
\begin{proof}
We prove that $1_{\cO_{M_{\bullet}}}^*\mapsto (a_{M_\bullet}u_{[M_\bullet]})$ defines an algebra homomorphism. For any $\cO_{M_\bullet},\cO_{N_\bullet}$ and $\cO_{L_\bullet}$, by definition and Proposition \ref{key proposition}, we have 
\begin{align*}
&1_{\cO_{M_\bullet}}^**_r1_{\cO_{N_\bullet}}^*(1_{\cO_{L_\bullet}})=(1_{\cO_{M_\bullet}}^*\otimes1_{\cO_{N_\bullet}}^*)(r(1_{\cO_{L_\bullet}}))\\
=&r(1_{\cO_{L_\bullet}})(M_\bullet,N_\bullet)=v_q^{\langle\hat{M^1},\hat{N^1}\rangle+\langle\hat{M^0},\hat{N^0}\rangle}\frac{|\Ext^1_{\cC_2(\cP_q)}(M_\bullet, N_\bullet)_{L_\bullet}|}{|\Hom_{\cC_2(\cP_q)}(M_\bullet, N_\bullet)|},
\end{align*}
and so
\begin{align*}
&1_{\cO_{M_\bullet}}^**_r1_{\cO_{N_\bullet}}^*=\sum_{\cO_{L_\bullet}}(1_{\cO_{M_\bullet}}^**_r1_{\cO_{N_\bullet}}^*(1_{\cO_{L_\bullet}}))1_{\cO_{L_\bullet}}^*\\
=&\sum_{\cO_{L_\bullet}}v_q^{\langle\hat{M^1},\hat{N^1}\rangle+\langle\hat{M^0},\hat{N^0}\rangle}\frac{|\Ext^1_{\cC_2(\cP_q)}(M_\bullet, N_\bullet)_{L_\bullet}|}{|\Hom_{\cC_2(\cP_q)}(M_\bullet, N_\bullet)|}1_{\cO_{L_\bullet}}^*,
\end{align*}
as desired. The other statements are clear.
\end{proof}

Replacing the Frobenius morphism $\sigma$ by its power $\sigma^n$ for $n\geqslant 1$, we can obtain similar results over the finite field $\bbF_{q^n}$. This completes a realization of Bridgeland's Hall algebra via functions.

\section{Constructible sheaves realization of Bridgeland's Hall algebra}\label{Constructible sheaves realization of Bridgeland's Hall algebra}

\subsection{Sheaf-function correspondence}\label{Sheaf-function correspondence}

In this subsection, we first review the sheaf-function correspondence theorem and refer to \cite[\S 5.3]{Pramod-2021} for details. 

We fix an isomorphism $\overline{\bbQ}_l\cong \bbC$. Recall that the category $\cD^b_{G,m}(X)$ consists of mixed objects of the form $(L,\varphi)$, where $L$ is a $G$-equivariant complex and $\varphi:\sigma^*(L)\rightarrow L$ is an isomorphism, where $\sigma:X\rightarrow X$ is the Frobenius morphism. For any $\sigma$-fixed point $x\in X^\sigma$, the isomorphism $\varphi$ induces an isomorphism $\varphi_x:L_x= (\sigma^*(L))_x\rightarrow L_x$, and then for any $s\in \bbZ$, it induces an isomorphism $H^s(\varphi_x):H^s(L_x)\rightarrow H^s(L_x)$ between $\overline{\bbQ}_l$-vector spaces, where $H^s(-)$ is the cohomology functor of complexes. Taking the alternating sum of the traces of these isomorphisms, we obtain 
$$\chi_L(x)=\sum_{s\in \bbZ}(-1)^s\tr (H^s(\varphi_x))\in \overline{\bbQ}_l\cong \bbC.$$
In this way, we obtain a $G^\sigma$-invariant function $\chi_L:X^{\sigma}\rightarrow \bbC$.

\begin{theorem}[{ \cite[\S 5.3]{Pramod-2021}}]\label{sheaf-function correspondence}
Let $X,Y$ be $k$-varieties together with connected algebraic group $G$-actions, suppose they are defined over $\bbF_q$, and $f:X\rightarrow Y$ is a $G$-equivariant morphism which is compatible with $\bbF_q$-structures. Then\\
(a) for any distinguished triangle $L'\rightarrow L\rightarrow L''\rightarrow L'[1]$ in $\cD^b_{G,m}(X)$, we have $\chi_L=\chi_{L'}+\chi_{L''}$;\\
(b) for any $L\in \cD^b_{G,m}(X),L'\in \cD^b_{G,m}(Y)$ and $n\in \bbZ$, we have 
\begin{align*}
&\chi_{L[n]}=(-1)^n\chi_L,\ \chi_{L(\frac{n}{2})}=\sqrt{q}^{-n}\chi_L,\ \chi_{L\boxtimes L'}=\chi_L\otimes \chi_{L'},\\
&\chi_{f_!(L)}=f_!(\chi_L),\ \chi_{f^*(L')}=f^*(\chi_{L'}).
\end{align*}
\end{theorem}
The trace map is $\chi_{\_}:K_0(\cD^b_{G,m}(X))\rightarrow \tilde{\cH}_{G^\sigma}(X^\sigma)$.

Next we relate \S \ref{Induction functor}, \S \ref{Restriction functor} and \S \ref{Bridgeland's Hall algebra via functions} via sheaf-function correspondence. We take the fixed square root $v_q$ in \S \ref{Bridgeland's Hall algebra via functions} to be $-\sqrt{q}$. 

Since Grothendieck six operators sends mixed complexes to mixed complexes, for any $\ue,\ue',\ue''\in\bbN[I]\times \bbN[I]$ such that $\ue=\ue'+\ue''$, the induction functor $\Ind^{\ue}_{\ue',\ue''}$ and the restriction functor $\Res^{\ue}_{\ue',\ue''}$ can be restricted to
\begin{align*}
&\Ind^{\ue}_{\ue',\ue''}:\cD^b_{\rG_{\ue',m}}(\rC_{\ue'})\boxtimes \cD^b_{\rG_{\ue'',m}}(\rC_{\ue''})\rightarrow \cD^b_{\rG_{\ue,m}}(\rC_{\ue}),\\
&\Res^{\ue}_{\ue',\ue''}:\cD^b_{\rG_{\ue,m}}(\rC_{\ue})\rightarrow \cD^b_{\rG_{\ue'}\times \rG_{\ue''},m}(\rC_{\ue'}\times \rC_{\ue''}).
\end{align*}

\begin{proposition}\label{chi}
Let $\ue,\ue',\ue''\in\bbN[I]\times \bbN[I]$ be such that $\ue=\ue'+\ue''$. Then for any $A\in \cD^b_{\rG_{\ue',m}}(\rC_{\ue'}), B\in \cD^b_{\rG_{\ue'',m}}(\rC_{\ue''})$ and $C\in \cD^b_{\rG_{\ue,m}}(\rC_{\ue})$, we have 
\begin{align*}
\chi_{\Ind^{\ue}_{\ue',\ue''}(A\boxtimes B)}&=\ind^{\ue}_{\ue',\ue''}(\chi_A\otimes \chi_B),\\
\chi_{\Res^{\ue}_{\ue',\ue''}(C)}&=\res^{\ue}_{\ue',\ue''}(\chi_C).
\end{align*}
\end{proposition}
\begin{proof}
We apply Theorem \ref{sheaf-function correspondence}. For any $L\in \cD^b_{\rG_{\ue'}\times \rG_{\ue''}\times \rG_{\ue},m}(\rC')$, the function $\chi_{(p_2)_\flat(L)}$ is the unique function such that $(p_2)^*(\chi_{(p_2)_\flat(L)})=\chi_{(p_2)^*(p_2)_\flat(L)}=\chi_L$,
since $(p_2)_\flat$ is the quasi-inverse of $(p_2)^*$. Notice that the function $h=\frac{1}{|\rG_{\ue'}^\sigma\times \rG_{\ue''}^\sigma|}(p_2)_!(\chi_L)$ also satisfies $(p_2)^*(h)=\chi_L$, and so $\chi_{(p_2)_\flat(L)}=\frac{1}{|\rG_{\ue'}^\sigma\times \rG_{\ue''}^\sigma|}(p_2)_!(\chi_L)$. Hence, we have
\begin{align*}
\chi_{\Ind^{\ue}_{\ue',\ue''}(A\boxtimes B)}=&\chi_{(p_3)_!(p_2)_\flat(p_1)^*(A\boxtimes B)[-|\ue',\ue''|](-\frac{|\ue',\ue''|}{2})}\\
=&v_q^{|\ue',\ue''|}(p_3)_!(\frac{1}{|\rG_{\ue'}^\sigma\times \rG_{\ue''}^\sigma|}(p_2)_!(p_1)^*(\chi_A\otimes \chi_B))\\
=&\ind^{\ue}_{\ue',\ue''}(\chi_A\otimes \chi_B),
\end{align*}
as desired. The second formula is clear.
\end{proof}

\subsection{Grothendieck group and localization}\label{Grothendieck group and localization}

Let $\ue\in \bbN[I]\times \bbN[I]$. We define $\cK_{\ue}$ to be the Grothendieck group of $\cD^b_{\rG_{\ue},m}(\rC_{\ue})$ and the direct sum $\cK=\bigoplus_{\ue\in \bbN[I]\times \bbN[I]}\cK_{\ue}$.

For any $\rG_{\ue}$-orbit $\cO_{M_{\bullet}}\subset \rC_{\ue}$, let $j_{M_\bullet}:\cO_{M_{\bullet}}\rightarrow \rC_{\ue}$ be the inclusion. We define $S_{M_{\bullet}}, I_{M_{\bullet}}\in \cK_{\ue}$ to be the images of the standard sheaf and intersection cohomology complex with Tate twists
\begin{align*}
&(j_{M_\bullet})_!(\overline{\bbQ}_l|_{\cO_{M_{\bullet}}})[\dim \cO_{M_\bullet}](\frac{\dim \cO_{M_\bullet}}{2}),\\
\IC_{M_\bullet}=\IC(\cO_{M_{\bullet}},&\overline{\bbQ}_l)(\frac{\dim \cO_{M_{\bullet}}}{2})=(j_{M_\bullet})_{!*}(\overline{\bbQ}_l|_{\cO_{M_{\bullet}}}[\dim \cO_{M_{\bullet}}])(\frac{\dim \cO_{M_{\bullet}}}{2})
\end{align*}
respectively. Then we define the sets
\begin{align*}
\cS_{\ue}&=\{S_{M_{\bullet}}\mid \cO_{M_{\bullet}}\subset \rC_{\ue}\},\ \cS=\bigsqcup_{\ue\in \bbN[I]\times \bbN[I]}\cS_{\ue},\\
\cI_{\ue}&=\{I_{M_{\bullet}}\mid \cO_{M_{\bullet}}\subset \rC_{\ue}\},\ 
\cI=\bigsqcup_{\ue\in \bbN[I]\times \bbN[I]}\cI_{\ue}.
\end{align*}

There is a partial order $\preccurlyeq$ on the set of $\rG_{\ue}$-orbits in $\rC_{\ue}$ by degeneration. More precisely, $\cO_{N_{\bullet}}\preccurlyeq\cO_{M_{\bullet}}$ if and only if $\cO_{N_{\bullet}}\subset \overline{\cO_{M_{\bullet}}}$. 

\begin{proposition}
Let $\bbZ[\overline{\bbQ}_l^*]$ be the group ring of $\overline{\bbQ}_l^*$ over $\bbZ$. Then for any $\ue\in \bbN[I]\times \bbN[I]$,\\
(a) the Grothendieck group $\cK_{\ue}$ is a free $\bbZ[\overline{\bbQ}_l^*]$-module with two basis $\cS_{\ue}$ and $\cI_{\ue}$;\\
(b) the transition matrix between $\cI_{\ue}$ and $\cS_{\ue}$ is upper triangular (with respect to the degeneration order $\preccurlyeq$) whose diagonal entries are $1$.
\end{proposition}
\begin{proof}
It follows from the same argument in \cite[\S 9.4]{Lusztig-1990}.
\end{proof}

Both $\cS$ and $\cI$ are $\bbZ[\overline{\bbQ}_l^*]$-bases of $\cK$. Note that these two bases are parametrized by orbits, and so parametrized by isomorphism classes of objects in $\cC_2(\cP)$.

For any $P\in \cP$, we define 
\begin{align*}
B_{K_P}&=(j_{K_P})_!(f_{K_P})_!(\overline{\bbQ}_l|_{\rG_{\ue_{K_P}}})\in \cD^b_{\rG_{\ue_{K_P}},m}(\rC_{\ue_{K_P}}),\\
B_{K_P^*}&=(j_{K_P^*})_!(f_{K_P^*})_!(\overline{\bbQ}_l|_{\rG_{\ue_{K_P^*}}})\in \cD^b_{\rG_{\ue_{K_P^*}},m}(\rC_{\ue_{K_P^*}})
\end{align*}
associated to $K_P=(1,0)\in \rC_{\ue_{K_P}},K_P=(0,1)\in \rC_{\ue_{K_P^*}}$ respectively, where 
\begin{align*}
f_{K_P}:\rG_{\ue_{K_P}}&\rightarrow \cO_{K_P} &f_{K_P^*}:\rG_{\ue_{K_P^*}}&\rightarrow \cO_{K_P^*}\\
(g^1,g^0)&\mapsto (g^1,g^0).(1,0), &(g^1,g^0)&\mapsto (g^1,g^0).(0,1),
\end{align*}
and $f_{K_P}$ is a principal $\Aut_{\cC_2(\cP)}(K_P)$-bundle, $f_{K_P^*}$ is a principal $\Aut_{\cC_2(\cP)}(K_P^*)$-bundle.

\begin{theorem}\label{induction algebra}
All induction functors $\Ind^{\ue}_{\ue',\ue''}$ for $\ue=\ue'+\ue''$ induce a multiplication 
$$*:\cK\otimes_{\bbZ[\overline{\bbQ}_l^*]}\cK\rightarrow \cK$$
such that $\cK$ is an associative algebra, and the trace map induces an algebra isomorphism
$$\chi:\bbC\otimes_{\bbZ[\overline{\bbQ}_l^*]} \cK\rightarrow \tilde{\cH}^{\tw}(\cC_2(\cP_q)),$$
where $\bbC$ is viewed as a $\bbZ[\overline{\bbQ}_l^*]$-module via the fixed isomorphism $\bbC\cong \overline{\bbQ}_l$. Moreover, the subset $\{[B_{K_P}],[B_{K_P^*}]\mid P\in \cP\}\subset \cK$ satisfies the Ore conditions, and so there is a well-defined localization 
$$\cD\cK=\cK[[B_{K_P}]^{-1},[B_{K_P^*}]^{-1}\mid P\in \cP]$$
with a reduced quotient
$$\cD\cK^{\red}=\cD\cK/\langle [B_{K_P}]*[B_{K_P^*}]-1\mid P\in \cP\rangle$$
such that the trace map induces algebra isomorphisms 
$$\bbC\otimes_{\bbZ[\overline{\bbQ}_l^*]} \cD\cK\cong \cD\tilde{\cH}(\cA_q),\ \bbC\otimes_{\bbZ[\overline{\bbQ}_l^*]} \cD\cK^{\red}\cong \cD\tilde{\cH}^{\red}(\cA_q).$$
\end{theorem}
\begin{proof}
By Proposition \ref{associativity}, the $\bbZ[\overline{\bbQ}_l^*]$-module $\cK$ is an associative algebra. By Proposition \ref{chi}, the trace map induces an algebra homomorphism. Since $\mathbf{Q}$ is a Dynkin quiver, all finite-dimensional representations over $k=\overline{\bbF}_q$ are defined over $\bbF_q$, that is, for any $A\in \cA_k$, where exists $A_0\in \cA_{q}$ such that $A\cong k\otimes_{\bbF_q} A_0$. In particular, objects in $\cP_k$ are defined over $\bbF_q$, and moreover, objects in $\cC_2(\cP_k)$ are defined over $\bbF_q$ by \cite[Lemma 4.2]{Bridgeland-2013}. Thus for any $\rG_{\ue}$-orbit $\cO_{M_\bullet}$, it contains a $\sigma$-fixed point $M_{\bullet0} \in \rC_\ue^\sigma$. By Lang's theorem, we have $\cO_{M_{\bullet}}^\sigma\cong\rG_{\ue}^\sigma/\Stab(\cO_{M_{\bullet0}})$ which coincides with the $\rG_{\ue}^\sigma$-orbit of $M_{\bullet0}$ in $\rC_\ue^\sigma$. The trace map sends a basis element $S_{M_\bullet}\in \cK_{\ue}$ to $v_q^{-\dim \cO_{M_{\bullet}}}1_{\cO_{M_{\bullet}}^\sigma}\in \tilde{\cH}_{\rG_\ue^\sigma}(\rC_{\ue}^\sigma)$. All these elements $v_q^{-\dim \cO_{M_{\bullet}}}1_{\cO_{M_{\bullet}}^\sigma}$ also form a basis of $\tilde{\cH}_{\rG_\ue^\sigma}(\rC_{\ue}^\sigma)$, and so the trace map induces an isomorphism $\bbC\otimes_{\bbZ[\overline{\bbQ}_l^*]} \cK\cong \tilde{\cH}^{\tw}(\cC_2(\cP_q))$. Moreover, notice that the trace map sends the elements $[B_{K_P}]$ and $[B_{K_P^*}]$ to $a_{K_P}1_{\cO_{K_P}}$ and $a_{K_P^*}1_{\cO_{K_P^*}}$ respectively, thus the subset $\{[B_{K_P}],[B_{K_P^*}]\mid P\in \cP\}$ also satisfies the Ore conditions. It is obvious the trace map induces algebra isomorphisms 
$$\bbC\otimes_{\bbZ[\overline{\bbQ}_l^*]} \cD\cK\cong \cD\tilde{\cH}(\cA_q),\ \bbC\otimes_{\bbZ[\overline{\bbQ}_l^*]} \cD\cK^{\red}\cong \cD\tilde{\cH}^{\red}(\cA_q)$$
between the localizations, and between the reduced quotients. 
 \end{proof}

Combining with Corollary \ref{function-isomorphism}, we obtain isomorphisms
$$\bbC\otimes_{\bbZ[\overline{\bbQ}_l^*]} \cD\cK^{\red}\cong \cD\tilde{\cH}^{\red}(\cA_q)\cong \cD\cH^{\red}(\cA_q).$$
Replacing the Frobenius morphism $\sigma$ by its power $\sigma^n$ for $n\geqslant 1$, we can obtain similar results over the finite field $\bbF_{q^n}$. This completes a realization of Bridgeland's Hall algebra via constructible sheaves.

\section{Perverse sheaves realization of Bridgeland's Hall algebra}\label{Perverse sheaves realization of Bridgeland's Hall algebra}

\subsection{Hyperbolic localization functor}\label{Hyperbolic localization}

In this subsection, we review a theorem by Braden in \cite{Braden-2003}. Another proof without the assumption that $X$ is a normal variety is given by Drinfeld-Gaitsgory in \cite{Drinfeld-Gaitsgory-2014}.

Let $X$ be a (normal) $k$-variety together with a $k^*$-action, and $X^{k^*}$ be the subvariety of fixed points with connected components $X_1,...,X_r$. We define
\begin{align*}
X_i^+=\{x\in X\mid \lim_{t\rightarrow 0}t.x\in X_i\},\ X_i^-=\{x\in X\mid \lim_{t\rightarrow\infty}t.x\in X_i\}
\end{align*}
for $i=1,...,r$. Let $X^\pm=\bigsqcup^r_{i=1}X_i^\pm$ be their disjoint unions, and 
\begin{align*}
&f^\pm:X^{k^*}=\bigsqcup^r_{i=1}X_i\rightarrow \bigsqcup^r_{i=1}X_i^\pm=X^\pm,\ g^\pm:X^\pm\rightarrow X
\end{align*}
be the inclusions. The hyperbolic localization functors are defined by 
\begin{align*}
(-)^{!*}:\cD^b(X)&\rightarrow \cD^b(X^{k^*}) &(-)^{*!}:\cD^b(X)&\rightarrow \cD^b(X^{k^*})\\
K&\mapsto(f^+)^!(g^+)^*(K), &K&\mapsto(f^-)^*(g^-)^!(K).
\end{align*}

An object $K$ in $\cD^b(X)$ is said to be weakly equivariant, if $\mu^*(K)\cong L\boxtimes K$ for some locally constant sheaf $L$ on $k^*$, where $\mu:k^*\times X\rightarrow X$ is the morphism defining the $k^*$-action on $X$. Note that if $K$ is a $k^*$-equivariant perverse sheaf, or more generally, a $k^*$-equivariant semisimple complex on $X$, then $K$ is weakly equivariant, since $\mu^*(K)\cong p^*(K)\cong \overline{\bbQ}_l\boxtimes K$, where $p:k^*\times X\rightarrow X$ is the natural projection.

By \cite[Theorem 1]{Braden-2003}, for any weakly equivariant object $K$, there is an isomorphism $K^{*!}\xrightarrow{\cong}K^{!*}$. Moreover, \cite[Theorem 8]{Braden-2003}, the hyperbolic localization functors preserve purity of weakly equivariant mixed sheaves. As a result, the hyperbolic localization functors send weakly equivariant mixed semisimple complexes to semisimple complexes (see \cite[Theorem 5.7.9]{Pramod-2021}). The hyperbolic localization functors have another descriptions. We define 
\begin{align*}
\pi^+:X^+&\rightarrow X^{k^*} &\pi^-:X^-&\rightarrow X^{k^*}\\
x&\mapsto \lim_{t\rightarrow 0}t.x, &x&\mapsto \lim_{t\rightarrow \infty}t.x,
\end{align*}
then for any weakly equivariant object $K$, by \cite[formula (1)]{Braden-2003}, there are isomorphisms
\begin{align*} 
K^{!*}\cong (\pi^+)_!(g^+)^*(K),\ K^{*!}\cong (\pi^-)_*(g^-)^!(K).
\end{align*}

Notice that replacing the $k^*$-action by the opposite action interchanges $X^+$ and $X^-$, so the hyperbolic localization functors are sent to their Verdier duals. More precisely, for any weakly equivariant object $K$, we have
$$(\bbD K)^{!*}\cong (\pi^+)_!(g^+)^*(\bbD K)\cong \bbD(\pi^+)_*(g^+)^!(K)\cong \bbD(K^{*!,\textrm{opp}})\cong \bbD(K^{!*,\textrm{opp}}),$$
where $K^{*!,\textrm{opp}},K^{!*,\textrm{opp}}$ are defined with respect to the opposite $k^*$-action on $X$, that is, $t._{\textrm{opp}}x=t^{-1}.x$.

\subsection{Restriction functor and hyperbolic localization functor}

In this subsection, we prove that the restriction functor $\Res^{\ue}_{\ue',\ue''}$ defined in \S \ref{Restriction functor} is a hyperbolic localization functor. We use the same notations as them in \S \ref{Restriction functor}.

\begin{proposition}\label{restriction-hyperbolic}
Let $\ue,\ue',\ue''\in  \bbN[I]\times \bbN[I]$ be such that $\ue=\ue'+\ue''$. Then the composition of functors 
$$\cD^b_{m}(\rC_{\ue})\xrightarrow{\iota^*}\cD^b_{m}({\rm{F}})\xrightarrow{\kappa_!}\cD^b_{m}(\rC_{\ue'}\times \rC_{\ue''})$$
is a hyperbolic localization functor, and so the restriction functor can be restricted to a functor
$$\Res^{\ue}_{\ue',\ue''}\!:\!\cD^{b,ss}_{\rG_{\ue},m}(\rC_{\ue})\rightarrow \cD^{b,ss}_{\rG_{\ue'}\times \rG_{\ue''},m}(\rC_{\ue'}\times \rC_{\ue''})$$
which sends semisimple complexes to semisimple complexes.
\end{proposition}
\begin{proof}
We denote by $P^j=\bigoplus_{i\in I}e^j_iP_i,P'^j=\bigoplus_{i\in I}e'^j_iP_i,P''^j=\bigoplus_{i\in I}e''^j_iP_i$, and so $P^j=P'^j\oplus P''^j$ for any $j\in \bbZ_2$. For any $(d^1,d^0)\in \rC_{\ue}$, we write $d^j\in \Hom_{\cA}(P^j,P^{j+1})$ in the matrix form
$$\xymatrix@C=3cm{P'^1\oplus P''^1 \ar@<.5ex>[r]^-{{d^1=\begin{pmatrix}d'^1 &\beta^1\\
\alpha^1 &d''^1\end{pmatrix}}} &P'^0\oplus P''^0 \ar@<.5ex>[l]^-{d^0=\begin{pmatrix}d'^0 &\beta^0\\
\alpha^0 &d''^0\end{pmatrix}}}$$
There is a one-parameter subgroup 
\begin{align*}
\zeta:k^*&\hookrightarrow \rQ=\rQ^1\times \rQ^0\\
t&\mapsto (\Id_{P'^1}\oplus t\Id_{P''^1},\Id_{P'^0}\oplus t\Id_{P''^0})
\end{align*}
such that $k^*$ acts on $\rC_{\ue}$ via $t.(d^1,d^0)=\zeta(t).(d^1,d^0)$, more precisely, 
$$t.(\begin{pmatrix}d'^1 &\beta^1\\ \alpha^1 &d''^1\end{pmatrix},\begin{pmatrix}d'^0 &\beta^0\\ \alpha^0 &d''^0\end{pmatrix})=(\begin{pmatrix}d'^1 &t^{-1}\beta^1\\
t\alpha^1 &d''^1\end{pmatrix},\begin{pmatrix}d'^0 &t^{-1}\beta^0\\
t\alpha^0 &d''^0\end{pmatrix}).$$
Then there is a commutative diagram
$$\xymatrix{\rC_{\ue}^{k^*} \ar[d]_-{\cong} &\rC_{\ue}^+ \ar[d]_-{\cong} \ar[l]_-{\pi^+} \ar[r]^-{g^+} &\rC_{\ue} \ar@{=}[d]\\ \rC_{\ue'}\times \rC_{\ue''} &\rF \ar[l]_-{\kappa} \ar[r]^-{\iota} &\rC_{\ue} }$$
Hence the functor $\kappa_!\iota^*\cong (\pi^+)_!(g^+)^*$ is a hyperbolic localization functor. For any $C\in \cD^{b,ss}_{\rQ,m}(\rC_{\ue})$, it can be viewed as an object in $\cD^{b,ss}_{k^*,m}(\rC_{\ue})$ via $\zeta:k^*\hookrightarrow \rQ$ and the forgetful functor $\cD^{b,ss}_{\rQ,m}(\rC_{\ue})\rightarrow\cD^{b,ss}_{k^*,m}(\rC_{\ue})$, then it is weakly equivariant. By Braden's Theorem (see \S \ref{Hyperbolic localization}), $\kappa_!\iota^*(C)$ is a semisimple complex on $\rC_{\ue'}\times \rC_{\ue''}$. Therefore, 
$$\Res^{\ue}_{\ue',\ue''}:\cD^{b,ss}_{\rG_{\ue},m}(\rC_{\ue})\rightarrow \cD^{b,ss}_{\rG_{\ue'}\times \rG_{\ue''},m}(\rC_{\ue'}\times \rC_{\ue''})$$
sends semisimple complexes to semisimple complexes.
\end{proof}

We denote by $\tau:\rC_{\ue''}\times \rC_{\ue'}\rightarrow \rC_{\ue'}\times \rC_{\ue''}$ the isomorphism switching two coordinates.

\begin{corollary}\label{Verdier}
For any $\ue=\ue'+\ue''\in \bbN[I]\times \bbN[I]$ and $C\in\cD^{b,ss}_{\rG_{\ue},m}(\rC_{\ue})$, we have 
$$(\bbD\boxtimes \bbD)\Res^{\ue}_{\ue',\ue''}(\bbD C)\cong\tau_!\Res^{\ue}_{\ue'',\ue'}(C)[-\Vert\ue',\ue''\Vert](-\frac{\Vert\ue',\ue''\Vert}{2}),$$
where $\bbD\boxtimes \bbD$ is the Verdier dual of $\cD^{b,ss}_{\rG_{\ue'}\times \rG_{\ue''},m}(\rC_{\ue'}\times \rC_{\ue''})$, and $\Vert\ue',\ue''\Vert=|\ue',\ue''|+|\ue'',\ue'|$ is a symmetric bilinear form.
\end{corollary}
\begin{proof}
With respect to the $k^*$-action defined in the proof of Proposition \ref{restriction-hyperbolic}, notice that 
\begin{align*}
\Res^{\ue}_{\ue',\ue''}(\bbD C)&\cong (\bbD C)^{!*}[|\ue',\ue''|](\frac{|\ue',\ue''|}{2}),\\
\tau_!\Res^{\ue}_{\ue'',\ue'}(C)&\cong(C)^{!*,\textrm{opp}}[|\ue'',\ue'|](\frac{|\ue'',\ue'|}{2}).
\end{align*}
By the relation between the Verdier dual and the hyperbolic localization functors, see \S \ref{Hyperbolic localization}, we have 
\begin{align*}
(\bbD\boxtimes \bbD)\Res^{\ue}_{\ue',\ue''}(\bbD C)\cong&(\bbD\boxtimes \bbD)((\bbD C)^{!*}[|\ue',\ue''|](\frac{|\ue',\ue''|}{2}))\\
\cong &(\bbD\boxtimes \bbD)(\bbD C)^{!*}[-|\ue',\ue''|](-\frac{|\ue',\ue''|}{2})\\
\cong &(\bbD\boxtimes \bbD)(\bbD\boxtimes \bbD)(C)^{!*,\textrm{opp}}[-|\ue',\ue''|](-\frac{|\ue',\ue''|}{2})\\
\cong &\tau_!\Res^{\ue}_{\ue'',\ue'}(C)[-|\ue',\ue''|-|\ue'',\ue'|](-\frac{|\ue',\ue''|+|\ue'',\ue'|}{2}),
\end{align*}
as desired.
\end{proof}

\subsection{Semisimple Grothendieck group and localization}\label{Bridgeland's Hall algebra via perverse sheaves}

\begin{definition}
{\rm{(a)}} For $\ue\in \bbN[I]\times \bbN[I]$, we define $\cK_{\ue}^{ss}$ to be the Grothendieck group of $\cD^{b,ss}_{\rG_{\ue},m}(\rC_{\ue})$, and define a $\cZ$-module structure on it via
 $$v.[L]=[L[-1](-\frac{1}{2})].$$
{\rm{(b)}} We define the direct sum and its graded dual
$$
\cK^{ss}=\bigoplus_{\ue}\cK_{\ue}^{ss},\ \
\cK^{ss,*}=\bigoplus_{\ue}\cK^{ss,*}_{\ue}=\bigoplus_{\ue}\Hom_{\cZ}(\cK_{\ue}^{ss},\cZ).$$
\end{definition}

\begin{proposition}
The $\cZ$-module $\cK_{\ue}^{ss}$ is free with a basis $\cI_{\ue}$.
\end{proposition}
\begin{proof}
The $\cZ$-module $\cK_{\ue}^{ss}$ is free with a basis given by $\rG_{\ue}$-equivariant simple perverse sheaves. By Lemma \ref{Finite orbits}, any $\rG_{\ue}$-equivariant simple perverse sheaf is of the form $\IC(\cO_{M_\bullet},\overline{\bbQ}_l)$ for some $\rG_{\ue}$-orbit $\cO_{M_\bullet}\subset \rC_{\ue}$. Hence $\cI_{\ue}$ is a $\cZ$-basis of $\cK_{\ue}^{ss}$.
\end{proof}

\begin{definition}
We define $\cI^*_{\ue}=\{I_{M_\bullet}^*|\cO_{M_{\bullet}}\subset \rC_{\ue}\}\subset \cK^{ss,*}_{\ue}$ to be the dual basis of $\cI_{\ue}=\{I_{M_\bullet}|\cO_{M_{\bullet}}\subset \rC_{\ue}\}\subset \cK^{ss}_{\ue}$, and define the unions 
$$\cI^*=\bigsqcup_{\ue}\cI_{\ue}^*$$
which is a $\cZ$-basis of $\cK^{ss,*}$.
\end{definition}

\begin{lemma}\label{external tensor product}
For any $\ue',\ue''\in \bbN[I]\times \bbN[I]$, the exterior tensor functor $$-\boxtimes-:\cD^{b}_{\rG_{\ue'},m}(\rC_{\ue'})\times \cD^{b}_{\rG_{\ue''},m}(\rC_{\ue''})\rightarrow \cD^{b}_{\rG_{\ue'}\times \rG_{\ue''},m}(\rC_{\ue'}\times \rC_{\ue''})$$
gives an identification from the set of the isomorphism classes of the exterior tensor products of simple perverse sheaves in $\cD^{b,ss}_{\rG_{\ue'},m}(\rC_{\ue'}), \cD^{b,ss}_{\rG_{\ue''},m}(\rC_{\ue''})$ to the set of the isomorphism classes of simple perverse sheaves in $ \cD^{b}_{\rG_{\ue'}\times \rG_{\ue''},m}(\rC_{\ue'}\times \rC_{\ue''})$. As a result, the Grothendieck group of $\cD^{b,ss}_{\rG_{\ue'}\times \rG_{\ue''},m}(\rC_{\ue'}\times \rC_{\ue''})$ can be identified with $\cK^{ss}_{\ue'}\otimes_{\cZ}\cK^{ss}_{\ue''}$.
\end{lemma}
\begin{proof}
By \cite[Lemma 3.3.14]{Pramod-2021}, the external tensor product of simple perverse sheaves is simple. Conversely, any simple perverse sheaf in $\cD^{b,ss}_{\rG_{\ue'}\times \rG_{\ue''},m}(\rC_{\ue'}\times \rC_{\ue''})$ is of the form $\IC(U,\cL)$, where $U\subset \rC_{\ue'}\times \rC_{\ue''}$ is a smooth, locally closed, irreducible, $\rG_{\ue'}\times \rG_{\ue''}$-invariant subset and $\cL$ is an irreducible $\rG_{\ue'}\times \rG_{\ue''}$-equivariant local system on $U$, see \cite[Section 5.2]{Bernstein-Lunts-1994}. By Lemma \ref{Finite orbits}, the subset $U$ can be written as the union of finitely many $\rG_{\ue'}\times \rG_{\ue''}$-orbits. Suppose $\cO$ is the orbit having maximal dimension, then $\cO$ is open dense in $U$, and so $\IC(U,\cL)\cong\IC(\cO,\cL|_{\cO})$. Note that the $\rG_{\ue'}\times \rG_{\ue''}$-orbit $\cO\subset \rC_{\ue'}\times \rC_{\ue''}$ can be written as $\cO_1\times \cO_2$ for some $\rG_{\ue'}$-orbit $\cO_1\subset \rC_{\ue'}$ and $\rG_{\ue''}$-orbit $\cO_2\subset \rC_{\ue'}$.
By similar argument as Lemma \ref{Finite orbits}, the local system $\cL|_{\cO}$ must be constant, that is, $\cL|_{\cO}=\overline{\bbQ}_l|_{\cO}=\overline{\bbQ}_l|_{\cO_1}\boxtimes \overline{\bbQ}_l|_{\cO_2}$, and so 
\begin{align*}
\IC(U,\cL)\cong\IC(\cO,\cL|_{\cO})=&\IC(\cO_1\times \cO_2,\overline{\bbQ}_l|_{\cO_1}\boxtimes \overline{\bbQ}_l|_{\cO_2})\\ \cong&\IC(\cO_1,\overline{\bbQ}_l)\boxtimes \IC(\cO_2,\overline{\bbQ}_l),
\end{align*}
where the last isomorphism follows from \cite[Lemma 3.3.14]{Pramod-2021}, as desired.
\end{proof}

\begin{lemma}
For any $\ue\in \bbN[I]\times \bbN[I]$, we view $\bbC$ as a $\cZ$-module via $v.z=v_qz$, then
\begin{align}\label{dual space isomorphism}
\Hom_{\bbC}(\bbC\otimes_{\cZ}\cK_{\ue}^{ss},\bbC)\cong \bbC\otimes_{\cZ}\cK_{\ue}^{ss,*}.
\end{align}
\end{lemma}
\begin{proof}
On the one hand, by the adjoint isomorphism, we have 
$$\Hom_{\bbC}(\bbC\otimes_{\cZ}\cK_{\ue}^{ss},\bbC)\cong \Hom_{\cZ}(\cK^{ss}_{\ue},\Hom_{\bbC}(\bbC,\bbC))=\Hom_{\cZ}(\cK^{ss}_{\ue},\bbC),$$
see \cite[Theorem 2.76]{Rotman-2009}. On the other hand, there is a morphism 
\begin{align*}
\bbC\otimes_{\cZ}\cK_{\ue}^{ss,*}=\bbC\otimes_{\cZ}\Hom_{\cZ}(\cK_{\ue}^{ss},\cZ)\rightarrow  \Hom_{\cZ}(\cK^{ss}_{\ue},\bbC)
\end{align*}
induced by $(z,f)\mapsto (x\mapsto f(x).z)$ which is an isomorphism, since $\cK_{\ue}^{ss}$ is a free $\cZ$-module of finite rank.
\end{proof}

\begin{theorem}\label{dual algebra}
{\rm{(a)}} All restriction functors $\Res^{\ue}_{\ue',\ue''}$ for $\ue=\ue'+\ue''$ induce a comultiplication 
$$r:\cK^{ss}\rightarrow \cK^{ss}\otimes_{\cZ}\cK^{ss}$$
such that $\cK^{ss}$ is a coassociative coalgebra, and the trace map induces an isomorphism
$$\chi:\bbC\otimes_{\cZ}\cK^{ss}\xrightarrow{\cong} \tilde{\cH}^{\tw}(\cC_2(\cP_q)).$$
{\rm{(b)}} Dually, the map $r$ induces a multiplication 
$$*_r:\cK^{ss,*}\otimes_{\cZ}\cK^{ss,*}\rightarrow \cK^{ss,*}$$
such that $\cK^{ss,*}$ is an associative algebra and the dual of the isomorphism $\chi$ together with isomorphisms {\rm{(\ref{dual space isomorphism})}} induces an algebra isomorphism 
$$\chi^*:\tilde{\cH}^{\tw,*}(\cC_2(\cP_q))\rightarrow \bigoplus_{\ue}\Hom_{\bbC}(\bbC\otimes_{\cZ}\cK_{\ue}^{ss},\bbC)\cong \bbC\otimes_{\cZ}\cK^{ss,*}.$$
\end{theorem}
\begin{proof}
(a) By Proposition \ref{coassociativity}, Lemma \ref{restriction-hyperbolic} and Lemma \ref{external tensor product}, the $\cZ$-module $\cK^{ss}$ is a coassociative coalgebra. By Proposition \ref{chi}, the trace map induces a homomorphism. Since $\bbC\otimes_{\cZ}\cK_{\ue}^{ss}$ has a $\bbC$-basis $\cI_{\ue}$ and $\tilde{\cH}^{\tw}(\cC_2(\cP_q))$ has a $\bbC$-basis $\{1_{\cO^{\sigma}_{M_\bullet}}\}_{\cO^{\sigma}_{M_\bullet}\subset\rC_{\ue}^{\sigma}}$, both of them are parametrized by isomorphism classes of objects in $\cC_2(\cP)$ of projective dimension vector pair $\ue$, the $\bbN[I]\times \bbN[I]$-graded spaces $\bbC\otimes_{\cZ}\cK^{ss}$ and $\tilde{\cH}^{\tw}(\cC_2(\cP_q))$ have the same graded dimensions. Recall that for any $I_{M_\bullet}\in \cI_{\ue}$, the complex $\IC_{M_\bullet}$ is the unique $\rG_{\ue}$ simple perverse sheaf on $\rC_{\ue}$ supported on $\overline{\cO_{M_\bullet}}$ whose restriction on $\cO_{M_\bullet}$ is $\overline{\bbQ}_l|_{\cO_{M_\bullet}}[\dim \cO_{M_\bullet}](\frac{\dim\cO_{M_\bullet}}{2})$. Thus the homomorphism induced by the trace map corresponds to a upper triangular matrix whose diagonal entries are powers of $v$ with respect to the bases $\cI_{\ue}$ and $\{1_{\cO^{\sigma}_{M_\bullet}}\}_{\cO^{\sigma}_{M_\bullet}\subset\rC_{\ue}^{\sigma}}$. Hence it is an isomorphism.\\
(b) The dual statement follows directly.
\end{proof}

By Corollary \ref{function-isomorphism} and Theorem \ref{dual algebra}, there are algebra isomorphisms
\begin{equation}
\begin{aligned}\label{three algebra isomorphism}
\cH^{\tw}(\cC_2(\cP_q))\xrightarrow{\cong} &\tilde{\cH}^{\tw,*}(\cC_2(\cP_q))\xrightarrow\cong \bbC\otimes_{\cZ}\cK^{ss,*}\\
a_{M_\bullet}u_{[M_\bullet]}\mapsto &\ \ \ \ 1_{\cO_{M_\bullet}}^*\ \ \ \mapsto  \ \ \chi^*(1_{\cO_{M_\bullet}}^*),
\end{aligned}
\end{equation}
where
$$\chi^*(1_{\cO_{M_\bullet}}^*)=\sum_{I_{N_\bullet}^*\in \cI^*}\chi^*(1_{\cO_{M_\bullet}}^*)(I_{N_\bullet})\otimes I_{N_\bullet}^*=\sum_{I_{N_\bullet}^*\in \cI^*}1_{\cO_{M_\bullet}}^*(\chi_{I_{N_\bullet}})\otimes I_{N_\bullet}^*.$$
Here we write the notations $a_{M_\bullet}u_{[M_\bullet]}$ and $1_{\cO_{M_\bullet}}^*$ for convenience and keep in mind that the more precise notations should be $a_{M_{\bullet0}}u_{[M_{\bullet0}]}$ and $1_{\cO_{M_\bullet}^\sigma}^*$ respectively, where $M_{\bullet0}\in \cC_2(\cP_q)$ is an object over $\bbF_q$ such that $M_\bullet\cong k\otimes_{\bbF_q}M_{\bullet0}$, and $\cO_{M_\bullet}^\sigma$ is the fixed point subset of $\cO_{M_\bullet}$ which coincides with the $\rG_{\ue_{M_\bullet}}^\sigma$-orbit of $M_{\bullet0}$ in $\rC_{\ue_{M_\bullet}}^\sigma$. 

\begin{lemma}\label{image of b}
If the orbit $\cO_{M_\bullet}\subset \rC_{\ue}$ is maximal with respect to the partial order $\preccurlyeq$, then under the algebra isomorphisms {\rm{(\ref{three algebra isomorphism})}}, we have 
$$a_{M_\bullet}u_{[M_\bullet]}\mapsto 1_{\cO_{M_\bullet}}^*\mapsto 1\otimes v^{-\dim \cO_{M_\bullet}}I_{M_\bullet}^*.$$
In particular, for any $P,Q\in \cP$, the orbit $\cO_{K_P\oplus K_Q^*}\subset \rC_{\ue_{K_P\oplus K_Q^*}}$ is maximal, and so under the algebra isomorphisms {\rm{(\ref{three algebra isomorphism})}}, we have
\begin{align*}
b_{K_P}*b_{K_Q^*}\mapsto 1_{\cO_{K_P\oplus K_Q^*}}^*\mapsto 1\otimes v^{-\langle \hat{P}+\hat{Q},\hat{P}+\hat{Q}\rangle}I_{K_P\oplus K_Q^*}^*.
\end{align*}
\end{lemma}
\begin{proof}
By definition, there is no any other orbit $\cO_{M'_\bullet}\subset \rC_{\ue}$ such that $\cO_{M_\bullet}\subset \overline{\cO_{M'_\bullet}}$, and so 
$$\chi^*(1_{\cO_{M_\bullet}}^*)(I_{M'_\bullet})=1_{\cO_{M_\bullet}}^*(\chi_{I_{M'_\bullet}})=0,$$
for any $I_{M'_\bullet}^*\in \cI^*_{\ue}\setminus \{I_{M_\bullet}^*\}$, since the support of $\IC_{M'_\bullet}$ is $\overline{\cO_{M'_\bullet}}$. Thus
$$\chi^*(1_{\cO_{M_\bullet}}^*)=\chi^*(1_{\cO_{M_\bullet}}^*)(I_{M_\bullet})\otimes I_{M_\bullet}^*=v_q^{-\dim \cO_{M_\bullet}}\otimes I_{M_\bullet}^*=1\otimes v^{-\dim \cO_{M_\bullet}}I_{M_\bullet}^*,$$
since the restriction of $\IC_{M_\bullet}$ to $\cO_{M_\bullet}$ is $\overline{\bbQ}_l|_{\cO_{M_\bullet}}[\dim \cO_{M_\bullet}](\frac{\dim \cO_{M_\bullet}}{2})$. 

In particular, for any $P,Q\in \cP$, it is clear that any non-isomorphism can not degenerate to be an isomorphism, and so $\cO_{K_P\oplus K_Q^*}$ is a maximal orbit in $\rC_{\ue_{K_P\oplus K_Q^*}}$. Moreover, since $b_{K_P}*b_{K_Q^*}=a_{K_P\oplus K_Q^*}u_{[K_P\oplus K_Q^*]}$, and the dimension of the orbit
\begin{align*}
\cO_{K_P\oplus K_Q^*}\cong&\rG_{\ue_{K_P\oplus K_Q^*}}/\Stab(\cO_{K_P\oplus K_Q^*})\\
\cong &(\Aut_{\cA}(P\oplus Q)\times \Aut_{\cA}(P\oplus Q))/\Aut_{\cC_2(\cP)}(K_P\oplus K_Q^*)
\end{align*}
is $\langle \hat{P}+\hat{Q},\hat{P}+\hat{Q}\rangle$, the images of $b_{K_P}*b_{K_Q^*}$ are as desired. 
\end{proof}

\begin{example}
For any indecomposable $P\in \cP$, since $\mathbf{Q}$ is a Dynkin quiver, $\End_{\cA}(P)=k$. So there are only three orbits $\cO_{C_P\oplus C_P^*},\cO_{K_P},\cO_{K_P^*}$ in the variety $\rC_{\ue_{K_P}}$. It is clear that $S_{C_P\oplus C_P^*}=I_{C_P\oplus C_P^*}\in \cK^{ss}_{\ue_{K_P}}$, since the orbit $\cO_{C_P\oplus C_P^*}=\{(0,0)\}$ is a single point. Note that 
$$\cO_{K_P}=\{(d,0)|d\in k^*\}\subset \End_{\cA}(P)\times \{0\}\subset\rC_{\ue_{K_P}},$$
where $\End_{\cA}(P)\times \{0\}$ is an irreducible smooth closed subset of $\rC_{\ue_{K_P}}$, and $\cO_{K_P}$ is an open dense subset of $\End_{\cA}(P)\times \{0\}$ with complement $\cO_{C_P\oplus C_P^*}=\{(0,0)\}$. Hence we have 
\begin{align*}
\IC_{K_P}\cong s_!(\overline{\bbQ}_l|_{\End_{\cA}(P)\times \{0\}}[1](\frac{1}{2})),
\end{align*}
where $s:\End_{\cA}(P)\times \{0\}\rightarrow \rC_{\ue_{K_P}}$ is the inclusion, and so $S_{K_P}=I_{K_P}-v^{-1}I_{C_P\oplus C_P^*}\in \cK^{ss}_{\ue_{K_P}}$.
Similarly, we have $S_{K_P^*}=I_{K_P^*}-v^{-1}I_{C_P\oplus C_P^*}\in \cK^{ss}_{\ue_{K_P}}$. Moreover, we have 
\begin{align*}
(S_{C_P\oplus C_P^*},\ S_{K_P},\ S_{K_P^*})=(I_{C_P\oplus C_P^*},\ I_{K_P},\ I_{K_P^*})\left(\begin{smallmatrix}
1 &-v^{-1} &-v^{-1}\\
0 &1 &0\\
0 &0 &1
\end{smallmatrix}\right).
\end{align*}
Thus $\cS_{\ue_{K_P}}=\{S_{C_P\oplus C_P^*},S_{K_P},S_{K_P^*}\}$ is also a $\cZ$-basis of $\cK_{\ue_{K_P}}^{ss}$, and its dual basis is
\begin{align*}
&(I^*_{C_P\oplus C_P^*},\ I^*_{K_P},\ I^*_{K_P^*})\left(\left(\begin{smallmatrix}
1 &-v^{-1} &-v^{-1}\\
0 &1 &0\\
0 &0 &1
\end{smallmatrix}\right)^{-1}\right)^T\\
=&(I^*_{C_P\oplus C_P^*}+v^{-1}I^*_{K_P}+v^{-1}I^*_{K_P^*},\ I^*_{K_P},\ I^*_{K_P^*}).
\end{align*}
The values of $\chi^*(1_{\cO_{K_P}}^*)$ on the basis elements $S_{C_P\oplus C_P^*},S_{K_P},S_{K_P^*}$ are $0,v_q^{-1},0$ respectively, and so $\chi^*(1_{\cO_{K_P}}^*)=v_q^{-1}\otimes I_{K_P}^*=1\otimes v^{-1}I_{K_P}^*$. Similarly, we have $\chi^*(1_{\cO_{K_P^*}}^*)=1\otimes v^{-1} I_{K_P^*}^*$. 
\end{example}

\begin{definition}
For any $P,Q\in \cP$, we define 
$$\tilde{I}^*_{K_P\oplus K_Q^*}=v^{-\langle \hat{P}+\hat{Q},\hat{P}+\hat{Q}\rangle}I^*_{K_P\oplus K_Q^*}\in \cK^{ss,*}_{\ue_{K_P\oplus K_Q^*}}.$$
In particular, $\tilde{I}_{K_P}^*=v^{-\langle \hat{P},\hat{P}\rangle}I_{K_P}^*\in \cK^{ss,*}_{\ue_{K_P}},\ \tilde{I}_{K_Q^*}^*=v^{-\langle \hat{Q},\hat{Q}\rangle}I_{K_Q^*}^*\in \cK^{ss,*}_{\ue_{K_Q^*}}$.
\end{definition}

By Lemma \ref{image of b}, the elements $1\otimes \tilde{I}_{K_P}^*, 1\otimes  \tilde{I}_{K_P^*}^*$ are the images of $b_{K_P}, b_{K_P^*}$ respectively, under the algebra isomorphisms (\ref{three algebra isomorphism}). This inspires us take the localization of $\cK^{ss,*}$ with respect to $\{\tilde{I}_{K_P}^*, \tilde{I}_{K_P^*}^*|P\in \cP\}$, as Bridgeland took the localization of $\cH^{\tw}(\cC_2(\cP_q))$ by with respect to $\{b_{K_P},b_{K_P^*}|P\in \cP_q\}$.

\begin{proposition}\label{Ore_2}
For any $P,Q\in \cP$ and $x\in \cK^{ss,*}_{\ue}$, we have 
\begin{align}
\tilde{I}^*_{K_P}*_rx&=v^{(\hat{P},\hat{P^0}-\hat{P^1})}x*_r\tilde{I}^*_{K_P}, \label{I-Ore}  \\
\tilde{I}^*_{K_Q^*}*_rx&=v^{(-\hat{Q},\hat{P^0}-\hat{P^1})}x*_r\tilde{I}^*_{K_Q^*}, \label{I-Ore'} \\
[\tilde{I}_{K_P}^*,\tilde{I}_{K_Q}^*]&=[\tilde{I}_{K_P^*}^*, \tilde{I}_{K_Q^*}^*]=[\tilde{I}_{K_P}^*, \tilde{I}_{K_Q^*}^*]=0, \label{commutative}
\end{align}
where $P^j\cong\bigoplus_{i\in I}e^j_iP_i$, and $[a, b]=a*_rb-b*_ra$ is the commutator.
\end{proposition}
\begin{proof}
It is clear that the formula (\ref{I-Ore}) is equivalent to the following formula
\begin{align}\label{equivalent Ore}
I^*_{K_P}*_rx=v^{(\hat{P},\hat{P^0}-\hat{P^1})}x*_rI^*_{K_P}.
\end{align}
We only need to prove it in the case that $x=I^*\in \cI_{\ue}^*$ is a basis element. For any $y\in \cK_{\ue+\ue_{K_P}}$, suppose that $r(y)=\sum \zeta_{I_1,I_2}(v)I_1\otimes I_2$, where $I_1,I_2\in \cI$ and $\zeta_{I_1,I_2}(v)\in \cZ$. By definition, we have
\begin{align*}
I_{K_P}^**_rI^*(y)=(I_{K_P}^*\otimes I^*)(r(y))=\zeta_{I_{K_P},I}(v),\\
I^**_rI_{K_P}^*(y)=(I^*\otimes I_{K_P}^*)(r(y))=\zeta_{I,I_{K_P}}(v).
\end{align*}
By \cite[Lemma 3.5]{Bridgeland-2013}, the elements $b_{K_P}$ satisfies
$$b_{K_P}*u_{[M_\bullet]}=v_q^{(\hat{P},\hat{M^0}-\hat{M^1})}u_{[M_\bullet]}*b_{K_P}\in \cH^{\tw}(\cC_2(\cP_q)).$$
By the algebra isomorphisms (\ref{three algebra isomorphism}), the formula (\ref{equivalent Ore}) holds for $v=v_q$ in $\bbC\otimes_{\cZ}\cK^{ss,*}$. Replacing the Frobenius morphism $\sigma$ by its power $\sigma^n$ for $n\geqslant 1$, the formula (\ref{equivalent Ore}) holds for $v=v_q^n$ in $\bbC\otimes_{\cZ}\cK^{ss,*}$, that is, 
$$1\otimes (I_{K_P}^**_rI^*)=(v_q^n)^{(\hat{P},\hat{P^0}-\hat{P^1})}\otimes (I^**_rI_{K_P}^*).$$
In particular, taking their values at $y$, we obtain
$$\zeta_{I_{K_P},I}(v_q^n)=(v_q^n)^{(\hat{P},\hat{P^0}-\hat{P^1})}\zeta_{I,I_{K_P}}(v_q^n),$$
Hence the Laurent polynomials $\zeta_{I_{K_P},I}(v),\ v^{(\hat{P},\hat{P^0}-\hat{P^1})}\zeta_{I,I_{K_P}}(v)\in \cZ$ take the same values in infinitely many $v=v_q^n$ for $n\geqslant 1$, which implies that they are equal to each other, and  this finish the proof of formula (\ref{equivalent Ore}). Similarly, we can prove $I^*_{K_P^*}*_rx=v^{-(\hat{P},\hat{P^0}-\hat{P^1})}x*_rI^*_{K_P^*}$. The formulas (\ref{commutative}) follows (\ref{I-Ore}) and (\ref{I-Ore'}) directly.
\end{proof}

\begin{theorem}\label{Bridgeland's algebra via perverse sheaf}
The subset 
$\{\tilde{I}_{K_P}^*, \tilde{I}_{K_P^*}^*|P\in \cP\}\subset \cK^{ss,*}$
satisfies the Ore conditions, and so there is a well-defined localization
$$\cD\cK^{ss,*}=\cK^{ss,*}[(\tilde{I}_{K_P}^*)^{-1}, (\tilde{I}_{K_P^*}^*)^{-1}|P\in \cP]$$
with a reduced quotient 
$$\cD\cK^{ss,*,\red}=\cD\cK^{ss,*}/\langle \tilde{I}_{K_P}^**_r\tilde{I}_{K_P^*}^*-1|P\in \cP\rangle$$
such that the algebra isomorphisms (\ref{three algebra isomorphism}) induce algebra isomorphisms
\begin{align*}
\cD\cH(\cA_q)\cong \cD\tilde{\cH}^*(&\cA_q)\cong \bbC\otimes_{\cZ}\cD\cK^{ss,*},\\
\cD\cH^{\red}(\cA_q)\cong \cD\tilde{\cH}^{*,\red}(&\cA_q)\cong \bbC\otimes_{\cZ}\cD\cK^{ss,*,\red}.
\end{align*}
\end{theorem}
\begin{proof}
By Proposition \ref{Ore_2}, the subset 
$\{\tilde{I}_{K_P}^*, \tilde{I}_{K_P^*}^*|P\in \cP\}$
satisfies the Ore conditions. By Lemma \ref{image of b}, the algebra isomorphisms (\ref{three algebra isomorphism}) identify the subsets 
$$\{b_{K_P},b_{K_P^*}|P\in \cP\},\ \{1_{\cO_{K_P}}^*,1_{\cO_{K_P^*}}^*|P\in \cP\},\ \{1\otimes \tilde{I}_{K_P}^*, 1\otimes\tilde{I}_{K_P^*}^*|P\in \cP\},$$ and so they induce algebra an isomorphism between the localizations and an isomorphism between the reduced quotients.
\end{proof}

Replacing the Frobenius morphism $\sigma$ by its power $\sigma^n$ for $n\geqslant 1$, we can obtain similar results over the finite field $\bbF_{q^n}$. This completes a realization of Bridgeland's Hall algebra via perverse sheaves.

\section{Bases}\label{Basis}

\subsection{Bases of $\cK^{ss}$ and $\cK^{ss,*}$}

\begin{definition}\label{bar-involution}
(a) The bar-involution on the Laurent polynomial ring $\cZ$ is defined to be the $\bbZ$-linear isomorphism interchanging $v$ and $v^{-1}$, denoted by $\overline{\zeta(v)}=\zeta(v^{-1})$.\\
(b) The bar-involution on $\cK^{ss}$ is induced by the Verdier dual $\bbD$, denoted by $\bar{[L]}=[\bbD L]$, which is compatible with the $\cZ$-module structure and the bar-involution on $\cZ$.\\
(c) The bar-involution on $\cK^{ss,*}$ is induced by the bar-involutions on $\cZ$ and $\cK^{ss}$, that is, for any $f\in \cK^{ss,*},x\in \cK^{ss}$, we have 
$\bar{f}(x)=\overline{f(\bar{x})}$.
\end{definition}

\begin{lemma}\label{bar}
For any $x\in \cK^{ss}$, if $r(x)=\sum x_1\otimes x_2$, where $x_1,x_2$ are homogeneous of degree $|x_1|,|x_2|$, we have
$r(\bar{x})=\sum v^{-\Vert|x_1|,|x_2|\Vert}\overline{x_2}\otimes \overline{x_1}$. 
Dually, for any $y_1,y_2\in \cK^{ss,*}$ which are homogeneous of degree $|y_1|,|y_2|$, we have 
$\overline{y_1*_ry_2}=v^{\Vert|y_1|,|y_2|\Vert}\overline{y_2}*_r\overline{y_1}$.
\end{lemma}
\begin{proof}
The first statement follows from Corollary \ref{Verdier} directly. We prove the dual statement. For any $x\in \cK^{ss}$, suppose $r(x)=\sum x_1\otimes x_2$, where $x_1,x_2$ are homogeneous, we have
\begin{align*}
&\overline{y_1*_ry_2}(x)=\overline{y_1*_ry_2(\bar{x})}=\overline{(y_1\otimes y_2)(r(\bar{x})})\\
=&\sum v^{\Vert|x_1|,|x_2|\Vert}\overline{y_1(\overline{x_2})y_2(\overline{x_1})}=v^{\Vert|y_1|,|y_2|\Vert}\sum \overline{y_1(\overline{x_2})y_2(\overline{x_1})},\\
&v^{\Vert|y_1|,|y_2|\Vert}\overline{y_2}*_r\overline{y_1}(x)=v^{\Vert|y_1|,|y_2|\Vert}(\overline{y_2}\otimes\overline{y_1})(r(x))\\
=&v^{\Vert|y_1|,|y_2|\Vert}\sum\overline{y_2}(x_1)\overline{y_1}(x_2)=v^{\Vert|y_1|,|y_2|\Vert}\sum \overline{y_2(\overline{x_1})}\overline{y_1(\overline{x_2})},
\end{align*}
where we use $y_1(\overline{x_2})y_2(\overline{x_1})=0$ unless $|x_1|=|y_2|,|x_2|=|y_1|$, as desired.
\end{proof}

\begin{proposition}\label{positivity}
The $\cZ$-basis $\cI$ of $\cK^{ss}$ is bar-invariant and it has positivity. Dually, the $\cZ$-basis $\cI^*$ of $\cK^{ss,*}$ is bar-invariant and it has positivity. More precisely, we have $\overline{I_{M_{\bullet}}}=I_{M_{\bullet}}, \overline{I^*_{M_{\bullet}}}=I^*_{M_{\bullet}}$, and if 
\begin{align*}
r(I_{L_\bullet})=\sum_{M_\bullet, N_\bullet} \zeta^{L_\bullet}_{M_\bullet,N_\bullet}(v)I_{M_\bullet}\otimes I_{N_\bullet},\ I^*_{M_{\bullet}}*_rI^*_{N_{\bullet}}=\sum_{L_{\bullet}}\xi^{L_{\bullet}}_{M_{\bullet},N_{\bullet}}(v)I^*_{L_{\bullet}}
\end{align*}
then $\zeta^{L_{\bullet}}_{M_{\bullet},N_{\bullet}}(v), \xi^{L_{\bullet}}_{M_{\bullet}N_{\bullet}}(v)\in \bbN[v,v^{-1}]$. Moreover, we have 
$$\xi^{L_{\bullet}}_{M_\bullet N_\bullet}(v)=\zeta^{L_{\bullet}}_{M_\bullet N_\bullet}(v)=v^{-\Vert\ue_{M_\bullet},\ue_{N_\bullet}\Vert}\zeta^{L_{\bullet}}_{N_\bullet M_\bullet}(v^{-1}).$$
\end{proposition}
\begin{proof}
By \cite[Lemma 3.3.13]{Pramod-2021}, we have 
\begin{align*}
&\bbD\IC(\cO_{M_\bullet},\overline{\bbQ}_l)\cong \IC(\cO_{M_\bullet},\overline{\bbQ}_l)(\dim \cO_{M_\bullet}),\\ 
\bbD&(\IC_{M_\bullet})\cong \bbD\IC(\cO_{M_\bullet},\overline{\bbQ}_l)(-\frac{\dim \cO_{M_{\bullet}}}{2})\cong \IC_{M_\bullet},
\end{align*}
thus $\overline{I_{M_{\bullet}}}=I_{M_{\bullet}}$ and then $\overline{I^*_{M_{\bullet}}}=I^*_{M_{\bullet}}$, since $$\overline{I^*_{M_{\bullet}}}(I_{N_{\bullet}})=\overline{I^*_{M_{\bullet}}(\overline{I_{N_{\bullet}}})}=\overline{I^*_{M_{\bullet}}(I_{N_{\bullet}})}=\delta_{M_\bullet, N_\bullet}.$$ The positivity of $\cI$ follows from our definition in the framework of perverse sheaves, the coefficients of $\zeta^{L_{\bullet}}_{M_\bullet N_\bullet}(v)\in\cZ$ are dimensions of some $\overline{\bbQ}_l$-vector spaces. By definition, 
\begin{align*}
&\xi^{L_{\bullet}}_{M_\bullet, N_\bullet}(v)=I^*_{M_{\bullet}}*_rI^*_{N_{\bullet}}(I_{L_{\bullet}})=(I^*_{M_{\bullet}}\otimes I^*_{N_{\bullet}})(r(I_{L_{\bullet}}))\\
=&(I^*_{M_{\bullet}}\otimes I^*_{N_{\bullet}})(\sum_{M_\bullet,N_\bullet}\zeta^{L_{\bullet}}_{M_\bullet, N_\bullet}(v)I_{M_{\bullet}}\otimes I_{N_{\bullet}})=\zeta^{L_{\bullet}}_{M_\bullet, N_\bullet}(v)\in \bbN[v,v^{-1}].
\end{align*}
By $r(I_{L_\bullet})=r(\overline{I_{L_\bullet}})$ and Lemma \ref{bar}, we have 
$$\sum_{M_\bullet,N_\bullet}\zeta^{L_{\bullet}}_{M_\bullet N_\bullet}(v)I_{M_{\bullet}}\otimes I_{N_{\bullet}}=\sum_{M_\bullet,N_\bullet}v^{-\Vert\ue_{M_\bullet},\ue_{N_\bullet}\Vert}\overline{\zeta^{L_{\bullet}}_{M_\bullet N_\bullet}(v)}I_{N_{\bullet}}\otimes I_{M_{\bullet}},$$
and so $\zeta^{L_{\bullet}}_{M_\bullet N_\bullet}(v)=v^{-\Vert\ue_{M_\bullet},\ue_{N_\bullet}\Vert}\overline{\zeta^{L_{\bullet}}_{N_\bullet M_\bullet}(v)}=v^{-\Vert\ue_{M_\bullet},\ue_{N_\bullet}\Vert}\zeta^{L_{\bullet}}_{N_\bullet M_\bullet}(v^{-1})$.
\end{proof}

\subsection{Bases of $\cD\cK^{ss,*}$ and $\cD\cK^{ss,*,\red}$}

For any $x\in \cK^{ss,*}$, we still denote by $x$ its images in the localization $\cD\cK^{ss,*}$ and the reduced quotient $\cD\cK^{ss,*,\red}$.

\begin{lemma}\label{maximal orbit induction}
For any $P,Q\in \cP$ and $\cO_{M_\bullet}\subset \rC_{\ue}$, if the orbit $\cO_{K_P\oplus K_Q^*\oplus M_\bullet}\subset \rC_{\ue_{K_P\oplus K_Q^*}+\ue}$ is maximal with respect to the partial order $\preccurlyeq$, then $\cO_{M_\bullet}\subset\rC_{\ue}$ is also maximal. Moreover, we have 
\begin{equation}
\begin{aligned}\label{maximal orbit formula}
\tilde{I}^*_{K_P\oplus K_Q^*}*_rv^{-\dim \cO_{M_\bullet}}I^*_{M_\bullet}=v^{\langle \hat{P}-\hat{Q},\hat{M^0}-\hat{M^1}\rangle}v^{-\dim \cO_{K_P\oplus K_Q^*\oplus M_\bullet}}I^*_{K_P\oplus K_Q^*\oplus M_\bullet},
\end{aligned}
\end{equation}
in particular, we have 
\begin{align*}
&\tilde{I}_{K_P}^**_r\tilde{I}_{K_Q}^*=\tilde{I}_{K_P\oplus K_Q}^*,\ \tilde{I}_{K_P^*}^**_r\tilde{I}_{K_Q^*}^*=\tilde{I}_{K_P^*\oplus K_Q^*}^*,\\ 
&\tilde{I}_{K_P}^**_r\tilde{I}_{K_Q^*}^*=\tilde{I}_{K_Q^*}^**_r\tilde{I}_{K_P}^*=\tilde{I}_{K_P\oplus K_Q^*}^*,
\end{align*}
in $\cK^{ss,*}$, and also in $\cD\cK^{ss,*},\cD\cK^{ss,*,\red}$.
\end{lemma}
\begin{proof}
Assume there exists another orbit $\cO_{M'_\bullet}\not=\cO_{M_\bullet}$ such that $\cO_{M_\bullet}\subset \overline{\cO_{M'_\bullet}}$, then it is clear that $\cO_{K_P\oplus K_Q^*\oplus M_\bullet}\subset \overline{\cO_{K_P\oplus K_Q^*\oplus M'_\bullet}}$, which is a contradiction to the maximality of $\cO_{K_P\oplus K_Q^*\oplus M_\bullet}$. For any $I\in \cI_{\ue_{K_P\oplus K_Q^*}+\ue}$, suppose 
\begin{align*}
(\tilde{I}_{K_P\oplus K_Q^*}^**_rv^{-\dim \cO_{M_\bullet}}I^*_{M_\bullet})(I)&=\zeta_1(v),\\
v^{-\dim \cO_{K_P\oplus K_Q^*\oplus M_\bullet}}I^*_{K_P\oplus K_Q^*\oplus M_\bullet}(I)&=\zeta_2(v).
\end{align*}
By Lemma \ref{image of b}, the elements 
$$1\otimes \tilde{I}_{K_P\oplus K_Q^*}^*,\ 1\otimes v^{-\dim \cO_{M_\bullet}}I_{M_\bullet}^*,\ 1\otimes v^{-\dim \cO_{K_P\oplus K_Q^*\oplus M_\bullet}}I^*_{K_P\oplus K_Q^*\oplus M_\bullet}$$ are the images of $$b_{K_P}*b_{K_Q^*}=a_{K_P\oplus K_Q^*}u_{[K_P\oplus K_Q^*]},\ a_{M_\bullet}u_{[M_\bullet]},\ a_{K_P\oplus K_Q^*\oplus M_\bullet}u_{[K_P\oplus K_Q^*\oplus M_\bullet]}$$ respectively, under the algebra isomorphisms (\ref{three algebra isomorphism}). In the algebra $\cH_q^{\tw}(\cC_2(\cP))$, we have 
$$(b_{K_P}*b_{K_Q^*})*(a_{M_\bullet}u_{[M_\bullet]})=v_q^{\langle \hat{P}-\hat{Q},\hat{M^0}-\hat{M^1}\rangle}(a_{K_P\oplus K_Q^*\oplus M_\bullet}u_{[K_P\oplus K_Q^*\oplus M_\bullet]}),$$
see \cite[Lemma 3.4]{Bridgeland-2013}. Hence the formula (\ref{maximal orbit formula}) holds for $v=v_q$ in $\bbC\otimes_{\cZ}\cK^{ss,*}$. Replacing the Frobenius morphism $\sigma$ by its power $\sigma^n$ for $n\geqslant 1$, the formula (\ref{maximal orbit formula}) holds for $v=v_q^n$ in $\bbC\otimes_{\cZ}\cK^{ss,*}$, that is, 
\begin{align*}
1\otimes (\tilde{I}_{K_P\oplus K_Q^*}^**_rv^{-\dim \cO_{M_\bullet}}I^*_{M_\bullet})
=(v_q^n)^{\langle \hat{P}-\hat{Q},\hat{M^0}-\hat{M^1}\rangle}\otimes v^{-\dim \cO_{K_P\oplus K_Q^*\oplus M_\bullet}}I^*_{K_P\oplus K_Q^*\oplus M_\bullet}.
\end{align*}
Taking their values at $I$, we obtain 
$$\zeta_1(v_q^n)=(v_q^n)^{\langle \hat{P}-\hat{Q},\hat{M^0}-\hat{M^1}\rangle}\zeta_2(v_q^n).$$
Hence the Laurent polynomials $\zeta_1(v),\ v^{\langle \hat{P}-\hat{Q},\hat{M^0}-\hat{M^1}\rangle}\zeta_2(v)\in \cZ$ take the same values in infinitely many $v=v_q^n$ for $n\geqslant 1$, which implies that they are equal to each other, and this finish the proof of the formula (\ref{maximal orbit formula}). The other statements follow the formula (\ref{maximal orbit formula}) and Lemma \ref{image of b} directly.
\end{proof}

\begin{definition}
(a) We define $\cT$ to be the $\cZ$-subalgebra of $\cD\cK^{ss,*}$ generated by $(\tilde{I}_{K_{P_i}}^*)^{\pm 1},(\tilde{I}_{K^*_{P_i}}^*)^{\pm 1}$ for $i\in I$, which is a commutative algebra. Then the multiplication $$*_r:\cT\times \cD\cK^{ss,*}\rightarrow \cD\cK^{ss,*}$$ 
defines a $\cT$-module structure on $\cD\cK^{ss,*}$.\\
(a) We define $\cT^{\red}$ to be the $\cZ$-subalgebra of $\cD\cK^{ss,*,\red}$ generated by $\tilde{I}_{K_{P_i}}^*,\tilde{I}_{K^*_{P_i}}^*$ for $i\in I$, which is a commutative algebra. Then the multiplication $$*_r:\cT^{\red}\times \cD\cK^{ss,*,\red}\rightarrow \cD\cK^{ss,*,\red}$$ 
defines a $\cT^{\red}$-module structure on $\cD\cK^{ss,*,\red}$.\\
(c) A complex $M_\bullet=(M^1,M^0,d^1,d^0)\in \cC_2(\cP)$ is said to be radical, if $\im d^j\subset \mbox{\rm{rad}\ } {M^{j+1}}$. We define
$$\cI^{*,\rad}=\{I^*_{M_\bullet}| M_\bullet\ \textrm{is radical} \}.$$
\end{definition}

\begin{lemma}\label{basis of torus}
(a) The algebra $\cT$ has a $\cZ$-basis $\{(\tilde{I}^*_{K_P\oplus K_Q^*})^{\pm 1}|P,Q\in \cP\}$.\\
(b) For any $\alpha\in K(\cA)$, it can be written as $\alpha=\hat{P}-\hat{Q}$ for some $P,Q\in \cP$, and there is a well-defined element 
$$\tilde{I}^*_{\alpha}=\tilde{I}^*_{K_P\oplus K_Q^*}=\tilde{I}^*_{K_P}*_r\tilde{I}^*_{K_Q^*}\in \cT^{\red}.$$
Moreover, there is an algebra isomorphism $\cT^{\red}\cong\cZ[K(\cA)]$, where $\cZ[K(\cA)]$ is the group algebra of $K(\cA)$ over $\cZ$, and $\cT^{\red}$ has a $\cZ$-basis $\{\tilde{I}^*_{\alpha}|\alpha\in K(\cA)\}$.
\end{lemma}
\begin{proof}
(a) By Lemma \ref{maximal orbit induction}, $\cT$ can be spanned by $(\tilde{I}^*_{K_P\oplus K_Q^*})^{\pm 1}$ for $P,Q\in \cP$. It is clear that these elements are linearly independent over $\cZ$, and so they form a $\cZ$-basis of $\cT$.

(b) For any representation $M\in \cA$, let $0\rightarrow Q\rightarrow P\rightarrow M\rightarrow 0$ be a projective resolution, then $\hat{M}=\hat{P}-\hat{Q}$. So any element $\alpha\in K(\cA)$ can be written as $\alpha=\hat{P}-\hat{Q}$ for some $P,Q\in \cP$. Suppose $\alpha=\hat{P}-\hat{Q}=\hat{P'}-\hat{Q'}$ for another $P',Q'\in \cP$, then we have $\hat{P}+\hat{Q'}=\hat{P'}+\hat{Q}$ and $P\oplus Q'\cong P'\oplus Q$, $\tilde{I}^*_{K_P\oplus K_{Q'}}=\tilde{I}^*_{K_{P'}\oplus K_Q}$. By Lemma \ref{maximal orbit induction} and Proposition \ref{Ore_2}, we obtain 
$$\tilde{I}^*_{K_P\oplus K_{Q'}}*_r\tilde{I}_{K_Q^*\oplus K_{Q'}^*}=\tilde{I}^*_{K_{P'}\oplus K_Q}*_r\tilde{I}_{K_Q^*\oplus K_{Q'}^*},\ \tilde{I}^*_{K_P}*_r\tilde{I}^*_{K_Q^*}=\tilde{I}^*_{K_{P'}}*_r\tilde{I}^*_{K_{Q'}^*}.$$
So $\tilde{I}^*_{\alpha}$ is well-defined. Moreover, $\tilde{\cI}^*_{K_P\oplus K_Q^*}\mapsto \hat{P}-\hat{Q}$ induces a surjective algebra homomorphism $\cT\rightarrow\cZ[K(\cA)]$. Notice that this homomorphism factors through $\cT^{\red}\rightarrow\cZ[K(\cA)]$ which has an inverse induced by $\alpha\mapsto \tilde{I}^*_{\alpha}$. Therefore, $\cT^{\red}\cong \cZ[K(\cA)]$ and $\cT^{\red}$ has a $\cZ$-basis $\{\tilde{I}^*_{\alpha}|\alpha\in K(\cA)\}$.
\end{proof}

\begin{lemma}\label{inductive lemma}
For any $\ue=\ue'+\ue''$ and $\cO_{L_\bullet}\subset\rC_{\ue},\ \cO_{M_\bullet}\subset \rC_{\ue'},\ \cO_{N_\bullet}\subset \rC_{\ue''}$ satisfying $\Ext^1_{\cC_2(\cP)}(M_\bullet, N_\bullet)=0$, if $\IC_{M_\bullet}\boxtimes \IC_{N_\bullet}$ is a direct summand of $\Res^{\ue}_{\ue',\ue''}(\IC_{L_\bullet})$ up to shifts and Tate twists, then $\cO_{M_\bullet\oplus N_\bullet}\subset \overline{\cO_{L_\bullet}}$. 
\end{lemma}
\begin{proof}
Since the supports of $\IC_{L_\bullet}$ and $\IC_{M_\bullet}\boxtimes \IC_{N_\bullet}$ are $\overline{\cO_{L_\bullet}}$ and $\overline{\cO_{M_\bullet}}\times \overline{\cO_{N_\bullet}}$ respectively, we have 
$\overline{\cO_{M_\bullet}}\times \overline{\cO_{N_\bullet}}\subset\kappa(\overline{\cO_{L_\bullet}}\cap \rF)$, and in particular, 
$$\cO_{M_\bullet}\times \cO_{N_\bullet}\subset \kappa(\overline{\cO_{L_\bullet}}\cap \rF),$$
where $\rF,\kappa$ appear in the definition of $\Res^{\ue}_{\ue',\ue''}$, see \S \ref{Restriction functor}. If there exists an orbit $\cO_{L'_\bullet}\subset \rC_{\ue}$ such that $\cO_{M_\bullet}\times \cO_{N_\bullet}\subset\kappa(\cO_{L'_\bullet}\cap \rF)$, then $L'_\bullet$ has a subobject isomorphic to $N_\bullet$ with quotient isomorphic to $M_\bullet$. Since $\Ext^1_{\cC_2(\cP)}(M_\bullet,N_\bullet)=0$, we know that $\cO_{M_\bullet\oplus N_\bullet}\subset \rC_{\ue}$ is the unique orbit such that $\cO_{M_\bullet}\times \cO_{N_\bullet}\subset\kappa(\cO_{M_\bullet\oplus N_\bullet}\cap \rF)$. Therefore, $\cO_{M_\bullet\oplus N_\bullet}\subset \overline{\cO_{L_\bullet}}$.
\end{proof}

\begin{theorem}\label{free over torus}
(a) The $\cT$-module $\cD\cK^{ss,*}$ is free with a $\cT$-basis $\cI^{*,\rad}$.\\
(b) The $\cT^{\red}$-module $\cD\cK^{ss,*,\red}$ is free with a $\cT^{\red}$-basis $\cI^{*,\rad}$
\end{theorem}
\begin{proof}
(a) We denote by $\cD\cK'$ the $\cT$-submodule of $\cD\cK^{ss,*}$ generated by $\cI^{*,\rad}$, and make an induction to prove $\cI^*_{M_\bullet}\in \cD\cK'$ for any $\cO_{M_\bullet}\subset \rC_{\ue}$. By \cite[Lemma 4.2]{Bridgeland-2013} or \cite[Lemma 2.3]{Fang-Lan-Xiao-2024}, the object $M_\bullet$ can be decomposed as $M_\bullet\cong K_P\oplus K_Q^*\oplus M'_\bullet$, where $P,Q\in \cP$ and $M'_\bullet$ is radical. If $\cO_{M_\bullet}\subset\rC_{\ue}$ is a maximal orbit, by Lemma \ref{maximal orbit induction}, we have 
$$\tilde{I}^*_{K_P\oplus K_Q^*}*_rv^{-\dim\cO_{M'_\bullet}}I^*_{M'_\bullet}=v^{\langle \hat{P}-\hat{Q}, \hat{M'^0}-\hat{M'^1} \rangle}v^{-\dim \cO_{M_\bullet}}I^*_{M_\bullet}\in \cD\cK',$$
and then $I^*_{M_\bullet}\in \cD\cK'$, since $v^{\langle \hat{P}-\hat{Q}, \hat{M'^0}-\hat{M'^1} \rangle}v^{-\dim \cO_{M_\bullet}}$ is invertible in $\cZ$.

We inductively assume that $I^*_{N_\bullet}\in \cD\cK'$ for any other orbits $\cO_{N_\bullet}\subset\rC_{\ue}$ satisfying $\cO_{M_\bullet}\preccurlyeq \cO_{N_\bullet}$. Consider the element $\tilde{I}^*_{K_P\oplus K_Q^*}*_rI^*_{M'_\bullet}\in \cD\cK'$ and express it into a $\cZ$-linear combination under the basis $\cI^*_{\ue}$ as follows
\begin{align*}
\tilde{I}^*_{K_P\oplus K_Q^*}*_rI^*_{M'_\bullet}=\sum_{I_{N_\bullet}^*\in \cI^*_{\ue}}v^{-\langle\hat{P}+\hat{Q},\hat{P}+\hat{Q}\rangle}\zeta^{N_\bullet}_{K_P\oplus K_Q^*,M'_\bullet}(v) I_{N_\bullet}^*,
\end{align*}
where $\zeta^{N_\bullet}_{K_P\oplus K_Q^*,M'_\bullet}(v)\in \bbN[v,v^{-1}]$ is given by
\begin{align*}
\zeta^{N_\bullet}_{K_P\oplus K_Q^*,M'_\bullet}(v)=(I^*_{K_P\oplus K_Q^*}*_rI^*_{M'_\bullet})(I_{N_\bullet})
=(I^*_{K_P\oplus K_Q^*}\otimes I^*_{M'_\bullet})(r(I_{N_\bullet})).
\end{align*}
If $\zeta^{N_\bullet}_{K_P\oplus K_Q^*,M'_\bullet}(v)\not=0$, then $\IC_{K_P\oplus K_Q^*}\boxtimes \IC_{M'_\bullet}$, up to shifts and Tate twists, is a direct summand of 
$$\Res^{\ue}_{\ue',\ue''}(\IC_{N_\bullet})=\kappa_!\iota^*(\IC_{N_\bullet})[|\ue',\ue''|](\frac{|\ue',\ue''|}{2}),$$
where $\ue'=\ue_{K_P\oplus K_Q^*},\ \ue''=\ue_{M'_\bullet}$ and $|\ue',\ue''|=\langle\hat{P}+\hat{Q},\hat{M'^1}\rangle+\langle\hat{P}+\hat{Q},\hat{M'^0}\rangle$, see \S \ref{Restriction functor}. Since $\Ext^1_{\cC_2(\cP)}(K_P\oplus K_Q^*,M'_\bullet)=0$, by Lemma \ref{inductive lemma}, we have $\cO_{M_\bullet}\preccurlyeq \cO_{N_\bullet}$. Hence 
\begin{equation}
\begin{aligned}\label{I_M^* in DK' up to v power}
&\zeta^{M_\bullet}_{K_P\oplus K_Q^*,M'_\bullet}(v)I^*_{M_\bullet}\\=&v^{\langle\hat{P}+\hat{Q},\hat{P}+\hat{Q}\rangle}\tilde{I}^*_{K_P\oplus K_Q^*}*_rI^*_{M'_\bullet}-\sum\zeta^{N_\bullet}_{K_P\oplus K_Q^*,M'_\bullet}(v)I^*_{N_\bullet}\in \cD\cK',
\end{aligned}
\end{equation}
where the sum is taken over $I_{N_\bullet}^*\in \cI^*_{\ue}$ satisfying 
$\cO_{M_\bullet}\not=\cO_{N_\bullet},\cO_{M_\bullet}\preccurlyeq \cO_{N_\bullet}$, and these $I_{N_\bullet}^*\in \cD\cK'$ have been proved by the inductive hypothesis. 

We calculate $\zeta^{M_\bullet}_{K_P\oplus K_Q^*,M'_\bullet}(v)$ explicitly to prove that it is invertible in $\cZ$. Consider the restriction of $\Res^{\ue}_{\ue',\ue''}(\IC_{M_\bullet})$ to $\cO_{K_P\oplus K_Q^*}\times \cO_{M'_\bullet}$. For any direct summand $\IC_{S_\bullet}\boxtimes \IC_{T_\bullet}$ of $\Res^{\ue}_{\ue',\ue''}(\IC_{M_\bullet})$ up to shifts and Tate twists, if it has non-zero restriction on $\cO_{K_P\oplus K_Q^*}\times \cO_{M'_\bullet}$, we have $\cO_{K_P\oplus K_Q^*}\subset \overline{\cO_{S_\bullet}},\cO_{M'_\bullet}\subset \overline{\cO_{T_\bullet}}$, then $\cO_{K_P\oplus K_Q^*}=\cO_{S_\bullet}$ by the maximality. On the one hand, Lemma \ref{inductive lemma}, we have $\cO_{K_P\oplus K_Q^*\oplus T_\bullet}\subset \overline{\cO_{M_\bullet}}$; on the other hand, by $\cO_{M'_\bullet}\subset \overline{\cO_{T_\bullet}}$, we have $\cO_{M_\bullet}\subset\overline{\cO_{K_P\oplus K_Q^*\oplus T_\bullet}}$. Hence $\cO_{M_\bullet}=\cO_{K_P\oplus K_Q^*\oplus T_\bullet}, \cO_{M'_\bullet}=\cO_{T_\bullet}$, and $\IC_{K_P\oplus K_Q^*}\boxtimes \IC_{M'_\bullet}$ is the unique direct summand, up to shifts and Tate twists, which has non-zero restriction on $\cO_{K_P\oplus K_Q^*}\times \cO_{M'_\bullet}$. Consider the following commutative diagram
$$\xymatrix{\cO_{K_P\oplus K_Q^*}\times \cO_{M'_\bullet} \ar[d]_-{j_{K_P\oplus K_Q^*}\times j_{M'_\bullet}} &\tilde{\rF} \ar[d]_-{\tilde{j}} \ar[l]_-{\tilde{\kappa}} \ar[r]^-{\iota\tilde{j}} \ar@{}[dl]|{\square}&\rC_{\ue} \ar@{=}[d]\\ \rC_{\ue'}\times \rC_{\ue''} &\rF \ar[l]_-{\kappa} \ar[r]^-{\iota} &\rC_{\ue},}$$
where $j_{K_P\oplus K_Q^*}\times j_{M'_\bullet}$ is the inclusion, $\tilde{\rF}=\kappa^{-1}(\cO_{K_P\oplus K_Q^*}\times \cO_{M'_\bullet})$ and $\tilde{\kappa},\tilde{j}$ are natural morphisms. By Lemma \ref{fibre of kappa}, $\tilde{\kappa}$ is a vector bundle of rank 
$$\eta=\dim \Hom_{\cC_2(\cP)}(K_P\oplus K_Q^*,M'^*_{\bullet})=\langle\hat{P},\hat{M'^0}\rangle+\langle\hat{Q},\hat{M'^1}\rangle.$$
Since $\Ext^1_{\cC_2(\cP)}(K_P\oplus K_Q^*,M'_\bullet)=0$, we know that $\cO_{M_\bullet}\subset\rC_{\ue}$ is the unique orbit such that $\cO_{K_P\oplus K_Q^*}\times \cO_{M'_\bullet}\subset\kappa(\cO_{M_\bullet}\cap \rF)$, and so $\tilde{\rF}=\cO_{M_\bullet}\cap \rF$. Hence the restriction of $\Res^{\ue}_{\ue',\ue''}(\IC_{M_\bullet})$ on $\cO_{K_P\oplus K_Q^*}\times \cO_{M'_\bullet}$ is 
\begin{align*}
&(j_{K_P\oplus K_Q^*}\times j_{M'_\bullet})^*\kappa_!\iota^*(\IC_{M_\bullet})[|\ue',\ue''|](\frac{|\ue',\ue''|}{2})\\
\cong &\tilde{\kappa}_!\tilde{j}^*\iota^*(\IC_{M_\bullet})[|\ue',\ue''|](\frac{|\ue',\ue''|}{2})\\
\cong &\tilde{\kappa}_!(\overline{\bbQ}_l|_{\cO_{M_\bullet}\cap\rF})[\dim \cO_{M_\bullet}+|\ue',\ue''|](\frac{\dim \cO_{M_\bullet}+|\ue',\ue''|}{2})\\
\cong &(\overline{\bbQ}_l|_{\cO_{K_P\oplus K_Q^*}}\boxtimes \overline{\bbQ}_l|_{\cO_{M'_\bullet}})[\dim \cO_{M_\bullet}+|\ue',\ue''|-2\eta](\frac{\dim \cO_{M_\bullet}+|\ue',\ue''|-2\eta}{2}).
\end{align*}
Comparing with the restriction of $\IC_{K_P\oplus K_Q^*}\boxtimes \IC_{M'_\bullet}$ on $\cO_{K_P\oplus K_Q^*}\times \cO_{M'_\bullet}$,  
$$(\overline{\bbQ}_l|_{\cO_{K_P\oplus K_Q^*}}\boxtimes \overline{\bbQ}_l|_{\cO_{M'_\bullet}})[\dim \cO_{K_P\oplus K_Q^*}+\dim \cO_{M'_\bullet}](\frac{\dim \cO_{K_P\oplus K_Q^*}+\dim \cO_{M'_\bullet}}{2}),$$
we obtain 
$$v^{-(\dim \cO_{M_\bullet}+|\ue',\ue''|-2\eta)}=v^{-(\dim \cO_{K_P\oplus K_Q^*}+\dim \cO_{M'_\bullet})}\zeta^{M_\bullet}_{K_P\oplus K_Q^*,M'_\bullet}(v),$$
and so $\zeta^{M_\bullet}_{K_P\oplus K_Q^*,M'_\bullet}(v)$ is a power of $v$, and is invertible in $\cZ$, which implies that $I_{M_\bullet}^*\in \cD\cK'$ from the formula (\ref{I_M^* in DK' up to v power}). Thus $\cD\cK^{ss,*}=\cD\cK'$ is generated by $\cI^{*,\rad}$ as a $\cT$-module. In the meantime, above argument also proves that $\cD\cK^{ss,*,\red}$ is generated by $\cI^{*,\rad}$ as a $\cT^{\red}$-module.

It remains to prove that elements in $\cI^{*,\rad}$ are linearly independent over $\cT$. For any $\sum_{s=1}^nt_s*_rI^*_{M_{s\bullet}}=0$ in $\cD\cK^{ss,*}$, where $t_s\in \cT$ and $I^*_{M_{s\bullet}}\in \cI^{*,\rad}$, we need to prove all $t_s=0$. By Lemma \ref{basis of torus}, we can suppose that
$$t_s=\sum_{(P,Q)\in \cP_s}\xi_{s,P,Q}(v)(\tilde{I}^*_{K_P\oplus K_Q^*})^{\pm 1},$$
where $\cP_s\subset \cP\times \cP$ and $\xi_{s,P,Q}(v)\in \cZ$. Replacing $\sum_{s=1}^nt_s*_rI^*_{M_{s\bullet}}=0$ by $\tilde{I}^*_{K_{P'}\oplus K_{Q'}^*}*_r\sum_{s=1}^nt_s*_rI^*_{M_{s\bullet}}=0$ for some $P',Q'\in \cP$, we may assume that there is no any $(\tilde{I}^*_{K_P\oplus K_Q^*})^{-1}$ appearing in $t_s$. Since $\cD\cK^{ss,*}$ is $\bbN[I]\times \bbN[I]$-graded, without loss of generality, we may assume that $\ue_{K_P\oplus K_Q^*}+\ue_{M_{s\bullet}}=\ue$ are constant for various $(P,Q)\in \cP_s,s=1,...,n$.
By the formula (\ref{I_M^* in DK' up to v power}), we have 
$$\tilde{I}^*_{K_P\oplus K_Q^*}*_rI^*_{M_{s\bullet}}=\sum_{\cO_{N_\bullet}\in T_{s,P,Q}} v^{-\langle\hat{P}+\hat{Q},\hat{P}+\hat{Q}\rangle}\zeta^{N_\bullet}_{K_P\oplus K_Q^*,M_{s\bullet}}(v)I^*_{N_\bullet},$$
where $T_{s,P,Q}=\{\cO_{N_\bullet}\subset \rC_{\ue}|\cO_{K_P\oplus K_Q^*\oplus M_{s\bullet}}\preccurlyeq \cO_{N_\bullet}\}$, and so
\begin{equation}
\begin{aligned}\label{sum=0}
\sum^n_{s=1} \sum_{(P,Q)\in \cP_s}\xi_{s,P,Q}(v)\sum_{\cO_{N_\bullet}\in T_{s,P,Q}} v^{-\langle\hat{P}+\hat{Q},\hat{P}+\hat{Q}\rangle}\zeta^{N_\bullet}_{K_P\oplus K_Q^*,M_{s\bullet}}(v)I^*_{N_\bullet}=0.
\end{aligned}
\end{equation}
Consider the minimal elements in the finite set
$$\{\cO_{K_P\oplus K_Q^*\oplus M_{s\bullet}}\subset \rC_{\ue}|(P,Q)\in \cP_s,s=1,...,n\},$$
let $\cO_{K_{P'}\oplus K_{Q'}^*\oplus M_{s'\bullet}}$ be a minimal element, then it is also minimal in 
$$\bigsqcup_{s=1}^n\bigsqcup_{(P,Q)\in \cP_s}T_{s,P,Q},$$
and so the coefficient of $I^*_{K_{P'}\oplus K_{Q'}^*\oplus M_{s'\bullet}}$ in the sum (\ref{sum=0}) vanishes, that is, 
$$\xi_{s',P',Q'}(v)\cdot v^{-\langle\hat{P'}+\hat{Q'},\hat{P'}+\hat{Q'}\rangle}\cdot\zeta^{K_{P'}\oplus K_{Q'}^*\oplus M_{s'\bullet}}_{K_{P'}\oplus K_{Q'}^*, M_{s'\bullet}}(v)=0,$$
which implies $\xi_{s',P',Q'}(v)=0$, since $v^{-\langle\hat{P'}+\hat{Q'},\hat{P'}+\hat{Q'}\rangle}\zeta^{K_{P'}\oplus K_{Q'}^*\oplus M_{s'\bullet}}_{K_{P'}\oplus K_{Q'}^*, M_{s'\bullet}}(v)\in \cZ$ is invertible. Then the sum (\ref{sum=0}) becomes 
\begin{align*}
\sum_{(s,P,Q)\not=(s',P',Q')} \xi_{s,P,Q}(v)\sum_{\cO_{N_\bullet}\in T_{s,P,Q}} v^{-\langle\hat{P}+\hat{Q},\hat{P}+\hat{Q}\rangle}\zeta^{N_\bullet}_{K_P\oplus K_Q^*,M_{s\bullet}}(v)I^*_{N_\bullet}=0.
\end{align*}
Similarly, consider the minimal elements in 
$$\{\cO_{K_P\oplus K_Q^*\oplus M_{s\bullet}}\subset \rC_{\ue}|(P,Q)\in \cP_s,s=1,...,n\}\setminus \{\cO_{K_{P'}\oplus K_{Q'}^*\oplus M_{s'\bullet}}\}$$
and repeat above process again, we eventually obtain all $\xi_{s,P,Q}(v)=0$ and $t_s=0$. Therefore, $\cI^{*,\rad}$ is a $\cT$-basis of $\cD\cK^{ss,*}$.

(b) The proof of (a) also proves that $\cD\cK^{ss,*,\red}$ is generated by $\cI^{*,\rad}$ as a $\cT^{\red}$-module. It remains to prove that elements in $\cI^{*,\rad}$ are linearly independent over $\cT^{\red}$. For any $\sum_{s=1}^nt_s*_rI^*_{M_{s\bullet}}=0$ in $\cD\cK^{ss,*,\red}$, where $t_s\in \cT^{\red}$ and $I^*_{M_{s\bullet}}\in \cI^{*,\rad}$, we need to prove all $t_s=0$ in $\cT^{\red}$. Take a lifting $\tilde{t_s}\in \cT$ of $t_s\in \cT^{\red}$, then by definition, $\sum_{s=1}^nt_s*_rI^*_{M_{s\bullet}}=0$ in $\cD\cK^{ss,*,\red}$ is equivalent to
$$\sum^n_{s=1}\tilde{t_s}*_rI^*_{M_{s\bullet}}\in \langle\tilde{I}^*_{K_R}*_r\tilde{I}_{K_R^*}^*-1|R\in \cP\rangle\subset \cD\cK^{ss,*}.$$
By Proposition \ref{Ore_2}, the element $(\tilde{I}_{K_P}*_r\tilde{I}_{K_P^*}^*-1)$ belongs to the center of $\cD\cK^{ss,*}$, so any element in the ideal $\langle\tilde{I}_{K_P}*_r\tilde{I}_{K_P^*}^*-1|P\in \cP\rangle$ can be written as $\sum_{R\in \cP} (\tilde{I}_{K_R}*_r\tilde{I}_{K_R^*}^*-1)*_rL_{R}$ for some $L_R\in \cD\cK^{ss,*}$. By (a), any $L_{R}$ can be written as $L_R=\sum_{(P,Q)\in \cP}\sum_{I^*_{M_{\bullet}}\in\cI^{*,\rad}}\theta_{P,Q,M_\bullet}(v) \tilde{I}^*_{K_{P}\oplus K_{Q}^*}*_rI^*_{M_{\bullet}}$, where $\theta_{P,Q,M_\bullet}(v)\in \cZ$. Hence we can suppose that 
\begin{align*} 
\sum^n_{s=1}\tilde{t_s}*_rI^*_{M_{s\bullet}}
=\!\!\!\!\!\sum_{R,P,Q\in \cP}\sum_{I^*_{M_{\bullet}}\in\cI^{*,\rad}}\!\!\theta_{P,Q,M_\bullet}(v)(\tilde{I}^*_{K_R}*_r\tilde{I}_{K_R^*}^*-1)*_r\tilde{I}^*_{K_{P}\oplus K_{Q}^*}*_rI^*_{M_{\bullet}}
\end{align*}
in $\cD\cK^{ss,*}$. By (a), we have
$$\tilde{t_s}=\sum_{R,P,Q\in \cP}\theta_{P,Q,M_{s\bullet}}(v)(\tilde{I}^*_{K_R}*_r\tilde{I}_{K_R^*}^*-1)*_r\tilde{I}^*_{K_{P}\oplus K_{Q}^*} \in \langle\tilde{I}^*_{K_R}*_r\tilde{I}_{K_R^*}^*-1|R\in \cP\rangle,$$
and so $t_s=0$ in $\cT^{\red}$. Therefore, $\cI^{*,\rad}$ is a $\cT^{\red}$-basis of $\cD\cK^{ss,*,\red}$.
\end{proof}

Note that the basis $\cI^{*,\rad}$ is parametrized by the isomorphism classes of radical complexes. We remark that there is an analogue result for the algebra $\cD\cH(\cA_q)$. More precisely, let $\cT_q$ be the subalgebra of $\cD\cH(\cA_q)$ generated by $b_{K_{P_i}},b_{K_{P_i}}^*$ for $i\in I$, then the multiplication defines a $\cT_q$-module structure on $\cD\cH(\cA_q)$. By \cite[Theorem 3.8]{Gorsky-2018}, $\cD\cH(\cA_q)$ is a free $\cT_q$-module with a basis $\{a_{M_\bullet}u_{[M_\bullet]}|M_\bullet\ \textrm{is radical}\}$. Under the algebra isomorphism $\cD\cH(\cA_q)\cong\bbC\otimes_{\cZ}\cD\cK^{ss,*}$ in Theorem \ref{Bridgeland's algebra via perverse sheaf}, $\cT_q$ can be identified with $\bbC\otimes_{\cZ}\cT$, but for any radical $M_\bullet\in \cC_2(\cP)$, the basis element $a_{M_\bullet}u_{[M_\bullet]}$ is not identified with $1\otimes I^*_{M_\bullet}$ in general. Also, our proof of Theorem \ref{free over torus} is independent of \cite[Theorem 3.8]{Gorsky-2018}.

\begin{definition}
We define 
\begin{align*}
&\tilde{\cI}^*=\{(\tilde{I}^*_{K_P\oplus K_Q^*})^{\pm 1}|P,Q\in \cP\}\subset \cT,\ \tilde{\cI}^{*,\red}=\{\tilde{I}^*_\alpha|\alpha\in K(\cA)\}\subset \cT^{\red}\\
&\tilde{\cI}^**_r\cI^{*,\rad}=\{(\tilde{I}^*_{K_P\oplus K_Q^*})^{\pm 1}*_rI^*_{M_\bullet}|P,Q\in \cP,I^*_{M_\bullet}\in \cI^{*,\rad}\}\subset \cD\cK^{ss,*},\\
&\tilde{\cI}^{*,\red}*_r\cI^{*,\rad}=\{\tilde{I}^*_\alpha*_rI^*_{M_\bullet}|\alpha\in K(\cA),I^*_{M_\bullet}\in \cI^{*,\rad}\}\subset \cD\cK^{ss,*,\red}.
\end{align*}
\end{definition}

\begin{theorem}\label{global basis}
(a) The algebra $\cD\cK^{ss,*}$ has a $\cZ$-basis $\tilde{\cI}^**_r\cI^{*,\rad}$.\\
(b) The algebra $\cD\cK^{ss,*,\red}$ has a $\cZ$-basis $\tilde{\cI}^{*,\red}*_r\cI^{*,\rad}$.
\end{theorem}
\begin{proof}
They follow from Lemma \ref{basis of torus} and Theorem \ref{free over torus} directly. 
\end{proof}

We remark that the basis $\tilde{\cI}^**_r\cI^{*,\rad},\tilde{\cI}^{*,\red}*_r\cI^{*,\rad}$ do not have positivity, because the formula ({\ref{I_M^* in DK' up to v power}}) involve subtraction. We also remark that the bar-involution on $\cK^{ss,*}$, see Definition \ref{bar-involution}, does not induce a bar-involution on $\cD\cK^{ss,*,\red}$, since the ideal $\langle \tilde{I}^*_{K_P}*_r \tilde{I}^*_{K_P^*}-1\mid P\in \cP\rangle$ is not closed under the bar-involution. Hence we are not able to study the bar-invariant property of $\tilde{\cI}^{*,\red}*_r\cI^{*,\rad}$ now. 

\section{Comparison with quantized enveloping algebra}\label{compare with Lusztig}

\subsection{Generic Hall algebra and integral form}\label{Generic Hall algebra and integral form}

In this subsection, we review the definitions of generic Ringel-Hall algebra and Bridgeland's Hall algebra defined in \cite{Ringel-1995} and \cite{Chen-Deng-2015} respectively. We also refer to \cite[Part I, II]{Murphy-2018} for details. 

For the Dynkin quiver $\mathbf{Q}=(I,H,s,t)$, let $\Phi^+$ be the set of positivity roots and $\Phi=\Phi^+\sqcup(-\Phi^+)$ be the set of roots associated to the corresponding semisimple Lie algebra $\fg$. Let $\Lambda$ be the set of functions $\Phi^+\rightarrow \mathbb{N}$. Then by Gabriel's theorem and the Krull-Schmidt property for the representations of $\mathbf{Q}$, for any finite field $\bbF_q$ and $k=\overline{\bbF_q}$, there are canonical bijections $\Lambda\cong \mathrm{Iso}(\mathcal{A}_q)\cong \mathrm{Iso}(\mathcal{A})$, where $\cA_q=\mathrm{rep}_{\bbF_q}(\mathbf{Q}), \cA=\mathrm{rep}_{k}(\mathbf{Q})$, see \cite[\S 7.2]{Murphy-2018}. We identify $\Lambda=\mathrm{Iso}(\mathcal{A})$, and denote $[M_q]\in \mathrm{Iso}(\mathcal{A}_q)$ the image of $[M]\in \Lambda$. The generic Ringel-Hall algebra $\cH_v^{\tw}$ is the $\mathbb{C}(v)$-vector space having a basis $\{u_{[M]}\mid [M]\in \Lambda\}$ together with the multiplication defined by 
$$u_{[M]}* u_{[N]}=v^{\langle \hat{M_q},\hat{N_q}\rangle}\sum_{[L]\in \Lambda}G^L_{M,N}(v)u_{[L]},$$
where $\langle \hat{M_q},\hat{N_q}\rangle$ is the Euler form which is independent of $\bbF_q$, and $G^L_{M,N}(v)\in \mathbb{Q}[v]$ is the Hall polynomial such that $G^L_{M,N}(v_q)=g^{L_q}_{M_qN_q}$, see \cite{Ringel-1995} or \cite[Proposition 7.3.1]{Murphy-2018}. The extension counting integral form $(\cH_v^{\tw})_{\mathrm{ex}}$ of $\cH_v^{\tw}$ is the $\cZ$-submodule of $\cH_v^{\tw}$ generated by $A_M(v)u_{[M]}$ for any $[M]\in \Lambda$, where $A_M(v)\in \mathbb{Q}[v]$ is the polynomial such that $A_M(v_q)=|\Aut_{\cA_q}(M_q)|$, see \cite[Proposition 7.3.1]{Murphy-2018}. By \cite[Definition 9.1.1]{Murphy-2018}, $(\cH_v^{\tw})_{\mathrm{ex}}$ is closed under multiplication, and so it is a $\cZ$-subalgebra. Indeed, we have  
$$(A_M(v)u_{[M]})* (A_N(v)u_{[N]})=v^{\langle \hat{M_q},\hat{N_q}\rangle}\sum_{[L]\in \Lambda}\frac{E^L_{M,N}(v)}{H_{M,N}(v)}(A_L(v)u_{[L]}),$$
where $E^L_{M,N}(v),H_{M,N}(v)\in \mathbb{Q}[v]$ satisfy $E^L_{M,N}(v_q)=|\Ext^1_{\cA_q}(M_q,N_q)_{L_q}|,H_{M,N}(v_q)=|\Hom_{\cA_q}(M_q,N_q)|$. Moreover, by \cite[Theorem 16.2.5]{Murphy-2018}, $(\cH_v^{\tw})_{\mathrm{ex}}$ is isomorphic to the Poisson integral form $_{\cZ}\bfU_v^{+,P}$ of $\bfU_v^+$ defined in \S \ref{Quantized enveloping algebra}. 

Let $\Lambda'$ be the set of functions $\Phi\sqcup(\bbZ_2\times I)\rightarrow \bbN$. By \cite[Lemma 4.2]{Bridgeland-2013}, for any finite field $\bbF_q$ and $k=\overline{\bbF_q}$, there are canonical bijections $\Lambda'\cong \mathrm{Iso}(\cC_2(\cP_q))\cong \mathrm{Iso}(\cC_2(\cP))$, see \cite[\S 12.1]{Murphy-2018}. In particular, let $\Lambda'_0$ be the subset of $\Lambda'$ consisting of functions supported on $\bbZ_2\times I$, then $\Lambda'_0$ corresponds to the subset of isomorphisms classes of contractible complexes. We identify $\Lambda'=\mathrm{Iso}(\cC_2(\cP))$, and denote $[M_{\bullet q}]\in \mathrm{Iso}(\cC_2(\cP_q))$ the image of $[M_{\bullet}]\in \Lambda'$. The generic Hall algebra $\cH^{\tw}_v(\Lambda')$ is the $\mathbb{C}(v)$-vector space having a basis $\{A_{M_\bullet}(v)u_{[M_\bullet]}\mid [M_\bullet]\in \Lambda'\}$ together with the multiplication defined by 
$$(A_{M_\bullet}(v)u_{[M_\bullet]})* (A_{N_\bullet}(v)u_{[N_\bullet]})=v^{\langle \hat{M_q^1},\hat{N_q^1}\rangle+\langle \hat{M_q^0},\hat{N_q^0}\rangle}\sum_{[L_\bullet]\in \Lambda'}\frac{E^{L\bullet}_{M_\bullet,N_\bullet}(v)}{H_{M_\bullet,N_\bullet}(v)}A_{L_\bullet}(v)u_{[L_\bullet]},$$
where $E^{L\bullet}_{M_\bullet,N_\bullet}(v), H_{M_\bullet,N_\bullet}(v)\in \mathbb{Q}[v]$ satisfy $E^{L\bullet}_{M_\bullet,N_\bullet}(v_q)=|\Ext^1_{\cC_2(\cP_q)}(M_{\bullet q},N_{\bullet q})_{L_{\bullet q}}|$ and $H_{M_\bullet,N_\bullet}(v_q)=|\Hom_{\cC_2(\cP_q)}(M_{\bullet q},N_{\bullet q})|$, see \cite{Chen-Deng-2015} or \cite[Proposition 12.2.1]{Murphy-2018}. Following \cite[Definition 13.1.3 \& 13.1.4]{Murphy-2018}, we define the generic Bridgeland's Hall algebra $\cD\cH_v^{\red}$ by 
\begin{align*}
\cD\cH_v&=\cH^{\tw}_v(\Lambda')[(A_{M_\bullet}u_{[M_\bullet]})^{-1}\mid [M_\bullet]\in \Lambda'_0],\\ 
\cD\cH_v^{\red}&=\cD\cH_v/\langle (A_{M_\bullet}u_{[M_\bullet]})*(A_{M_\bullet^*}u_{[M_\bullet^*]})-1\mid [M_\bullet]\in \Lambda'_0\rangle.
\end{align*}
The extension counting integral form $\cH^{\tw}_v(\Lambda')_{\ex}$ of $\cH^{\tw}_v(\Lambda')$ is the $\cZ$-submodule of $\cH^{\tw}_v(\Lambda')$ generated by $A_{M_\bullet}(v)u_{[M_\bullet]}$ for any $[M_\bullet]\in \Lambda'$. By \cite[Proposition 12.2.1]{Murphy-2018}, it is a $\cZ$-subalgebra. The extension counting integral form $(\cD\cH_v^{\red})_{\ex}$ of $\cD\cH_v^{\red}$ is the $\cZ$-submodule of $\cD\cH_v^{\red}$ generated by $A_{M_\bullet}(v)u_{[M_\bullet]}$ for any $[M_\bullet]\in \Lambda'$, see \cite[Definition 15.1.1]{Murphy-2018}. By \cite[Proposition 12.2.1]{Murphy-2018}, it is a $\cZ$-subalgebra. Moreover, $(\cD\cH_v^{\red})_{\ex}$ is isomorphic to the Poisson integral form $_{\cZ}\bfU_v^{P}$ of $\bfU_v$ defined in \S \ref{Quantized enveloping algebra}. 

\subsection{Relation between $\cD\cK^{ss,*,\red}$ and $(\cD\cH_v^{\red})_{\ex}$}

In this subsection, we study the relation between the $\cZ$-algebra $\cD\cK^{ss,*,\red}$ defined in \S \ref{Bridgeland's Hall algebra via perverse sheaves} and the extension counting integral form $(\cD\cH_v^{\red})_{\ex}$ of the generic Bridgeland's Hall algebra defined in \S \ref{Generic Hall algebra and integral form}.

Recall that there is an algebra isomorphism $\cH^{\tw}(\cC_2(\cP_q))\xrightarrow{\cong}  \bbC\otimes_{\cZ}\cK^{ss,*}$ for any finite field $\bbF_q$, see \eqref{three algebra isomorphism}. From these isomorphisms, we obtain a $\bbC(v)$-algebra isomorphism $\Psi_v:\cH^{\tw}_v(\Lambda')\xrightarrow{\cong} \bbC(v)\otimes_{\cZ}\cK^{ss,*}$. Moreover, by Lemma \ref{image of b}, if $\cO_{M_\bullet}\subset \rC_{\ue}$ is maximal with respect to the partial order $\preccurlyeq$, then $\Psi_v(A_{M_\bullet}(v)u_{[M_\bullet]})=v^{-\dim \cO_{M_\bullet}}I_{M_\bullet}^*$.

\begin{theorem}\label{two integral forms}
The $\bbC(v)$-algebra isomorphism $\Psi_v:\cH^{\tw}_v(\Lambda')\xrightarrow{\cong} \bbC(v)\otimes_{\cZ}\cK^{ss,*}$ can be restricted to a $\cZ$-algebra isomorphism $\Psi_v:\cH^{\tw}_v(\Lambda')_{\ex}\rightarrow \cK^{ss,*}$, and then it induces a $\cZ$-algebra isomorphism $\Phi_v:(\cD\cH_v^{\red})_{\ex}\rightarrow \cD\cK^{ss,*,\red}$.
\end{theorem}
\begin{proof}
We first prove that $\Psi_v(\cH^{\tw}_v(\Lambda')_{\ex})\subset \cK^{ss,*}$, that is $\Psi_v(A_{M_\bullet}(v)u_{[M_\bullet]})\in \cK^{ss,*}$ for any $[M_\bullet]\in \Lambda'$. By \cite[Lemma 4.2]{Bridgeland-2013}, $M_\bullet\cong C_A\oplus C_B^*\oplus K_P\oplus K_Q^*$ for some $A,B\in \cA$ and $P,Q\in \cP$, where $C_A,C_B^*$ are given by minimal projective resolutions of $A,B$ respectively. By \cite[Lemma 13.1.2]{Murphy-2018}, we have 
$$A_{M_\bullet}(v)u_{[M_\bullet]}=v^n(A_{C_A\oplus C_B^*}(v)u_{[C_A\oplus C_B^*]})*(A_{K_P}(v)u_{[K_P]})*(A_{K_Q^*}(v)u_{[K_Q^*]})$$ 
for some $n\in \bbZ$. Since $\cO_{K_P},\cO_{K_Q^*}$ are maximal, we have 
$$\Phi_v(A_{K_P}(v)u_{[K_P]})=v^{-\dim \cO_{K_P}}I_{K_P}^*,\ \Phi_v(A_{K_Q^*}(v)u_{[K_Q^*]})=v^{-\dim \cO_{K_Q^*}}I_{K_Q^*}^*\in \cK^{ss,*}.$$
We only need to prove that $\Phi_v(A_{C_A\oplus C_B^*}(v)u_{[C_A\oplus C_B^*]})\in \cK^{ss,*}$.
By the proof of \cite[Lemma 4.7]{Bridgeland-2013}, we have 
\begin{align*}
&(A_{C_A}(v)u_{[C_A]})*(A_{C_B^*}(v)u_{[C_B^*]})\\
=&v^{m}(A_{C_A\oplus C_B^*}(v)u_{[C_A\oplus C_B^*]})*(A_{K_{P'}}(v)u_{[K_{P'}]})*(A_{K_{Q'}^*}(v)u_{[K_{Q'}^*]})\\
+&\sum c_{A',B',P'',Q''}(v)(A_{C_{A'}\oplus C_{B'}^*}(v)u_{[C_{A'}\oplus C_{B'}^*]})*(A_{K_{P''}}(v)u_{[K_{P''}]})*(A_{K_{Q''}^*}(v)u_{[K_{Q''}^*]})
\end{align*}
where the sum is taken over $A',B'\in \cA,P'',Q''\in \cP$ such that $\hat{A'}+\hat{B'}<\hat{A}+\hat{B}$ and $<$ is the partial order on the Grothendieck group $K(\cA)$ defined by \cite[\S 1.3 (2)]{Bridgeland-2013}, and $c_{A',B',P'',Q''}(v)\in \cZ$. If $\hat{A}+\hat{B}=0$, then $A=B=0$, and we have $\Phi_v(A_{K_P\oplus K_Q^*}(v)u_{[K_P\oplus K_Q^*]})=v^{-\dim \cO_{K_P\oplus K_Q^*}}I_{K_P\oplus K_Q^*}^*\in \cK^{ss,*}$. Under the assumption that $\Phi_v(A_{C_A}(v)u_{[C_A]}),\Phi_v(A_{C_B^*}(v)u_{[C_B^*]})\in \cK^{ss,*}$, we can make an induction on $K(\cA)$ to prove $\Phi_v(A_{C_A\oplus C_B^*}(v)u_{[C_A\oplus C_B^*]})\in \cK^{ss,*}$.

Suppose that $\cO_{C_A}\subset \rC_{\ue}$. We make a descending induction on the set of $\rG_{\ue}$-orbits in $\rC_{\ue}$ to prove $\Phi_v(A_{C_A}(v)u_{[C_A]})\in \cK^{ss,*}$. If $\cO_{C_A}$ is maximal, then $\Psi_v(A_{C_A}(v)u_{[C_A]})=v^{-\dim \cO_{C_A}}I_{C_A}^*\in \cK^{ss,*}$. In general, since $\mathbf{Q}$ is a Dynkin quiver, $A\cong A_1\oplus\cdots\oplus A_n$, where $A_s$ is rigid and is the direct sum of some copies of an indecomposable object, such that,
$$\Hom_{\cA}(A_t,A_s)=\Ext^1_{\cA}(A_s,A_t)=0\ \textrm{for any} s<t.$$ 
Let $0\rightarrow P_s\rightarrow Q_s\rightarrow A_s\rightarrow 0$ be the minimal projective resolution of $A_s$, then the direct sum of these short exact sequence is the minimal projective resolution of $A$, that is, $C_A=C_{A_1}\oplus\cdots\oplus C_{A_n}$. Since $\Ext^1_{\cC_2(\cP)}(C_{A_s},C_{A_s})=\Ext^1_{\cA}(A_s,A_s)=0$, the orbit $\cO_{C_{A_s}}$ is maximal, we have $\Psi_v(A_{C_{A_s}}(v)u_{[C_{A_s}]})=v^{-\dim \cO_{C_{A_s}}}I_{C_{A_s}}^*\in \cK^{ss,*}$. By the proof of \ref{inductive lemma}, we have 
$$(A_{C_{A_1}}(v)u_{[C_{A_1}]})*\cdots*(A_{C_{A_n}}(v)u_{[C_{A_n}]})=v^l(A_{C_{A}}(v)u_{[C_{A}]})+\!\!\sum_{\cO_{A}\subset \overline{\cO_{A'}}} c_{A'}(v)(A_{C_{A'}}(v)u_{[C_{A'}]})$$
for some $l\in \bbZ$ and $c_{A'}(v)\in \cZ$. By inductive hypothesis, $\Psi_v((A_{C_{A'}}(v)u_{[C_{A'}]}))\in \cK^{ss,*}$, and so $\Phi_v(A_{C_A}(v)u_{[C_A]})\in \cK^{ss,*}$. Similarly, $\Phi_v(A_{C_B^*}(v)u_{[C_B^*]})\in \cK^{ss,*}$.

Hence $\Psi_v(\cH^{\tw}_v(\Lambda')_{\ex})\subset \cK^{ss,*}$. Note that both of them are free $\cZ$-modules with bases parametrized by $\Lambda'=\mathrm{Iso}(\cC_2(\cP))$. Hence $\Psi_v(\cH^{\tw}_v(\Lambda')_{\ex})=\cK^{ss,*}$. By the definitions of $\cD\cK^{ss,*,\red}$ and $(\cD\cH_v^{\red})_{\ex}$, $\Psi_v:(\cH^{\tw}_v(\Lambda')_{\ex})\rightarrow \cK^{ss,*}$, it is clear that $\Psi_v:\cH^{\tw}_v(\Lambda')_{\ex}\rightarrow \cK^{ss,*}$ induces a $\cZ$-algebra isomorphism $\Phi_v:(\cD\cH_v^{\red})_{\ex}\rightarrow \cD\cK^{ss,*,\red}$.
\end{proof}

Recall that in \S \ref{Lusztig's categorification}, for any $\nu\in \bbN[I]$, we fix a $I$-graded space $V$ to define $\rE_{\nu},\rG_{\nu}$. Now, we follow \cite{Lusztig-1998,Xiao-Xu-Zhao-2019}, and suppose that $V$ has a $\bbF_q$-rational structure with Frobenius map $\sigma:V\rightarrow V$. Since $\mathbf{Q}$ is a Dynkin quiver, by \cite[Theorem 2.8]{Schiffmann-2012} and \cite[Lemma 9.2.4]{Lusztig-1993}, Lusztig restriction functor $\Res^{\nu}_{\nu',\nu''}:\cD^b_{\rG_{\nu}}(\rE_{\nu})\rightarrow \cD^b_{\rG_{\nu'}\times \rG_{\nu''}}(\rE_{\nu'}\times \rE_{\nu''})$ for any $\nu,\nu',\nu''\in \bbN[I]$ such that $\nu=\nu'+\nu''$, can be restricted to 
\begin{align*}
\Res^{\nu}_{\nu',\nu''}:\cD^{b,ss}_{\rG_{\nu}}(\rE_{\nu})&\rightarrow \cD^{b,ss}_{\rG_{\nu'}}(\rE_{\nu'})\boxtimes \cD^{b,ss}_{\rG_{\nu''}}(\rE_{\nu''}),
\end{align*}
and we have the linear maps $\res^{\nu}_{\nu',\nu''}:\tilde{\cH}_{\rG_\nu^\sigma}(\rE_\nu^\sigma)\rightarrow \tilde{\cH}_{\rG_{\nu'}^\sigma}(\rE_{\nu'}^\sigma)\otimes \tilde{\cH}_{\rG_{\nu''}^\sigma}(\rE_{\nu''}^\sigma)$. We define 
$$\tilde{\Res}^{\nu}_{\nu',\nu''}=\Res^{\nu}_{\nu',\nu''}[2\sum_{i\in I}\nu'_i\nu''_i](\sum_{i\in I}\nu'_i\nu''_i),\ \tilde{\res}^{\nu}_{\nu',\nu''}=v_q^{-2\sum_{i\in I}\nu'_i\nu''_i}\res^{\nu}_{\nu',\nu''},$$ 
then all $\tilde{\Res}^{\nu}_{\nu',\nu''},\tilde{\res}^{\nu}_{\nu',\nu''}$ for $\nu=\nu'+\nu''$ induce multiplications on the graded duals 
\begin{align*}
&\cK^{+,*}=\bigoplus_{\nu}\cK^{+,*}_\nu=\bigoplus_{\nu}\Hom_{\cZ}(\cK^+_\nu,\cZ),\\
\tilde{\cH}^{*}(\cA_q)=&\bigoplus_\nu\tilde{\cH}^*_{\rG_\nu^\sigma}(\rE_\nu^\sigma)=\bigoplus_\nu\Hom_{\bbC}(\tilde{\cH}_{\rG_\nu^\sigma}(\rE_\nu^\sigma),\bbC)
\end{align*}
respectively. By \cite[Lemma 2]{Xiao-Xu-Zhao-2019} and Theorem \ref{sheaf-function correspondence}, there are algebra isomorphisms
\begin{equation}
\begin{aligned}\label{iso-1}
\cH^{\tw}(\cA_q)\xrightarrow{\cong} &\tilde{\cH}^{*}(\cA_q)\xrightarrow{\cong} \bbC\otimes_{\cZ}\cK^{+,*}\\
a_Mu_{[M]}\mapsto &\ \ 1_{\cO_M}^*\  \mapsto \ \chi^*(1_{\cO_M}^*),
\end{aligned}
\end{equation}
where $\{1_{\cO_M}^*\mid\cO_M\subset\rE^\sigma_\nu\}\subset \tilde{\cH}^*_{\rG_\nu^\sigma}(\rE_\nu^\sigma)$ is the dual basis $\{1_{\cO_M}\mid \cO_M\subset\rE^\sigma_\nu\}\subset \tilde{\cH}_{\rG_\nu^\sigma}(\rE_\nu^\sigma)$, which are analogue to the algebra isomorphisms in Theorem \ref{Bridgeland's algebra via perverse sheaf},
\begin{equation}
\begin{aligned}\label{iso-2}
\cD\cH^{\red}(\cA_q)\xrightarrow{\cong} &\cD\tilde{\cH}^{*,\red}(\cA_q)\xrightarrow\cong \bbC\otimes_{\cZ}\cD\cK^{ss,*,\red}\\
a_{M_\bullet}u_{[M_\bullet]}\mapsto &\ \ \ \ \ 1_{\cO_{M_\bullet}}^*\  \ \ \mapsto\ \ \  \ \chi^*(1_{\cO_{M_\bullet}}^*).
\end{aligned}
\end{equation}

By the same argument as $\Phi_v(A_{C_A}(v)u_{[C_A]})\in \cK^{ss,*}$ in the proof of Theorem \ref{two integral forms}, we have the following theorem.

\begin{theorem}\label{two integral forms'}
There is a $\cZ$-algebra isomorphism $\Psi_v^+:(\cH_v^{\tw})_{\ex}\rightarrow \cK^{+,*}$.
\end{theorem}

\subsection{Comparing the bases}

Recall that in \S \ref{Hall algebra for the category of two-periodic complexes}, for any $M\in \cA$, let $0\rightarrow P\xrightarrow{f} Q\xrightarrow{g} M\rightarrow 0$ be the minimal projective resolution of $M$, we define the element 
$$E_M=v_q^{\langle\hat{P},\hat{M}\rangle}b_{-\hat{P}}*(a_{C_M}u_{[C_M]})\in \cD\cH^{\red}(\cA_q),$$ 
where $C_M=(P,Q,f,0)\in \cC_2(\cP)$.

\begin{lemma}\label{inj homo}
There is an injective $\cZ$-algebra homomorphism $\Phi_v^+:\cK^{+,*}\rightarrow \cD\cK^{ss,*,\red}$.
\end{lemma}
\begin{proof}
By \cite[Lemma 4.3]{Bridgeland-2013}, there is an injective algebra homomorphism
\begin{align*}
\cH^{\tw}(\cA_q)\rightarrow \cD\cH^{\red}(\cA_q), a_Mu_{[M]}\mapsto E_M.
\end{align*}
Moreover, by Proposition \cite[Proposition 15.1.2]{Murphy-2018}, there is an injective $\cZ$-algebra homomorphism 
$$(\cH_v^{\tw})_{\ex}\rightarrow (\cD\cH_v^{\red})_{\ex},A_M(v)u_{[M]}\mapsto v^{\langle\hat{P},\hat{M}\rangle}(A_{K_P}(v)u_{[K_P]})^{-1}(A_{C_M}(v)u_{[C_M]}).$$
Combining with Theorem \ref{two integral forms} and \ref{two integral forms'}, we obtain the desired injective $\cZ$-algebra homomorphism $\Phi_v^+:\cK^{+,*}\rightarrow \cD\cK^{ss,*,\red}$.
\end{proof}

The canonical basis of $\cK^+_\nu$ is the $\cZ$-basis $\cB_\nu=\{I_M\mid \cO_M\subset \rE_\nu\}$, where $I_M\in \cK^+_\nu$ is the image of the intersection cohomology complex with Tate twist
$$\IC_M=\IC(\cO_M,\overline{\bbQ}_l)(\frac{\dim \cO_M}{2})$$
for any $\rG_\nu$-orbit $\cO_M\subset \rE_\nu$. Let $\cB^*_\nu=\{I^*_M\mid \cO_M\subset \rE_\nu\}\subset \cK^{+,*}_\nu$ be the dual basis of $\cB_{\nu}$, and 
$$\cB^*=\bigsqcup_{\nu}\cB_{\nu}^*.$$

\begin{proposition}\label{rigid lemma}
If $M\in \cA$ is rigid, that is, $\Ext^1_{\cA}(M,M)=0$, suppose the dimension vector of $M$ is $\nu$ and the projective dimension vector of $C_M$ is $\ue$, then the orbits $\cO_M\subset \rE_\nu, \cO_{C_M}\subset \rC_{\ue}$ are maximal with respect to the partial order $\preccurlyeq$. Moreover, we have
$$\Phi_v^+(v^{-\dim \cO_M}I_M^*)=v^{-\langle\hat{P},\hat{P}\rangle}\tilde{I}^*_{-\hat{P}}*_rI^*_{C_M}.$$
\end{proposition}
\begin{proof}
It is well-known that $\cO_M\subset \rE_\nu$ is a maximal orbit. Indeed, the codimension of $\cO_M\cong\rG_\nu/\Stab(\cO_M)\cong \rG_{\nu}/\Aut_{\cA}(M)$ in $\rE_\nu$ is 
$$\sum_{h\in H}\nu_{h'}\nu_{h''}-(\sum_{i\in I}\nu_i^2-\langle\hat{M},\hat{M}\rangle)=0.$$
By the short exact sequence $$0\rightarrow \Hom_{\cA}(Q,P)\rightarrow \End_{\cC_2(\cP)}(C_M)\rightarrow \End_{\cA}(M)\rightarrow 0,$$
we have
$$\dim \End_{\cC_2(\cP)}(C_M)=\langle\hat{Q},\hat{P}\rangle+\langle\hat{M},\hat{M}\rangle,$$
and then the dimension of $$\cO_{C_M}\cong \rG_{\ue}/\Stab(\cO_{C_M})\cong (\Aut_{\cA}(P)\times \Aut_{\cA}(Q))/\Aut_{\cC_2(\cP)}(C_M)$$ 
is equal to
$$\langle\hat{P},\hat{P}\rangle+\langle\hat{Q},\hat{Q}\rangle-(\langle\hat{Q},\hat{P}\rangle+\langle\hat{M},\hat{M}\rangle)=\langle \hat{P},\hat{Q}\rangle.$$
Thus $\cO_{C_M}\subset \Hom_{\cA}(P,Q)\times \{0\}$ is an open dense subset, and it is a maximal orbit in $\rC_{\ue}$.

By Lemma \ref{image of b}, under the algebra isomorphisms (\ref{iso-2}), the images of $b_{-\hat{P}}$ and $a_{C_M}u_{[C_M]}$ are $\tilde{I}^*_{-\hat{P}}$ and $v^{-\dim \cO_{C_M}}I^*_{C_M}$ respectively, and so the image of $E_M$ is 
$$v^{\langle\hat{P},\hat{M}\rangle-\langle \hat{P},\hat{Q}\rangle}\tilde{I}^*_{-\hat{P}}*_rI^*_{C_M}= v^{-\langle\hat{P},\hat{P}\rangle}\tilde{I}^*_{-\hat{P}}*_rI^*_{C_M}.$$
By similar argument as Lemma \ref{image of b}, under the algebra isomorphisms (\ref{iso-1}), the image of $a_Mu_{[M]}$ is $1\otimes v^{-\dim \cO_M}I_M^*$. Therefore, we have $$\Phi_v^+(v^{-\dim \cO_M}I_M^*)= v^{-\langle\hat{P},\hat{P}\rangle}\tilde{I}^*_{-\hat{P}}*_rI^*_{C_M},$$
as desired.
\end{proof}

\bibliography{mybibfile}

\end{document}